\documentclass[10pt]{amsart}
\usepackage{fancyvrb}
\usepackage[breakable]{tcolorbox}
\usepackage{amsfonts}
\usepackage{amsmath}
\usepackage{amssymb} 
\usepackage{amsthm}
\usepackage{mathtools}
\usepackage{tikz-cd}
\usetikzlibrary{calc}
\tikzset{curve/.style={settings={#1},to path={(\tikztostart)
    .. controls ($(\tikztostart)!\pv{pos}!(\tikztotarget)!\pv{height}!270:(\tikztotarget)$)
    and ($(\tikztostart)!1-\pv{pos}!(\tikztotarget)!\pv{height}!270:(\tikztotarget)$)
    .. (\tikztotarget)\tikztonodes}},
    settings/.code={\tikzset{quiver/.cd,#1}
        \def\pv##1{\pgfkeysvalueof{/tikz/quiver/##1}}},
    quiver/.cd,pos/.initial=0.35,height/.initial=0}

\usepackage[hidelinks]{hyperref}
\allowdisplaybreaks
\usepackage{geometry}
\makeatletter
\def\l@section{\@tocline{1}{0pt}{1.5em}{1.5em}{}}
\def\l@subsection{\@tocline{2}{0pt}{3.2em}{3.2em}{}}
\makeatother

\newcommand{\Tot}{\mathrm{Tot}}

\newcommand{\BZ}{\mathbb{Z}}
\newcommand{\BA}{\mathbb{A}}
\newcommand{\BQ}{\mathbb{Q}}

\newcommand{\BC}{\mathbb{C}}
\newcommand{\BK}{\mathbb{K}}
\newcommand{\BP}{\mathbb{P}}

\newcommand{\BH}{\mathbb{H}}
\newcommand{\BV}{\mathbb{V}}
\newcommand{\R}{\textnormal{R}}
\newcommand{\CO}{\mathcal{O}}
\newcommand{\CB}{\mathcal{B}}
\newcommand{\CH}{\mathcal{H}}

\newcommand{\CS}{\mathcal{S}}
\newcommand{\CQ}{\mathcal{Q}}
\newcommand{\CL}{\mathcal{L}}

\newcommand{\CF}{\mathcal{F}}

\newcommand{\CG}{\mathcal{G}}
\newcommand{\bt}{\mathbf{t}}

\newcommand{\BGm}{\textnormal{B}\mathbb{G}_m}

\newcommand{\Span}{\text{Span}} 
\newcommand{\Db}{\textrm{D}^{\textrm{b}}} 

\newcommand{\PP}{\mathbb{P}} 

\newcommand{\Perf}{\textnormal{Perf}}
\newcommand{\Quot}{\textnormal{Quot}}
\newcommand{\Coh}{\textnormal{Coh}}
\newcommand{\Hom}{\textnormal{Hom}}
\newcommand{\Ext}{\textnormal{Ext}}
\newcommand{\Rep}{\textnormal{Rep}} 
\newcommand{\BM}{\textnormal{BM}}
\newcommand{\RHom}{\textnormal{RHom}}
\newcommand{\RHHom}{\textnormal{R}\mathcal{H}\textnormal{om}}
\newcommand{\End}{\textnormal{End}}

\newcommand{\cone}{\textnormal{cone}}
\newcommand{\id}{\textnormal{id}}
\newcommand{\bd}{\textbf{d}} 
\newcommand{\op}{\textnormal{op}}
\newcommand{\pr}{\textnormal{pr}}

\newcommand{\tor}{\textnormal{tor}}
\newcommand{\full}{\textnormal{full}}

\newcommand{\Zero}{\textnormal{Zero}}

\newcommand{\vdim}{\textnormal{vdim}}
\newcommand{\rk}{\textnormal{rk}}

\newcommand{\pt}{\textnormal{pt}}
\newcommand{\vir}{\textnormal{vir}}

\newcommand{\Res}{\textnormal{Res}}
\renewcommand{\ss}{\textnormal{ss}}

\newcommand{\GL}{\textnormal{GL}}
\newcommand{\Sh}{\textnormal{Sh}}
\newcommand{\loc}{\textnormal{loc}}
\newcommand{\gr}{\textnormal{gr}}
\newcommand{\vac}{|0\rangle}
\newcommand{\Gr}{\mathrm{Gr}}

\DeclareMathOperator{\Sym}{Sym}

\newtheorem{theorem}{Theorem}[section]
\newtheorem {lemma}[theorem]{Lemma}

\newtheorem {corollary}[theorem]{Corollary}
\newtheorem {proposition}[theorem]{Proposition}
\theoremstyle{definition}
\newtheorem {example}[theorem]{Example} 
\theoremstyle {definition} 
\newtheorem{remark}[theorem]{Remark}
\newtheorem{notation}[theorem]{Notation}
\theoremstyle {definition} 
\newtheorem{definition}[theorem]{Definition}

\newtheorem{introtheorem}{Theorem}

\def\horizontaldistance{\kern2pt}
\def\verticaldistance{6pt}

\title{Cohomological Hall algebras and Quot schemes of curves}

\author[S. Jindal]{Shivang Jindal}
\address{\'Ecole Polytechnique F\'ed\'erale de Lausanne, Institute of Mathematics}
\email{shivang.jindal@epfl.ch}

\author[W. Lim]{Woonam Lim}
\address{Yonsei University, Department of Mathematics}
\email{woonamlim@yonsei.ac.kr}

\begin{document}
\begin{abstract}

We study cohomological Hall algebras of curves and their actions on the homology of Quot schemes. We introduce the virtual homology of Quot schemes and show that it is preserved by both creation and annihilation actions. We prove that the torsion CoHA is isomorphic to a shuffle algebra and, equivalently, to a braided symmetric algebra associated with a Yang–Baxter operator. We use this description to determine the ideal of tautological relations for punctual Quot schemes and obtain a new basis for their cohomology rings. Finally, we introduce a universal way to double the torsion CoHA and show that it acts naturally on the virtual homology of Quot schemes of arbitrary type.

\end{abstract}

\baselineskip=16pt
\maketitle
\vspace{-23pt}
\tableofcontents

\section{Introduction}

\subsection{Motivation}

The purpose of this paper is to study the cohomological Hall algebra of curves and their representations using the geometry of Quot schemes. The cohomological Hall algebras (CoHAs for short) were introduced by Kontsevich--Soibelman \cite{Kontsevich_Soibelman} for quivers with potential. Since then they have become a central topic in Donaldson--Thomas theory and geometric representation theory. A closely related two-dimensional Hall-algebra construction, now understood as the one-loop preprojective CoHA \cite{YZ}, appeared in the work of Schiffmann--Vasserot \cite{schiffmannvasserotcherednik} on the AGT conjecture. An important feature of CoHAs is that they act naturally on the cohomology of framed moduli spaces \cite{schiffmannvasserotcherednik, Soibelman,franzen, Davison_Meinhardt,Minets,young,YZ, giacomini}. In this paper, we develop an analogous theory for coherent sheaves on a curve, with Quot schemes playing the role of framed moduli spaces.

Over the past decade, there have been many foundational developments in the theory of CoHAs such as cohomological integrality theorems \cite{efimov,Davison_Meinhardt,Henncartint, BDNKP,hennecart2025bpsdecompositiontheorem}, the construction of coproducts and vertex-coalgebra structures \cite{davison2016criticalcohaquiverpotential,jindal2026criticalcohasvertexcoalgebras,  DHKSV}, and connections with Maulik--Okounkov Yangians \cite{MaulikOkounkov, BD, SV_Yangian}. CoHAs have also played a role in the cohomological $\chi$-independence conjecture \cite{DHKSV}, the $P=W$ conjecture \cite{hausel2025pwmathcalh2}, Langlands duality for 3-manifolds \cite{Kinjo-Park-Safronov} and recent approaches to Dolbeault geometric Langlands \cite{pădurariu2026dolbeaultgeometriclanglandsconjecture}. 

Despite these advances, there have been only a few explicit descriptions of CoHAs arising from geometry. One exception is the description of the zero dimensional CoHA of surfaces in terms of the positive half of the deformed $W$-algebra \cite{N1,MMSV}; see also \cite{arbesfeldschiffmann, Sala-Schiffmann,DPS,  Jindal, davisonaffinebps, DPSSV, sala2025kleinianorbifoldscohomologicalhall}. In this paper, we compute the torsion CoHA of a curve as a shuffle algebra with the kernel determined by the diagonal class. Using the shuffle algebra description, we give a presentation of the torsion CoHA by spherical generators and relations.

On the other hand, Quot schemes have played a fundamental role in enumerative geometry. The virtual intersection theory of Quot schemes of curves has applications to moduli spaces of semistable bundles \cite{MO1, M1, M2}, as well as to the quantum cohomology and quantum $K$-theory of Grassmannians \cite{Bertram, BDW, MO2, SZ1, SZ2}. More recently, the virtual intersection theory of torsion Quot schemes on surfaces has been studied, with particular emphasis on the rationality of generating series of virtual invariants \cite{OP, JOP, AJLOP, Anderson}.

One of our motivations is to investigate the representation-theoretic structures underlying the rationality of generating series associated with Quot schemes, in analogy with the role of the Heisenberg algebra \cite{Nakajima, Grojnowski} in the generating series of Betti numbers of Hilbert schemes of points on surfaces \cite{Gottsche}. Recently, Marian--Negu\c t \cite{marian2026cohomologyquotschemesmooth} constructed an action of the shifted Yangian on the cohomology of punctual Quot schemes which induces a basis explaining the rationality of the generating series of Betti numbers. Kaushik \cite{kaushik} subsequently generalized this picture to hyperquot schemes using the higher-rank shifted Yangians. We expect that our construction of the CoHA action on the virtual homology of Quot schemes similarly underlies the rationality of generating series of certain virtual invariants, which will be pursued elsewhere.

Another motivation of this paper is to propose a natural definition for the doubled torsion CoHA of a curve. CoHAs often behave as positive halves of some quasi-triangular algebras, but constructions of their doubles are known only in a limited number of cases, including the spherical part of the preprojective CoHA \cite{YZ}, the CoHA of tripled quiver with canonical cubic potential via Maulik-Okounkov Yangians \cite{MaulikOkounkov, BD, SV_Yangian}, and the zero dimensional CoHA of surfaces \cite{MMSV}. In this paper, we define the doubled torsion CoHA of curves based on our computation of the commutators between creation, annihilation, and tautological class multiplication operators. Surprisingly, these commutators are shown to be independent of the type of the Quot schemes under consideration, hence suggesting a natural way to double the torsion CoHA.

\subsection{Main results}

We summarize the main results of the paper. Let $C$ be a smooth projective curve over the complex numbers. Denote by $\Coh_{(r,d)}(C)$ the stack of all coherent sheaves on $C$ of rank $r$ and degree $d$. Using these stacks, one can define the cohomological Hall algebra
$$\BH^\full:=\bigoplus_{(r,d)}H^\BM_*(\Coh_{(r,d)}(C)),
$$
where the multiplication is given by the short exact sequence correspondence. In order to construct a representation of the CoHA, we fix a coherent sheaf $V$ on $C$ that plays the role of a framing dimension vector for CoHA of quivers. Let $\Quot_{(r,d)}(V)$ be the Quot scheme parametrizing rank $r$ and degree $d$ quotients of $V$ and consider 
$$\BV^\full:=\bigoplus_{(r,d)}H^\BM_*(\Quot_{(r,d)}(V)). 
$$
We also consider the virtual homology which is a subspace $$\BV^\vir\subseteq \BV^\full$$ 
spanned by the virtual fundamental classes capped with tautological classes, see Definition \ref{def: virtual homology}. The virtual homology should be thought of as the cohomological counterpart of looking at the integral of tautological classes against the virtual fundamental classes.

\begin{introtheorem}[Prop. \ref{prop: CoHA-module}, Def. \ref{def: annihilation}, Thm. \ref{thm: virtual homology is preserved}, Prop. \ref{prop: cyclic by vacuum}]
    The nested Quot scheme correspondence induces the creation and annihilation actions
    $$\BH^\full \curvearrowright \BV^\full,\quad (\BH^\full)^\op \curvearrowright \BV^\full,
    $$
    such that $\BV^\vir\subseteq \BV^\full$ is preserved. Furthermore, $\BV^\vir$ is a cyclic $\BH^\full$-module generated by the vacuum vector $|0\rangle\in H^\BM_0(\Quot_{(0,0)}(V))$. 

\end{introtheorem}

\begin{remark}
    Our construction of the annihilation action crucially relies on the projectivity assumption. Nevertheless, the torsion CoHA and its action on the Borel--Moore homology of punctual Quot schemes still make sense even for quasi-projective curves. When $C=\BA^1$, this recovers the framed CoHA-module of the Jordan quiver. 
\end{remark}

The next result computes the torsion CoHA which is the subalgebra of $\BH^\full$ given by 
$$\BH:=\bigoplus_{d\geq 0} H^\BM_*(\Coh_{(0,d)}(C)). 
$$
We define the shuffle algebra, also considered by Negu\c t in \cite[Sec. 4.7]{N2}, whose underlying vector space is given by 
$$\Sh^\Delta_{*}(H^*(C)[z]):=\bigoplus_{d\geq 0}\Sym_d(H^*(C)[z])=\bigoplus_{d\geq 0} \left(
(H^*(C)^{\otimes d})[z_1,\dots, z_d]
\right)^{S_d}. 
$$
The multiplication of $f\in \Sym_d(H^*(C)[z])$ and $g\in \Sym_e(H^*(C)[z])$ is defined as 
$$f\star g:=\sum_{\sigma\in \Sh_{d,e}}\sigma\left(
f(z_1,\dots, z_d)\cdot g(z_{d+1},\dots, z_{d+e})\cdot \hspace{-10pt}\prod_{\substack{1\leq i\leq d\\d+1\leq j\leq d+e}}
\hspace{-5pt}\left(1+\frac{\Delta_{ij}}{z_j-z_i}\right)
\right),
$$
where $\Sh_{d,e}$ consists of $\sigma\in S_{d+e}$ such that $\sigma(1)<\cdots<\sigma(d)$ and $\sigma(d+1)<\cdots<\sigma(d+e)$, and $\Delta_{ij}\in H^*(C\times \cdots \times C)$ denotes the diagonal class of $(i,j)$ factors. The shuffle algebra with the trivial kernel, denoted by $\Sh_{*}(H^*(C)[z])$, refers to the same vector space with the multiplication given by 
$$f\star g:=\sum_{\sigma\in \Sh_{d,e}}\sigma\left(
f(z_1,\dots, z_d)\cdot g(z_{d+1},\dots, z_{d+e})
\right).
$$
In the case of smooth projective curves, the next theorem was also proven in the unpublished note of Schiffmann--Vasserot \cite{SV_unpublished}.

\begin{introtheorem}[Thm. \ref{thm: CoHA equals shuffle}]\label{introthm: CoHA equals shuffle}
    For any smooth quasi-projective curve $C$, there is an isomorphism
    $$\BH\simeq \Sh^\Delta_{*}(H^*(C)[z])$$
    of doubly graded algebras.
\end{introtheorem}

\begin{remark}
    If $C=\BA^1$, the torsion CoHA is equal to the CoHA of the Jordan quiver. Since $H^*(\BA^1)=\BQ$ and $\Delta_{\BA^1}=0$, Theorem \ref{introthm: CoHA equals shuffle} computes the CoHA of the Jordan quiver as the shuffle algebra $\Sh_{*}(\BQ[z])$ as done in Kontsevich--Soibelman \cite[Sec. 2.5]{Kontsevich_Soibelman}.  If $C=\BP^1$, then we explain in Section \ref{sec: Kronecker} that the torsion CoHA is equal to a certain semistable CoHA of the Kronecker quiver which is also computed by Franzen--Reineke \cite{franzen2019cohomological}. 
\end{remark}

\begin{remark}
The case of $C=\mathbb{G}_m$ is particularly interesting. For any $d \geq 0$, the following stacks are all naturally isomorphic to each other:
\begin{enumerate}
    \item[i)] the stack $\Coh_{(0,d)}(\mathbb{G}_m)$ of torsion sheaves of length $d$ on $\mathbb{G}_m$,
    \item[ii)] the stack $\mathrm{Loc}_d(S^1)$ of rank $d$ local systems on the circle $S^1$,
    \item[iii)] the stack $\mathrm{Rep}_d^{\textnormal{mult}}(Q^{\textnormal{Jor}})$ of $d$ dimensional multiplicative representations of the Jordan quiver.
\end{enumerate}
The associated CoHA structures are all identified and Theorem \ref{introthm: CoHA equals shuffle} shows that they are isomorphic to the (Koszul signed) shuffle algebra $\Sh_{*}((\BQ\oplus\BQ[-1])[z])$, since $H^*(\mathbb{G}_m)=\BQ\oplus \BQ[-1]$ and $\Delta_{\mathbb{G}_m}=0$.
\end{remark}

Using the shuffle algebra descriptions, we obtain several structural results for the torsion CoHA. The first result concerns a presentation of $\BH$ in terms of spherical elements, i.e., 
$$e^\alpha_i:=\alpha \cdot z^i\in \BH_1=H^*(\Coh_{(0,1)}(C))=H^*(C\times \BGm)=H^*(C)[z]. 
$$
The next two results control the noncommutative nature of the torsion CoHA in two different ways. One expresses $\BH$ as a deformation of a supercommutative algebra and the other expresses it as a symmetric algebra in a certain braided category with respect to a Yang--Baxter operator.

\begin{introtheorem}[Thm. \ref{relations:torsioncoha},  Prop. \ref{prop: associatedgraded}, Prop. \ref{prop:twist}, Cor. \ref{cor: YB-twisted symmetric}]
For any smooth quasi-projective curve $C$, the following holds. 
\begin{enumerate}
    \item [1)] $\BH$ admits a presentation by generators $\{e^\alpha_i\,|\,\alpha\in H^*(C),\ i\geq 0\}$ and relations
$$e^{\alpha+\beta}_i=e^\alpha_i+e^\beta_i,\quad [e^{\alpha}_i, e^{\beta}_j ] =  
\begin{cases}
-\sum_{k=1}^{i-j} e_{i-k} \star^{\alpha \cup \beta} e_{j+k-1}, & \textnormal{if}\quad  i\geq j, \\[\verticaldistance] \sum_{k=1}^{j-i} e_{j-k} \star^{\alpha \cup \beta} e_{i+k-1} , & \textnormal{if}\quad  i <j, 
\end{cases}
$$
for each $\alpha, \beta\in H^*(C)$ and $i, j\geq 0$.
    \item [2)] There exists a multiplicative filtration $F_\bullet \BH$ such that  $\Gr^F_\bullet \BH\simeq \Sh_{*}(H^*(C)[z])$ as algebras. 
    \item [3)] There exists a Yang--Baxter operator $R$ on $H^*(C)[z]$ such that $\BH\simeq \Sym^R(H^*(C)[z])$. 
\end{enumerate}
\end{introtheorem}

\begin{remark}
In the first part of the above theorem, we used a notation 
$$e_i\star^\gamma e_j:=\sum e_i^{\gamma_{(1)}}\star e_j^{\gamma_{(2)}}\in \BH_2
$$
where $\Delta_*\gamma=\sum \gamma_{(1)}\otimes \gamma_{(2)}\in H^*(C\times C)$ using Sweedler's notation. In the second part, the filtration $F_\bullet \BH$ corresponds to the $z$-degree filtration upon the isomorphism $\BH\simeq \Sh^\Delta_{*}(H^*(C)[z])$. In the third part, the Yang--Baxter operator $R$ is a $H^*(C)$-decorated version of the divided difference operator, see Definition \ref{def: curve YB operator} for a precise formula. When $C=\BP^1$, our construction of the Yang--Baxter operator solves a conjecture of Franzen--Reineke \cite[Conj. 5]{franzen2019cohomological}.
\end{remark}

Let $V$ be a vector bundle on $C$ and $V^{[d]}$ be the tautological bundle on $\Coh_{(0,d)}(C)$ such that\footnote{In the literature, the notation $(V^\vee)^{[d]} $ is more commonly used for this tautological bundle.}
$$V^{[d]}\big|_F=\Hom(V,F),\quad F\in \Coh_{(0,d)}(C).
$$
The next result characterizes the Borel--Moore homology groups of punctual Quot schemes, i.e., 
$$ \BV:=\bigoplus_{d\geq 0}H^\BM_*(\Quot_{(0,d)}(V))\subseteq \BV^\full, 
$$
purely in terms of the structure of the torsion CoHA and tautological bundles. This would then induce an explicit basis for the cohomology of punctual Quot schemes. We fix an ordered basis $\alpha_1 < \alpha_2 <  \cdots< \alpha_m $ of pure codimension classes in $H^{*}(C)$.

\begin{introtheorem}[Thm. \ref{thm:quotschemekernel}, Thm. \ref{thm: Nakajima type basis}]
    Let $C$ be a smooth quasi-projective curve and $V$ be a vector bundle on it. Then $\BV$ is a cyclic $\BH$-module generated by the vacuum vector $|0\rangle$ such that
    $$\ker(\BH\twoheadrightarrow\BV)=\sum_{d_1\geq 0,\ d_2>0}\BH_{d_1}\star(e(V^{[d_2]})\cdot \BH_{d_2}). 
    $$
    Furthermore, $\BV$ has a basis given by monomials $e^{\alpha_{k_1}}_{i_1} \star e^{\alpha_{k_2}}_{i_2} \star \cdots \star e^{\alpha_{k_d}}_{i_d} \vac$ satisfying
\[ d\geq 0,\quad 0\leq i_1 \leq i_2\leq \cdots \leq i_d < \rk(V)\] and, whenever $i_j=i_{j+1}$ then one has $\alpha_{k_j} \leq \alpha_{k_{j+1}}$ with a strict inequality when both classes are odd.
\end{introtheorem}

\begin{remark}
    It appears that an explicit basis for the cohomology of punctual Quot schemes on {\it quasi-projective} curves is not known in general. For projective cases, an explicit Nakajima type basis has been obtained recently by Marian--Negu\c t \cite{marian2026cohomologyquotschemesmooth}. The basis from loc. cit. involves not only the creation operators but also multiplication operators by tautological classes, hence is different from the basis in the above theorem even in the projective case. In an another paper, Marian–-Negu\c t \cite[Thm. 2]{negutmarianderived} constructs a different basis in $K$-theory, which we expect to be related to ours under the passage from $K$-theory to cohomology. See Corollary \ref{cor: Poincare} for the discussion about its implication on the Poincar\'e polynomial of punctual Quot schemes. 
\end{remark}

\begin{remark}
    We remark that the above theorem can also be used to study a ring structure of the cohomology of punctual Quot schemes. Indeed, for each $d\geq 0$, the surjection $\BH\twoheadrightarrow\BV$ induces a tautological relation ideal exact sequence
    $$0\rightarrow I_d\rightarrow H^*(\Coh_{(0,d)}(C))\rightarrow H^*(\Quot_{(0,d)}(V))\rightarrow 0
    $$
    describing the cohomology ring of $\Quot_{(0,d)}(V)$. When $d=1$, this recovers the usual projective bundle formula for the cohomology ring of $\Quot_{(0,1)}(V)\simeq \BP(V)$, see Example \ref{ex: d=1 ring}. The case of $d=2$ is also worked out in Example \ref{ex: d=2 ring}. We note that Gautam's thesis \cite[Thm. 3.0.2]{Gautam} gives a very different type of a ring presentation for the cohomology of punctual Quot schemes on $\BP^1$, together with a conjectural formula for arbitrary smooth projective curves.
\end{remark}

Assume that $C$ is a smooth projective curve. In Section \ref{sec: Drinfeld double}, we define a double of the torsion CoHA as a completed tensor product
$$D(\BH):=\BH^{\textnormal{op}}\ \widehat \otimes\ \BH^{0-}\ \widehat \otimes\ \BH^{0+}\ \widehat \otimes\ \BH,
$$
where all four tensor factors are subalgebras. Here $\BH^{0\pm}$ is a  completion of a free supercommutative algebra related to tautological classes. The definition of the algebra structure on $D(\BH)$ is motivated by the results in Section \ref{sec: Commutator relations}, which compute the commutators among the creation and annihilation operators and multiplication by tautological classes acting on the virtual homology $\BV^\vir$.

\begin{introtheorem}[Sec. \ref{sec: Commutator relations}, Prop. \ref{prop: doubled algebra action}]
    There exists an action $D(\BH) \curvearrowright \BV^\vir$ which extends the action of $\BH$, $\BH^\op$ and $\BH^{0\pm}$ via creation and annihilation operators and multiplication by tautological classes, respectively. 
\end{introtheorem}

\begin{remark}
    Note that our doubled algebra $D(\BH)$ is universal in the sense that it works for any choice of a coherent sheaf $V$. When $V$ is a fixed rank $N$ vector bundle, the restriction of our action to the Borel--Moore homology of punctual Quot schemes $D(\BH)\curvearrowright \BV$ factors through a smaller algebra $D(\BH)_{V,0}$ which is essentially the same as the action of the $N$-shifted Yangian constructed by Marian--Negu\c t \cite{marian2026cohomologyquotschemesmooth}, see Remark \ref{rem: relation to MN}. In \cite{cao2026shiftedquantumgroupscritical} authors conjecturally construct a shifted double of the cohomological Hall algebra of a quiver with potential, which acts on the cohomology of framed representations. Our double $D(\mathbb{H})_{V,0}$ can be seen as an analog of the double for compact geometries.
\end{remark}

\subsection{Notation and convention}
We work over $\BC$. Cohomology and (Borel--Moore) homology are always taken with $\BQ$-coefficients. The notions of commutators, shuffle algebras, opposite algebras and symmetric algebras are all with respect to the super-grading induced from cohomological grading. Every (nested) Quot scheme is understood to be derived (nested) Quot schemes in the sense of \cite{Adhikari, Monavari-Pavia-Ricolfi}. We note that cohomology or Borel--Moore homology of derived stacks depends only on their classical truncation. For a quasi-smooth (resp. quasi-smooth and proper) morphism $f:X\rightarrow Y$, the virtual pullback (resp. virtual Umkehr map) is denoted by 
$$f^!:H^\BM_*(Y)\rightarrow H^\BM_{*+2\dim(f)}(X)\quad\quad \textnormal{(resp. }f_!:H^*(X)\rightarrow H^{*-2\dim(f)}(Y)\textnormal{)}.
$$
See Appendix for basic functorial properties of the virtual pullbacks and virtual Umkehr maps.

\medskip
\noindent \textbf{Acknowledgments.}
We thank Y. Bae, B. Davison, N. Giacomini, S. Kaubrys, A. Kaushik, M. Kool, A. Marian, M. Moreira, A. Negu\c t, D. Oprea, H. Park, T. Peerenboom, W. Pi, and O. Schiffmann for useful discussions. Especially, we would like to thank O. Schiffmann for sharing the unpublished note \cite{SV_unpublished} and explaining many new ideas. WL is supported by the Yonsei University Research Fund of 2024-22-0502, the POSCO Science Fellowship of POSCO TJ Park Foundation, and NRF
grant funded by the Korean government (MSIT) (RS-2025-00514643). The mathematical content of this paper is the product of human work. GPT-5.6 in ChatGPT was used for textual editing and finding an old reference \cite{LLT} used in Example \ref{ex: YB classical}. The authors take full responsibility for the content of the paper.

\section{Cohomological Hall algebra and module}

\subsection{Cohomological Hall algebra of curves}

Let $C$ be a smooth connected projective curve over the complex numbers. We denote a topological type of a coherent sheaf as a pair $(r,d)$ of its rank and degree. The set of all topological types forms a monoid $P=(\BZ_{>0}\times \BZ)\sqcup(\{0\}\times\BZ_{\geq 0})$ under pairwise addition. The stack of coherent sheaves on $C$, denoted by $\Coh$, decomposes into connected components according to their topological types\footnote{For simplicity, we drop $C$ from the notation whenever it is clear from the context.}
$$\Coh=\bigsqcup_{\alpha\in P}\Coh_{\alpha}.
$$

The notion of a short exact sequence is a distinctive feature of abelian categories. A moduli theoretic version of short exact sequences induces a diagram
\begin{equation}\label{eq: CoHA diagram}
\begin{tikzcd}
  & \Coh_{\alpha,\beta} \arrow[ld, "q"'] \arrow[rd, "p"] &   \\
\Coh_\alpha\times \Coh_\beta &                                    & \Coh_{\alpha+\beta},
\end{tikzcd}
\end{equation}
which records the data of short exact sequences and two projections
\begin{center}
\begin{tikzcd}
  & (0\rightarrow F_\alpha\rightarrow F_{\alpha+\beta}\rightarrow F_\beta\rightarrow 0) \arrow[ld, mapsto,"q"'] \arrow[rd, mapsto,"p"] &   \\
(F_\alpha,F_\beta) &                                    & F_{\alpha+\beta}.
\end{tikzcd}
\end{center}
This diagram induces a structure of a cohomological Hall algebra, CoHA for short, on the direct sum of Borel--Moore homology of the stacks (with rational coefficients)
$$\BH^\full:=\bigoplus_{\alpha\in P}H^\BM_*(\Coh_\alpha)$$ 
defined by the composition 
$$\star:=p_*\circ q^!:H^\BM_a(\Coh_\alpha)\otimes H^\BM_b(\Coh_\beta)\xrightarrow{q^!}
H^\BM_{a+b-2\chi(\beta,\alpha)}(\Coh_{\alpha,\beta})\xrightarrow{p_*} H^\BM_{a+b-2\chi(\beta,\alpha)}(\Coh_{\alpha+\beta}). 
$$
This definition utilizes the fact that $q$ is smooth and $p$ is proper. It is well-known by Sala--Schiffmann \cite[Sec. 3.2]{Sala-Schiffmann} that $(\BH^\full,\star)$ forms an associative algebra with a unit 
$$1\in \BQ=H^\BM_0(\Coh_{(0,0)})$$ 
where $\Coh_{(0,0)}$ is simply a point parametrizing the zero sheaf. We remark that $\star$ preserves the $P$-grading but not Borel--Moore homological grading. Restricting to the submonoid $P^\tor:=\{0\}\times \BZ_{\geq 0}$ of topological types of torsion sheaves, we get a subalgebra called the torsion CoHA
$$\Big(\BH:=\bigoplus_{d=0}^\infty H^{\BM}_*(\Coh_{(0,d)})\,\,\star\Big)\subseteq \Big(\BH^{\full}\,,\,\star\Big).
$$
We note that the Borel--Moore homological grading is preserved for the torsion CoHA and that the definition of the torsion CoHA makes sense even when the curve $C$ is quasi-projective. 

\subsection{Quot schemes and CoHA-modules}

Algebras are often best understood by their representations. A main purpose of this paper is to study CoHA-modules that are geometrically constructed using Quot schemes. Unlike the short exact sequence correspondence diagram \eqref{eq: CoHA diagram} which involves only smooth stacks, the nested Quot scheme correspondence diagram \eqref{eq: CoHA-module diagram} below will involve singular schemes in general. For this reason, we use derived Quot schemes by Adhikari \cite{Adhikari} and derived nested Quot schemes by Monavari--Pavia--Ricolfi \cite{Monavari-Pavia-Ricolfi} in order to use hidden smoothness behind the nested Quot scheme correspondence diagram. The use of derived algebraic geometry in the context of cohomological Hall algebras has been ubiquitous; (categorified) CoHA of surfaces \cite{Kapranov-Vasserot, Sala-Schiffmann, Porta-Sala} and CoHA of 3-Calabi--Yau categories \cite{Kinjo-Park-Safronov}. From now on, we refer to the derived (nested) Quot schemes simply by (nested) Quot schemes.

Let $V$ be a fixed coherent sheaf on $C$ which plays the role of a framing vector in the case of quiver varieties. The Quot scheme of $V$ decomposes into  
$$\Quot(V)=\bigsqcup_{\alpha\in P}\Quot_{\alpha}(V)
$$
according to the topological type of the quotient sheaf. We often omit $V$ from the notation when it is clear from the context. We consider the nested Quot scheme correspondence diagram 
\begin{equation}\label{eq: CoHA-module diagram}
\begin{tikzcd}
  & \Quot_{\alpha,\beta} \arrow[ld, "f"'] \arrow[rd, "g"] &   \\
\Coh_\alpha\times \Quot_\beta &                                    & \Quot_{\alpha+\beta},
\end{tikzcd}
\end{equation}
where $\Quot_{\alpha,\beta}$ is the nested Quot scheme parametrizing nested quotients\footnote{Subscripts $\alpha+\beta$ and $\alpha$ are used to emphasize the topological type of underlying sheaves. }
$$V\twoheadrightarrow Q'_{\alpha+\beta}\twoheadrightarrow Q_\beta. 
$$
It is useful to consider the following diagram induced from the nested quotients
\begin{equation}\label{eq: nested Quot}
\begin{tikzcd}
            &                                  &                       & F_\alpha \arrow[d, hook]     &   \\
0 \arrow[r] & S' \arrow[r] \arrow[d]            & V \arrow[r] \arrow[d, equal] & Q'_{\alpha+\beta} \arrow[r] \arrow[d] & 0 \\
0 \arrow[r] & S \arrow[d, two heads] \arrow[r] & V \arrow[r]           & Q_\beta \arrow[r]           & 0. \\
            & F_\alpha                                &                       &                       &  
\end{tikzcd}
\end{equation}
The maps $f$ and $g$ send this data to 
$$\Big(F_\alpha\,,\,[\,0\rightarrow S\rightarrow V\rightarrow Q_\beta\rightarrow 0\,]\Big)
\quad\textnormal{and}\quad 
[\,0\rightarrow S'\rightarrow V\rightarrow Q'_{\alpha+\beta}\rightarrow 0\,],$$
respectively. 

\begin{lemma}\label{lem: quasi-smoothness}
    In the diagram \eqref{eq: CoHA-module diagram}, $f$ is quasi-smooth and $g$ is quasi-smooth and proper. 
\end{lemma}
\begin{proof}
    In the proof, we use the notation as in \eqref{eq: nested Quot}. Consider the diagram 
\begin{equation}\label{eq: morphism g}
    \begin{tikzcd}
\Quot_{\alpha,\beta} \arrow[r, "\tilde{\phi}"] \arrow[d, "g"] & \Coh_{\alpha,\beta} \arrow[d, "\tilde{g}"] \\
\Quot_{\alpha+\beta} \arrow[r, "\phi"]                 & \Coh_{\alpha+\beta}       
\end{tikzcd}
\end{equation}
    where $\phi$ maps a quotient $V\twoheadrightarrow Q'_{\alpha+\beta}$ to $Q'_{\alpha+\beta}$ and $\tilde{\phi}$ maps a nested quotient $V\twoheadrightarrow Q'_{\alpha+\beta}\twoheadrightarrow Q_\beta$ to a short exact sequence
    $$0\rightarrow F_\alpha\rightarrow Q'_{\alpha+\beta}\rightarrow Q_\beta\rightarrow 0.
    $$
    Then the diagram \eqref{eq: morphism g} is clearly Cartesian. Since $\tilde{g}$ is quasi-smooth and proper, the same is true for $g$ after the base change. 

    Write $[V]:=(\rk(V),\deg(V))$. For the other case, we consider the diagram
\begin{equation}\label{eq: morphism f}
    \begin{tikzcd}
\Quot_{\alpha,\beta} \arrow[r,"\tilde{\rho}"] \arrow[d, "f"] & \Coh_{[V]-(\alpha+\beta), \alpha} \arrow[d, "\tilde{f}"] \\
\Coh_\alpha\times \Quot_{\beta} \arrow[r, "\rho"]                 & \Coh_{[V]-\beta} \times \Coh_\alpha    
\end{tikzcd}
\end{equation}
where $\tilde{f}$ is the natural forgetful morphism, $\rho$ sends $(F_\alpha, V\twoheadrightarrow Q_\beta)$ to $(S,F_\alpha)$ and $\tilde{\rho}$ sends a nested quotient $V\twoheadrightarrow Q'_{\alpha+\beta}\twoheadrightarrow Q_\beta$ to a short exact sequence 
$$0\rightarrow S'\rightarrow S\rightarrow F_\alpha\rightarrow 0.$$
The diagram \eqref{eq: morphism f} is also clearly Cartesian. Since $\tilde f$ is a morphism between smooth stacks, it is quasi-smooth. Therefore the same is true for $f$ after the base change. 

\end{proof}

Consider the direct sum of Borel--Moore homology of the Quot schemes 
$$\BV^\full:=\bigoplus_{\alpha\in P}H^\BM_*(\Quot_\alpha)
$$
and composition of the morphisms
$$\star:=g_*\circ f^!:H^\BM_a(\Coh_\alpha)\otimes H^\BM_b(\Quot_\beta)\xrightarrow{f^!}H^\BM_{a+b+2\chi(V-\beta,\alpha)}(\Quot_{\alpha,\beta})\xrightarrow{g_*} H^\BM_{a+b+2\chi(V-\beta,\alpha)}(\Quot_{\alpha+\beta}). 
$$
This definition utilizes the fact that $f$ is quasi-smooth and $g$ is proper. 

\begin{proposition}\label{prop: CoHA-module}
    The above construction yields a CoHA-module $\star:\BH^\full\otimes \BV^\full\rightarrow \BV^\full$. 
\end{proposition}
\begin{proof}  The $\star$-product is associative if the outer arrows of the diagram 
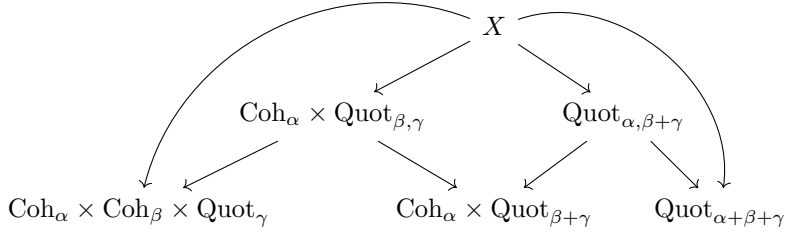
\begin{center}
\begin{tikzcd}[column sep=-20pt]
  &                         &  X\arrow[ld] \arrow[rd] \arrow[lldd, bend right=49] \arrow[rrdd, bend left=59] &                         &   \\
  & \Coh_\alpha\times \Quot_{\beta,\gamma} \arrow[ld] \arrow[rd] &                         & \Quot_{\alpha,\beta+\gamma} \arrow[ld] \arrow[rd] &   \\
\Coh_\alpha\times\Coh_\beta\times \Quot_\gamma &                         & \Coh_\alpha\times \Quot_{\beta+\gamma}                       &                         & \Quot_{\alpha+\beta+\gamma}
\end{tikzcd}
\end{center}
are isomorphic to the outer arrows of the diagram 
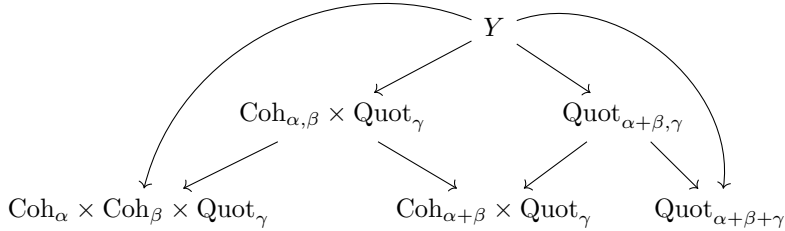
\begin{center}
\begin{tikzcd}[column sep=-20pt]
  &                         &  Y\arrow[ld] \arrow[rd] \arrow[lldd, bend right=49] \arrow[rrdd, bend left=59] &                         &   \\
  & \Coh_{\alpha,\beta}\times \Quot_\gamma \arrow[ld] \arrow[rd] &                         & \Quot_{\alpha+\beta,\gamma} \arrow[ld] \arrow[rd] &   \\
\Coh_\alpha\times\Coh_\beta\times \Quot_\gamma &                         & \Coh_{\alpha+\beta}\times \Quot_{\gamma}                       &                         & \Quot_{\alpha+\beta+\gamma}
\end{tikzcd}
\end{center}
where $X$ and $Y$ are obtained as Cartesian products. Both $X$ and $Y$ can be identified with three step derived nested Quot scheme $\Quot_{\alpha,\beta,\gamma}$ parametrizing
$$V\twoheadrightarrow Q''_{\alpha+\beta+\gamma}\twoheadrightarrow Q'_{\beta+\gamma}\twoheadrightarrow Q_\gamma
$$
with the outer arrows being obvious projections. It is straightforward to check that the action is unital, i.e., $1\star(-)$ is an identity.
\end{proof}

We remark that even when the curve $C$ is quasi-projective, torsion Quot schemes $\{\Quot_{d}\}_{d\geq 0}$ and the corresponding module of the torsion CoHA are well-defined.

\subsection{Annihilation action via dualities}\label{sec: annihilation}

In this section, we construct the annihilation action of the (Koszul signed) opposite algebra $(\BH^\full)^\op$ on the homology of Quot schemes. The construction crucially relies on the projectivity assumption of the curve $C$. 

\begin{definition}\label{def: Koszul signed opposite algebra}
    The opposite algebra of $\BH^\full$, denoted by $(\BH^\full)^\op$, is an associative algebra whose underlying group is the same as $\BH^\full$ but with a multiplication given by 
    $$a\star^{\op}b:=(-1)^{|a||b|}\cdot b\star a. 
    $$
\end{definition}

In what follows, we will construct the $(\BH^\full)^\op$-module structure on $\BV^\full$ by annihilation action, assuming that $C$ is projective. A key idea is to consider the CoHA and CoHA-module structures on the usual cohomology using the virtual Umkehr map, instead of using virtual pullback on the Borel--Moore homology groups. First, CoHA multiplication can be defined as a composition
$$\star:=p_!\circ q^*:H^a(\Coh_\alpha)\otimes H^b(\Coh_\beta)\xrightarrow{q^*}H^{a+b}(\Coh_{\alpha,\beta})\xrightarrow{p_!} H^{a+b-2\chi(\alpha,\beta)}(\Coh_{\alpha+\beta}). 
$$
This definition uses the fact that $p$ is a proper quasi-smooth morphism with no further condition on $q$. Via Poincar\'e duality isomorphism for smooth stacks
$$H^*(\Coh_\alpha)\xrightarrow{\sim}H^\BM_{-2\chi(\alpha,\alpha)-*}(\Coh_\alpha),\quad \gamma\mapsto \gamma\cap [\Coh_\alpha],
$$
CoHA multiplications on both sides coincide, so we will use them interchangeably. Second, the CoHA-module structure on the cohomology groups is defined as a composition
$$\star:=g_!\circ f^*:H^a(\Coh_\alpha)\otimes H^b(\Quot_\beta)\xrightarrow{f^*}H^{a+b}(\Quot_{\alpha,\beta})\xrightarrow{g_!} H^{a+b-2\chi(\alpha,\alpha+\beta)}(\Quot_{\alpha+\beta}). 
$$
This definition uses the fact that $g$ is a proper quasi-smooth morphism from Lemma \ref{lem: quasi-smoothness} with no further condition on $f$. The proof of Proposition \ref{prop: CoHA-module} also shows that this indeed defines a CoHA-module on cohomology groups. In other words, we have defined the creation action 
\begin{equation}\label{eq: creation on cohomology}
    \BH^\full\curvearrowright\bigoplus_\alpha H^*(\Quot_\alpha). 
\end{equation}

Consider the topological pairing, also known as Kronecker pairing, with an additional sign\footnote{The sign is inserted here to make the statement in Lemma \ref{lem: Hecke and CoHA actions} free of any signs. Since $\deg(x\cap v)=0$ unless $|x|=|v|$, the pairing can be equivalently defined by $(-1)^{|x|}\deg(x\cap v)$.}
$$H^*(\Quot_\alpha)\otimes H_*(\Quot_\alpha)\xrightarrow{\ (\,,\,)_\alpha\ }\BQ,\quad x\otimes v\mapsto (-1)^{|x||v|}\deg(x\cap v). 
$$
This is a perfect pairing by the universal coefficient theorem. We use this perfect pairing to turn the creation action on the cohomology of Quot schemes into the annihilation action of the opposite algebra on the (Borel--Moore) homology of Quot schemes.\footnote{Note that the homology group and the Borel--Moore homology group agree because Quot schemes are proper if $C$ is projective.} Roughly speaking, we will do the following:
\begin{align*}
    \displaystyle\BH^\full\curvearrowright \bigoplus_\alpha H^*(\Quot_\alpha)
    &\ \ \xleftrightarrow[\textnormal{pairing}]{\textnormal{\ Kronecker\ }}\ \ 
    \displaystyle \bigoplus_\alpha H_*(\Quot_\alpha)\curvearrowleft \BH^\full\\
    &\ \ \xleftrightarrow{\textnormal{left-to-right}}\ \ 
    \displaystyle(\BH^\full)^\op\curvearrowright \bigoplus_\alpha H_*(\Quot_\alpha). 
\end{align*}
Precisely, for each $f\in \BH^\full_\alpha$ and $v\in \BV_{\alpha+\beta}$, we define the right action $v\star f\in \BV_\beta$ as a unique element satisfying 
$$(x,v\star f)_\beta=(f\star x,v)_{\alpha+\beta}\quad \textnormal{for all}\quad x\in H^*(\Quot_\beta),
$$
where $f\star x$ on the right hand side is the creation action on the cohomology from \eqref{eq: creation on cohomology}. This indeed defines the right action because 
\begin{align*}
    (x,v\star(f_1\star f_2))=&\,((f_1\star f_2)\star x, v)\\
    =&\,(f_1\star (f_2\star x),v)\\
    =&\,(f_2\star x,v\star f_1)\\
    =&\,(x,(v\star f_1)\star f_2). 
\end{align*}
Purely algebraically, we can turn the right action into the left action of the (Koszul signed) opposite algebra by setting 
$$f\star^{\op}v:=(-1)^{|f||v|}\cdot v\star f. 
$$

\begin{definition}\label{def: annihilation}
We define the annihilation action $\star:(\BH^\full)^\op\otimes \BV^\full\rightarrow \BV^\full$ by the above construction. 
\end{definition}

\subsection{CoHA-submodule via virtual homology}

In enumerative geometry, numerical invariants are often defined by integrating tautological classes against the virtual fundamental classes. The resulting invariants, called the virtual invariants, remain unchanged under deformation of the underlying data, which are the curve $C$ and a choice of $V$ in the case of Quot schemes. In this section, we introduce the notion of virtual homology of Quot schemes which is a cohomological analogue of considering virtual invariants. 

\begin{definition}\label{def: tautological subring}
    The tautological subring $R^*(\Quot_\alpha)\subseteq H^*(\Quot_\alpha)$ is defined as a subring generated by K\"unneth components of
    $$c_i(\CS),\ c_i(\CQ_\alpha)\in H^*(\Quot_\alpha)\otimes H^*(C),\quad i\geq 1$$
    where $\CS$ and $\CQ_\alpha$ are the universal subsheaf and quotient over $\Quot_\alpha\times C$, respectively. Similarly, the tautological subring $R^*(\Quot_{\alpha,\beta})\subseteq H^*(\Quot_{\alpha,\beta})$ is defined as a subring generated by K\"unneth components of
    $$c_i(\CS),\ c_i(\CQ_\beta),\ c_i(\CS'),\ c_i(\CQ'_{\alpha+\beta}),\ c_i(\CF_\alpha)\in H^*(\Quot_{\alpha,\beta})\otimes H^*(C),\quad i\geq 1$$
    where $\CS$, $\CQ_\beta$, $\CS'$, $\CQ'_{\alpha+\beta}$, $\CF_\alpha$ are the universal objects over $\Quot_{\alpha,\beta}\times C$ as in the diagram \eqref{eq: nested Quot}. 
    
\end{definition}

\begin{remark}\label{rem: taut generation}
    By the universal short exact sequence
    $$0\rightarrow \CS\rightarrow \CO_{\Quot_{\alpha}}\boxtimes V\rightarrow \CQ\rightarrow 0,
    $$
    it suffices to use K\"unneth components of $\{c_i(\CS)\}_{i\geq 1}$ (or equivalently $\{c_i(\CQ_\alpha)\}_{i\geq 1}$) in order to define the tautological subring of $\Quot_\alpha$. Similarly, it suffices to consider K\"unneth components of $\{c_i(\CQ_\beta), c_i(\CF_\alpha)\}_{i\geq 1}$ (or equivalently $\{c_i(\CS'), c_i(\CF_\alpha)\}_{i\geq 1}$) to define the tautological subring of $\Quot_{\alpha,\beta}$. 
\end{remark}

\begin{definition}\label{def: virtual homology}
    The virtual homology group is a subspace $H^\vir_*(\Quot_\alpha)\subseteq H_*(\Quot_\alpha)$ defined as the image of a linear map
    $$R^*(\Quot_\alpha)\rightarrow H_*(\Quot_\alpha),\quad \gamma\mapsto \gamma\cap[\Quot_\alpha]^\vir,
$$
obtained by capping tautological classes with the virtual class. 
\end{definition}

\begin{remark}
    The virtual homology groups defined above should be more related to enumerative geometry than the usual homology. For instance, if the virtual dimension is negative, the virtual fundamental class vanishes and all virtual invariants are zero. In this situation, the virtual homology of the Quot scheme is trivial, even though its ordinary homology may be nontrivial.
\end{remark}

\begin{remark}
    Assume that $\Quot_\alpha$ satisfies the vanishing condition
    \begin{equation*}
\Ext^1(S_1,Q_2)=0\quad\textnormal{for all}\quad [V\twoheadrightarrow Q_1],\ [V\twoheadrightarrow Q_2]\in \Quot_\alpha.
    \end{equation*}
    By deformation theory of Quot schemes, this implies that the Quot scheme is smooth of dimension $\chi([V]-\alpha,\alpha)$ and $[\Quot_\alpha]^\vir=[\Quot_\alpha]$. Furthermore, the cohomology ring of $\Quot_\alpha$ is tautologically generated by a variation of Beauville's diagonal trick as in the proof of \cite[Prop. 4.7]{KLMP}. Therefore, the above vanishing condition implies that $H^\vir_*(\Quot_\alpha)=H_*(\Quot_\alpha)$ which holds in the following cases by a simple application of Serre duality:
    \begin{enumerate}
        \item [(i)] $C$ is any curve, $V$ is any vector bundle, and $\alpha=(0,d)$ is of torsion type,
        \item [(ii)] $C=\BP^1$, $V=\CO^{\oplus N}$ is a trivial bundle, and $\alpha$ is any topological type. 
    \end{enumerate}
    
\end{remark}

In order to construct a CoHA-module structure on the virtual homology groups of Quot schemes
$$\BV^\vir:=\bigoplus_{\alpha\in P}H_*^\vir(\Quot_\alpha)\subseteq \BV^\full,
$$
we prove below that the CoHA action on $\BV^\full$ preserves this subspace. We need the following lemma regarding the preservation of the tautological ring under the virtual Umkehr maps with respect to the forgetful maps from the nested Quot scheme
\begin{equation}\label{eq: two forgetful maps}
    g:\Quot_{\alpha,\beta}\rightarrow \Quot_{\alpha+\beta},\quad \pi:\Quot_{\alpha,\beta}\rightarrow \Quot_\beta.
\end{equation}
The above maps are both quasi-smooth and proper by Lemma \ref{lem: quasi-smoothness}, hence virtual Umkehr maps are well defined. 

\begin{lemma}\label{lem: tautological subring is preserved}
    The virtual Umkehr maps with respect to the forgetful morphisms in \eqref{eq: two forgetful maps} preserve the tautological subring, i.e., 
    $$g_!:R^*(\Quot_{\alpha,\beta})\rightarrow R^{*-2\,\vdim(g)}(\Quot_{\alpha+\beta}),\quad \pi_!:R^*(\Quot_{\alpha,\beta})\rightarrow R^{*-2\,\vdim(\pi)}(\Quot_{\beta}).
    $$
\end{lemma}
\begin{proof}
    Recall from the proof of Lemma \ref{lem: quasi-smoothness} that we have a Cartesian diagram
\begin{center}
    \begin{tikzcd}
\Quot_{\alpha,\beta} \arrow[r, "\tilde{\phi}"] \arrow[d, "g"] & \Coh_{\alpha,\beta} \arrow[d, "\tilde{g}"] \\
\Quot_{\alpha+\beta} \arrow[r, "\phi"]                 & \Coh_{\alpha+\beta}.       
\end{tikzcd}
\end{center}
By base change of the virtual Umkehr maps, the diagram
\begin{equation}\label{eq: virtual Umkehr base change}
\begin{tikzcd}
H^*(\Quot_{\alpha,\beta}) \arrow[d, "{g_!}"] & H^*(\Coh_{\alpha,\beta}) \arrow[l, "\tilde{\phi}^*"'] \arrow[d, "{\tilde{g}_!}"] \\
H^{*-2\,\vdim(g)}(\Quot_{\alpha+\beta})                & H^{*-2\,\vdim(g)}(\Coh_{\alpha+\beta}) \arrow[l, "{\phi^*}"']               
\end{tikzcd}
\end{equation}
commutes. By a result of Heinloth, building on the work of Atiyah-Bott, cohomology of $\Coh_{\alpha+\beta}$ is generated by tautological classes \cite{Heinloth_curve}. In other words, $H^*(\Coh_{\alpha+\beta})=R^*(\Coh_{\alpha+\beta})$. Since the forgetful morphism
$$q:\Coh_{\alpha,\beta}\rightarrow \Coh_\alpha\times\Coh_\beta
$$
induces an isomorphism on cohomology, the cohomology ring of $\Coh_{\alpha,\beta}$ is also tautologically generated. By Remark \ref{rem: taut generation}, the image of $\phi^*$ (resp. $\tilde{\phi}^*$) is precisely the tautological subring $R^*(\Quot_{\alpha+\beta})$ (resp. $R^*(\Quot_{\alpha,\beta})$). Therefore, commutativity of the diagram \eqref{eq: virtual Umkehr base change} implies that the virtual Umkehr map $g_!$ preserves the tautological subring. 

The other case is the same once we use \eqref{eq: morphism f} to construct a Cartesian diagram 
\begin{center}
    \begin{tikzcd}
\Quot_{\alpha,\beta} \arrow[r,"\tilde{\rho}"] \arrow[d, "\pi"] & \Coh_{[V]-(\alpha+\beta), \alpha} \arrow[d, "\tilde{\pi}"] \\
\Quot_{\beta} \arrow[r, "\rho"]                 & \Coh_{[V]-\beta}.    
\end{tikzcd}
\end{center}

\end{proof}

\begin{theorem}\label{thm: virtual homology is preserved}
    The virtual homology group $\BV^\vir\subseteq \BV^\full$ is preserved under the creation action of $\BH^\full$ and the annihilation action of $(\BH^\full)^\op$.
\end{theorem}

\begin{proof}
    We first consider the creation action. Recall the diagram 
    \begin{center}
\begin{tikzcd}
  & \Quot_{\alpha,\beta} \arrow[ld, "f"'] \arrow[rd, "g"] &   \\
\Coh_\alpha\times \Quot_\beta &                                    & \Quot_{\alpha+\beta}.
\end{tikzcd}
\end{center}
Note that all the spaces and morphisms in the diagram above are quasi-smooth. This implies that 
$$f^!([\Coh_\alpha]\boxtimes [\Quot_\beta]^\vir)=[\Quot_{\alpha,\beta}]^\vir=g^![\Quot_{\alpha+\beta}]^\vir. 
$$
Pick any $x\in R^*(\Coh_\alpha)=H^*(\Coh_\alpha)$ and $y\in R^*(\Quot_\beta)$. Then we have 
\begin{align*}
    (x\cap [\Coh_\alpha])\star (y\cap [\Quot_\beta]^\vir)
    &=g_*f^!\Big((x\cap [\Coh_\alpha])\boxtimes (y\cap [\Quot_\beta]^\vir)\Big)\\
    &=g_*\Big(f^*(x\boxtimes y)\cap [\Quot_{\alpha,\beta}]^\vir\Big)\\
    &=g_*\Big(f^*(x\boxtimes y)\cap g^![\Quot_{\alpha+\beta}]^\vir\Big)\\
    &=g_!(f^*(x\boxtimes y))\cap [\Quot_{\alpha+\beta}]^\vir, 
\end{align*}
where the last equality is the virtual projection formula. By Lemma \ref{lem: tautological subring is preserved}, we have $g_!(f^*(x\boxtimes y))\in R^*(\Quot_{\alpha+\beta})$ which proves that the virtual homology is preserved under creation action. 

We now turn to the annihilation action. Define the ideal
$$I_\alpha:=\{x\in H^*(\Quot_\alpha)\,|\,(x,y)=0\textnormal{ for all }y\in H^\vir_*(\Quot_\alpha)\}. 
$$
Then the topological pairing between the cohomology and homology groups induces a diagram 
\begin{center}
\begin{tikzcd}
  H^*(\Quot_\alpha)   \arrow[r, phantom, "\otimes"] \arrow[d]       & H_*(\Quot_\alpha) \arrow[r] & \BQ \arrow[d, equal] \\
 H^*(\Quot_\alpha)/I_\alpha  \arrow[r, phantom, "\otimes"]  &   H_*^\vir(\Quot_\alpha) \arrow[r]  \arrow[u]  & \BQ                      
\end{tikzcd}
\end{center}
where the bottom row is also a perfect pairing by definition of $I_\alpha$. Recall that the annihilation action on $H_*(\Quot_\alpha)$ was defined as a dual of the creation action on $H^*(\Quot_\alpha)$. Therefore, the annihilation action preserves the virtual homology group if and only if the creation action on the cohomology group preserves the ideal $I_\alpha$'s. Let $x\in R^*(\Coh_\alpha)=H^*(\Coh_\alpha)$ and $y\in I_\beta$. We need to show that $x\star y\in I_{\alpha+\beta}$. This amounts to showing that
$$\int_{[\Quot_{\alpha+\beta}]^\vir}(x\star y)\cup z=0\quad\textnormal{for all}\quad z\in R^*(\Quot_{\alpha+\beta}). 
$$
By the virtual projection formulas and $[\Quot_{\alpha,\beta}]^\vir=g^![\Quot_{\alpha+\beta}]^\vir$, we have
\begin{align*}
    \int_{[\Quot_{\alpha+\beta}]^\vir}(x\star y)\cup z
    &=\int_{[\Quot_{\alpha+\beta}]^\vir}g_! f^*(x\boxtimes y) \cup z\\
    &=\int_{[\Quot_{\alpha+\beta}]^\vir}g_! \Big(f^*(x\boxtimes y) \cup g^*(z)\Big)\\
    &=\int_{[\Quot_{\alpha,\beta}]^\vir}f^*(x\boxtimes y) \cup g^*(z).
\end{align*}
Using the projection map $\pi:\Quot_{\alpha,\beta}\rightarrow \Quot_\beta$ and $\phi:\Quot_{\alpha,\beta}\rightarrow \Coh_\alpha$, we further simplify this integral to
\begin{align*}
    \int_{[\Quot_{\alpha,\beta}]^\vir}f^*(x\boxtimes y) \cup g^*(z)
    &=\int_{[\Quot_\beta]^\vir}\pi_!\Big(f^*(x\boxtimes y) \cup g^*(z)\Big)\\
    &=(-1)^{|x||y|}\int_{[\Quot_\beta]^\vir}y\cup \pi_!\Big(\phi^*(x) \cup g^*(z)\Big).
\end{align*}
Since $y\in I_{\beta}$, this integral vanishes if $\pi_!\big(\phi^*(x) \cup g^*(z)\big)\in R^*(\Quot_\beta)$. The latter follows from Lemma \ref{lem: tautological subring is preserved}, hence completing the proof. 
\end{proof}

In the above, we have shown that $\BV^\vir\subseteq \BV^\full$ is a $\BH^\full$-submodule via creation action. In fact, we show below that the virtual homology is generated by the vacuum vector as a $\BH^\full$-module, hence giving a representation theoretic reason to consider the virtual homology groups instead of the ordinary homology groups. In other words, $\BV^\vir$ is the cyclic module of $\BH^\full$. 

We first define the vacuum vector using the fact that $\Quot_{(0,0)}=\pt$. 

\begin{definition}
    The vacuum vector $|0\rangle\in \BV_{(0,0)}$ is the one that corresponds to $1$ under the identification $H_0(\Quot_{(0,0)})\simeq \BQ$.
\end{definition}

\begin{proposition}\label{prop: cyclic by vacuum}
    The virtual homology $\BV^{\vir}$ is generated by the vacuum vector $|0\rangle$ as a $\BH^{\full}$-module.
\end{proposition}
\begin{proof}
    Since $\Quot_{(0,0)}=\pt$, the roof diagram for the CoHA-module with $\alpha\in P$ and $\beta=(0,0)$ becomes
    \begin{center}
\begin{tikzcd}
  & \Quot_{\alpha} \arrow[ld, "f"'] \arrow[rd, "g=\id"] &   \\
\Coh_\alpha&                                    & \Quot_{\alpha},
\end{tikzcd}
\end{center}
If $x\in R^*(\Coh_\alpha)=H^*(\Coh_\alpha)$, then $x\star|0\rangle = f^*(x)\cap [\Quot_\alpha]^\vir\in \BV^\vir_\alpha$. Since the image of $H^*(\Coh_\alpha)=R^*(\Coh_\alpha)$ under $f^*$ is precisely the tautological subring $R^*(\Quot_\alpha)$, this shows that $\BV^\vir$ is generated by $|0\rangle$ as a $\BH^\full$-module. In fact, this gives another proof for the fact that $\BV^\vir\subseteq \BV$ is a $\BH^\full$-submodule.

\end{proof}

\section{Structure of the torsion CoHA}\label{sec: Shuffle algebra}
For this section, we let $C$ be a smooth connected quasi-projective curve.  

\subsection{Shuffle algebra} 

We start by reviewing shuffle algebras. Let $V$ be a $\BZ$-graded vector space over $\BQ$ with the induced super-grading $|\cdot|$. Consider the tensor algebra of $V$
$$T_*(V):=\bigoplus_{d\geq 0}T_d(V)\quad\textnormal{where}\quad T_d(V)=\underbrace{V \otimes \cdots \otimes V}_{d \ \mathrm{times}}. 
$$
The symmetric group $S_d$ acts on $T_d(V)$ such that the action of a transposition is given by 
$$\sigma_{i,i+1}(v_{1} \otimes \cdots \otimes v_{i} \otimes v_{i+1} \otimes \cdots \otimes v_{d}) = (-1)^{\mid v_i \mid \mid v_{i+1} \mid}v_{1} \otimes \cdots \otimes v_{i+1} \otimes v_{i} \otimes \cdots \otimes v_{d}.$$
The subspace of symmetric tensors 
$$\Sh_*(V):=\bigoplus_{d\geq 0} \Sym_d(V)\quad\textnormal{where}\quad\Sym_d(V)=\Big(\underbrace{V \otimes \cdots \otimes V}_{d \ \mathrm{times}}\Big)^{S_d}
$$
is equipped with an algebra structure via shuffle product
\begin{align} \label{def:Shufflealgebrausualstructure}
(v_1 \otimes \cdots \otimes v_n) \cdot (v_{n+1} \otimes \cdots\otimes v_{n+m}) := \sum_{\sigma \in \Sh_{n,m}} \sigma(v_1 \otimes \cdots \otimes v_n \otimes v_{n+1} \otimes \cdots\otimes v_{n+m}). \end{align}
Here $\Sh_{n,m}\subseteq S_{n+m}$ denotes the subset of permutations $\sigma$ satisfying $\sigma(1)<\sigma(2)<\cdots<\sigma(n)$ and $\sigma(n+1)<\sigma(n+2)<\cdots<\sigma(n+m)$. 

Note that we can also consider the coinvariant ring, i.e., the quotient algebra 
\[ \Sym^*(V) := T_*(V)/\langle v_{1} \otimes v_{2} = (-1)^{\mid v_1 \mid \mid v_2 \mid} v_{2} \otimes v_{1} \rangle. \]
Both $\Sh_*(V)$ and $\Sym^{*}(V)$ are doubly graded supercommutative algebras; one grading is induced from the $\BZ$-grading of $V$ from which the super-grading is induced, and the other in terms of the number of tensor factors.

\begin{lemma}\label{lem: isomorphismofshufflealgebraandcoalgebra}
There is an isomorphism of doubly graded algebras 
\begin{align}  F: \Sym^*(V)\xrightarrow{\sim} \Sh_*(V) \end{align} given by linearly extending
\[ \overline{v_{1} \otimes \cdots \otimes v_n} \mapsto  \sum_{\sigma \in S_n} \sigma(v_1 \otimes \cdots \otimes v_n). \]
\end{lemma} 

\begin{proof}
    The map $F$ is well-defined because
    $$\sum_{\sigma\in S_n}\sigma (\tau(v_1\otimes \cdots\otimes v_n))=\sum_{\sigma\in S_n}\sigma (v_1\otimes \cdots\otimes v_n)=\tau\Big(\sum_{\sigma\in S_n}\sigma (v_1\otimes \cdots\otimes v_n)\Big)
    $$
    for any $\tau\in S_n$. By definition of the shuffle product, it follows that $F$ is an algebra homomorphism, which obviously preserves the $\BZ$-grading. The map $F$ is an isomorphism because it admits the inverse $G:\Sh_*(V)\rightarrow\Sym^*(V)$, defined by linearly extending $x\mapsto \overline{x}/d!$ where $\overline{x}$ is the image of $x$ under the composition $\Sym_d(V)\subseteq T_d(V)\twoheadrightarrow \Sym^d(V)$.
    
\end{proof}

We now introduce the shuffle algebra of a curve $C$. We start with a $\BZ$-graded vector space $V = H^{*}(C)[z]$ whose grading is given by the cohomological grading on $H^{*}(C)$ and $\deg(z^n):=2n$. Using the identification of graded vector spaces 
$$\underbrace{\BQ[z]\otimes \cdots \otimes \BQ[z]}_{d \ \mathrm{times}}=\BQ[z_1,\dots,z_d],\quad z^{k_1}\otimes\cdots \otimes z^{k_d}\longleftrightarrow z_1^{k_1}\cdots z_d^{k_d},
$$
we can also identify
\begin{align} 
\Sh_*(H^{*}(C)[z]) = \bigoplus_{d \geq 0} \Big((H^{*}(C)^{\otimes d}) [z_1,\cdots,z_d]\Big)^{S_d}.
\end{align}  
Under this identification, the action of the symmetric group $S_d$ on $(H^{*}(C)^{\otimes d}) [z_1,\cdots,z_d]$ is such that the transposition acts by 
\begin{multline*}
    \sigma_{i,i+1}\Big((\alpha_1 \otimes \cdots \otimes\alpha_i \otimes \alpha_{i+1}\otimes 
    \cdots \otimes \alpha_d) z_{1}^{k_1} \cdots z_{i}^{k_i} z^{k_{i+1}}_{i+1}\cdots z_{d}^{k_d}\Big)\\
    = (-1)^{|\alpha_i||\alpha_{i+1}|} (\alpha_1 \otimes \cdots \otimes \alpha_{i+1} \otimes \alpha_{i}\otimes \cdots \otimes \alpha_d) z_{1}^{k_1} \cdots z_{i}^{k_{i+1}} z^{k_i}_{i+1}\cdots z_{d}^{k_d}.
\end{multline*}

By the earlier discussion, $\Sh_*(H^{*}(C)[z])$ is equipped with the usual shuffle product. On the same underlying vector space, we will introduce a new shuffle product defined by the kernel
\begin{equation}\label{eq: shuffle kernel}
    K_{d,e}(z_1,\dots,z_{d+e}):=\prod_{\substack{1 \leq i \leq d \\ d+1 \leq j \leq d+e}}\left(1+\frac{\Delta_{ij}}{z_j-z_i}\right)\in H^*(C)^{\otimes (d+e)}(z_1,\dots,z_{d+e})
\end{equation}
where $\Delta_{ij}\in H^2(\,\underbrace{C \times \cdots \times C}_{d+e \textrm{ times }}\,)$ is pullback of the diagonal class from $i$ and $j$ factors. Note that this definition requires the ring structure on $H^*(C)^{\otimes (d+e)}(z_1,\dots, z_{d+e})$.

\begin{definition} \label{defnofshuffleproduct}
For each $f \in \Sym_{d}(H^{*}(C)[z])$ and $g \in \Sym_{e}(H^{*}(C)[z])$, the shuffle product is defined as 
\[ f(z_1,\cdots,z_d) \star g(z_1,\cdots,z_e) := \sum_{\sigma \in \Sh_{d,e}} \sigma \left( f(z_1,\cdots,z_d)\cdot g(z_{d+1},\cdots,z_{d+e})\cdot K_{d,e}(z_1,\cdots,z_{d+e})\right).\] 
The resulting algebra, denoted by $\Sh_*^{\Delta}(H^{*}(C)[z])$, is called the shuffle algebra of $C$. 

\end{definition}
Since the shuffle kernel has denominators of the form $(z_j-z_i)$, it is nontrivial that the shuffle product of two symmetric polynomials is again a polynomial. We prove this in the proposition below.\footnote{Alternatively, we can prove that the shuffle product equals the multiplication in the cohomological Hall algebra, which would then automatically imply the polynomiality of the shuffle product. }

\begin{proposition}\label{prop: no pole}
    In Definition \ref{defnofshuffleproduct}, $f\star g\in \Sym_{d+e}(H^*(C)[z])$. 
\end{proposition}
\begin{proof}

    Using the $S_d$ and $S_e$ invariance of $f$ and $g$, respectively, we can rewrite the shuffle product as \begin{equation} \label{shufflesymmetrization}
        f \star g = \frac{1}{d!e!}\sum_{\sigma\in S_{d+e}} \sigma\left(  f(z_1,\cdots,z_d)\cdot g(z_{d+1},\cdots,z_{d+e})\cdot\hspace{-10pt}\prod_{\substack{1 \leq i \leq d \\ d+1 \leq j \leq d+e}}\left(1+\frac{\Delta_{ij}}{z_j-z_i}\right)\right). 
    \end{equation} 
    Consider the Vandermonde polynomial
    $$V_{d+e}(z_1,\dots,z_{d+e}):=\prod_{1\leq i<j\leq d+e}(z_j-z_i).
    $$
    Since $f$ and $g$ are polynomials, we have 
    $$f\star g=\frac{h(z_1,\dots,z_{d+e})}{V_{d+e}(z_1,\dots,z_{d+e})}
    $$
    for some polynomial $h(z_1,\dots,z_{d+e})$. Since $f\star g$ is symmetric and $V_{d+e}$ is anti-symmetric, $h$ is anti-symmetric. However, any anti-symmetric polynomial is a product of the Vandermonde polynomial and a symmetric polynomial, hence proving the proposition. 
\end{proof}

From now on, we shall refer to the shuffle algebra $\Sh_{*}(H^*(C)[z])$ with usual shuffle product as the shuffle algebra with trivial kernel.

\subsection{Cohomology of torsion stacks}

The cohomology of the moduli stack of coherent sheaves on a smooth projective curve is calculated by Heinloth in \cite{Heinloth_curve}. We recall the argument for torsion sheaves, which also works for quasi-projective cases. Since we will be working with torsion sheaves on a fixed curve $C$, from now on we shall simply write $\Coh_d=\Coh_{(0,d)}(C)$. 

Given any torsion sheaf $F$ on $C$, its degree is defined as \[ \deg(F) = \sum_{p \in \mathrm{Supp}(F)} \mathrm{length}_{\mathcal{O}_{C,p}}(F_{p}).\] 
Since every torsion sheaf of degree $1$ is a skyscraper sheaf, there is an isomorphism of stacks
\[ \Coh_{1} \simeq C \times \BGm. \]
Therefore, its cohomology is canonically identified with $H^*(\Coh_1)\simeq H^*(C)[z]$. In order to study $H^*(\Coh_d)$ for higher degrees, it is useful to consider the stack $\tilde{\Coh}_{d}$ parametrizing flags of torsion sheaves
$$F_\bullet=(0=F_0\subseteq F_1\subseteq F_2\subseteq \cdots\subseteq F_d),\quad \deg(F_i)=i.
$$
We consider two morphisms from this stack
\[\begin{tikzcd}[ampersand replacement=\&]
	\& {\tilde{\Coh}_{d} } \\
	{\prod_{i=1}^d\Coh_{1}} \&\& {\Coh_{d}}
	\arrow["{\mathrm{gr}}"', from=1-2, to=2-1]
	\arrow["p", from=1-2, to=2-3]
\end{tikzcd}\]
where 
$$\gr(F_\bullet)=(F_{d}/F_{d-1},\dots, F_{1}/F_{0}),\quad p(F_\bullet)=F_d.
$$
The morphism $\gr$ is an iterated vector bundle stack, hence induces an isomorphism
\begin{equation}\label{eq: gr iso}
\mathrm{gr}^{*}: H^{*}(\prod_{i=1}^{d} \Coh_{1}) \simeq H^{*}(\tilde{\Coh}_{d}).
\end{equation}  
On the other hand, it is proved in \cite[Thm. 3.3.1]{laumon} that the forgetful map $p$ is small and generically a Galois covering with the Galois group $S_d$. Thus it follows that 
\begin{equation*}\label{eq: Galois iso}
p^{*}: H^{*}(\Coh_{d}) \simeq (H^{*}(\tilde{\Coh}_{d}))^{S_d}
\end{equation*} 
Under the isomorphism \eqref{eq: gr iso} the action of $S_d$ corresponds to swapping factors of $H^*(\prod_{i=1}^d \Coh_1)$ according to the Koszul sign rule. Therefore, we conclude that
\begin{equation}\label{eq: heinloth iso} 
H^{*}(\Coh_{d}) \simeq \Sym_d(H^{*}(\Coh_{1}(C))) \simeq \Big((H^{*}(C)^{\otimes d}) [z_1,\cdots,z_d]\Big)^{S_d}.
\end{equation}

\subsection{Proof of the comparison theorem}

In this section, we prove the following theorem. 
\begin{theorem}\label{thm: CoHA equals shuffle}
The identification \eqref{eq: heinloth iso} induces an isomorphism of doubly graded algebras 
$$\BH\simeq \Sh_*^{\Delta}(H^{*}(C)[z]). 
$$
\end{theorem}

The proof is done by carefully analyzing the geometry of the correspondence diagram
\begin{equation*}
\begin{tikzcd}
  & \Coh_{d,e} \arrow[ld, "q"'] \arrow[rd, "p"] &   \\
\Coh_d\times \Coh_e &                                    & \Coh_{d+e}.
\end{tikzcd}
\end{equation*}
More precisely, we will realize all the stacks in the above diagram as explicit global quotient stacks, apply the torus localization formula, and then compute the normal bundle in order to compute the pushforward $p_*$.

\subsubsection{Global quotient stack descriptions}\label{sec: quotient stack description} 

We start by explaining how to realize $\Coh_d$ as an explicit global quotient stack. Consider $\Quot_d(\BC^d\otimes\CO_C)$ parametrizing quotients $\phi:\BC^d\otimes\CO_C\twoheadrightarrow F_d$ where $F_d$ is a torsion sheaf of degree $d$. Note that this is a smooth quasi-projective variety of dimension $d^2$. Let $Q_d\subseteq \Quot_d(\CO^{\oplus d})$ be an open subscheme parametrizing quotients such that the adjunction map $\BC^d\rightarrow H^0(C,F_d)$ induced by $\phi$ is an isomorphism. Note that there exists a natural left action of $\GL_d$ on $\Quot_d(\BC^d\otimes\CO_C)$ defined by $g\cdot \phi:=\phi\circ g^{-1}$, under which $Q_d$ is invariant. Since every torsion sheaf of degree $d$ is globally generated, there exists a natural isomorphism of stacks 
\begin{equation}\label{eq: quotient description for Coh_d}
    \Coh_d\simeq[Q_d/\GL_d],
\end{equation}
whose proof we refer to \cite[Prop. 1.3]{Minets}. 

We now explain a global quotient stack description for $\Coh_{d,e}$. Fix an embedding of standard vector spaces $\BC^d\subseteq \BC^{d+e}$ by using the first $d$ coordinates.  For each point $[\phi:\BC^{d+e}\otimes \CO_C\twoheadrightarrow F_{d+e}]\in Q_{d+e}$, define
\begin{align}\label{subquotient}
F_d:=\textnormal{Im}(\BC^d\otimes\CO_C\subseteq \BC^{d+e}\otimes \CO_C\xrightarrow{\phi}F_{d+e}).
\end{align} 

Since $F_d$ is a torsion sheaf generated by $d$ sections, its degree is at most $d$. We define $Q_{d,e}\subseteq Q_{d+e}$ to be the locus where the degree of the resulting $F_d$ is exactly $d$. It was proved in \cite[Prop. 1.8]{Minets} that $Q_{d,e}\subseteq Q_{d+e}$ is a smooth closed subvariety of codimension $d\cdot e$. 

Let $\GL_{d,e}\subseteq \GL_{d+e}$ be a parabolic subgroup consisting of $g\in \GL_{d+e}$ such that $g(\BC^d)=\BC^d$, according to the fixed embedding $\BC^d\subseteq \BC^{d+e}$. Then $Q_{d,e}\subseteq Q_{d+e}$ is invariant under the $\GL_{d,e}$-action. It was shown in \cite[Sec. 3.2.1]{Sala-Schiffmann}\footnote{Much of the discussion in the loc. cit. greatly simplifies once we work with torsion sheaves. } that we have a natural isomorphism of stacks
\begin{equation}\label{eq: quotient description for Coh_d,e}
\Coh_{d,e}\simeq[Q_{d,e}/\GL_{d,e}]. 
\end{equation}

\subsubsection{Torus fixed locus descriptions}  Cohomology of global quotient stacks can be studied via torus localization formula. In order to apply the localization technique, we discuss the fixed locus of $Q_d$ and a certain space obtained from $Q_{d,e}$ according to the action of the maximal diagonal torus $$T_d\subseteq \GL_{d},\quad T_{d+e}\subseteq\GL_{d,e},$$ 
respectively.

\begin{proposition}\label{prop: fixed locus of Q_d}
    For any $d\geq 0$, we have $(Q_d)^{T_d}=C^d$. 
\end{proposition}
\begin{proof}
    Recall the result of \cite{Bifet} on the fixed loci of the Quot scheme
$$\Quot_d(\BC^d\otimes\CO_C)^{T_d}=\bigsqcup_{n_1+\dots+n_d=d} \prod_{i=1}^dC^{[n_i]},
$$
where $C^{[n]}=\Quot_n(\CO_C)$ is the Hilbert scheme of $n$ points on $C$. Because of the open condition defining $Q_d\subseteq \Quot_d(\BC^d\otimes\CO_C)$, only $n_1=\cdots=n_d=1$ case occurs as the fixed locus $(Q_d)^{T_d}$.
\end{proof}

We will need to discuss one more type of the torus fixed locus which we discuss now. Consider the balanced quotient space
$$\GL_{d+e}\times_{\GL_{d,e}}Q_{d,e}:=\GL_{d+e}\times Q_{d,e}/(g\cdot p,\phi)\sim (g, p\cdot\phi)
$$
where $(g,\phi)\in \GL_{d+e}\times Q_{d,e}$ and $p\in \GL_{d,e}$. On this quotient space, we define the left $T_{d+e}$-action via inclusion $T_{d+e} \subseteq \mathrm{GL}_{d+e}$. More precisely, we have 
$$t\cdot \overline{(g,\phi)}:=\overline{(t\cdot g, \phi)},\quad t\in T_{d+e}.
$$

We introduce some notations required to state the proposition below regarding the $T_{d+e}$-fixed loci of the balanced quotient. We write $[n]=\{1,\dots,n\}$ for any integer $n\geq 0$. For each subset $I\subseteq [d+e]$ of size $d$, we denote by $\sigma_I\in \mathrm{Sh}_{d,e} \subseteq S_{d+e}$ the unique permutation such that 
$$\sigma_I(\{1,\dots,d\})=I,\quad \sigma_I(1)<\cdots<\sigma_I(d),\quad \sigma_I(d+1)<\cdots<\sigma_I(d+e). 
$$
For each such $I$, we also define a subscheme
$$C^{d+e}_I:=\{\overline{(\sigma_I, (p_1,\dots,p_{d+e}))}\,|\, p_1,\dots, p_{d+e}\in C\}\hookrightarrow \GL_{d+e}\times_{\GL_{d,e}} Q_{d,e}.
$$

\begin{proposition}\label{prop: fixed locus of balanced quotient}
For any $d,e \geq 0$, we have
\[ (\GL_{d+e}\times_{\GL_{d,e}}Q_{d,e})^{T_{d+e}} = \bigsqcup_{I \subseteq [d+e],\ |I|=d} C^{d+e}_I.\] 
\end{proposition}

\begin{proof}
This follows from a standard fact about the action by parabolic subgroups like $\GL_{d,e}\subseteq \GL_{d+e}$, see for example \cite[Prop. A.18, Rem. A.19]{Minets}. In loc. cit., it is proved that 
$$(\GL_{d+e}\times_{\GL_{d,e}}Q_{d,e})^{T_{d+e}}=S_{d+e}\times_{(S_d\times S_e)} (Q_{d,e})^{T_{d+e}},
$$
where we used the fact that taking Weyl groups of $\GL_{d,e}\subseteq \GL_{d+e}$ yields $S_d\times S_e\subseteq S_{d+e}$. Here $\sigma\in S_d\times S_e$ if $\sigma(\{1,\dots,d\})=\{1,\dots,d\}$. By Proposition \ref{prop: fixed locus of Q_d}, we have $(Q_{d,e})^{T_{d+e}}=(Q_{d+e})^{T_{d+e}}=C^{d+e}$, so we can further identify the fixed locus with 
$$S_{d+e}\times_{(S_d\times S_e)}C^{d+e},
$$
where the left action of $S_d\times S_e$ on $C^{d+e}$ is given by
$$\sigma\cdot (p_1,\dots,p_{d+e}):=(p_{\sigma(1)},\dots,p_{\sigma(d+e)}).
$$
Note that every orbit of the right action of $S_d\times S_e$ on $S_{d+e}$ can be uniquely represented by a shuffle $\sigma\in \Sh_{d,e}$. Writing $I:=\sigma(\{1,\dots,d\})$, we obtain
$$S_{d+e}\times_{(S_d\times S_e)} C^{d+e}=\bigsqcup_{I\subseteq[d+e],\ |I|=d} C^{d+e}_I. 
$$

\end{proof}

\subsubsection{Normal bundle computations}
In this section, we compute normal bundles appearing in the localization formula to be applied later in the proof. For the diagonal torus $T_d$, we denote by $\bt_i$ the weight $1$ representation with respect to the $i$'th coordinate of $T_d$.

\begin{proposition} \label{prop: normalofqd}
    The $K$-theory class of the normal bundle of the fixed locus $C^d=(Q_d)^{T_d}\hookrightarrow Q_d$ is given by 
    $$N_{C^d/Q_d}=\bigoplus_{1\leq i\neq j\leq d}\CO_{C^d}(\Delta_{ij})\bt_i^{-1}\bt_j.
    $$

\end{proposition}

\begin{proof}
Recall that the Quot scheme $\Quot_{d}(\BC^d\otimes \CO_C)$ comes with the universal quotient and subsheaf which fit into the universal exact sequence 
\begin{equation}\label{eq: universal exact sequence for Q_d}
    0 \rightarrow \mathcal{S}_{d} \rightarrow \mathcal{O}^{d}_{\Quot_{d}(\BC^d\otimes \CO_C) \times C} \rightarrow \mathcal{Q}_{d} \rightarrow 0.
\end{equation}
The tangent bundle of the Quot scheme is given by 
\[ T(\Quot_{d}(\BC^d\otimes \CO_C)) =  \pi_{*}(\mathcal{H}\mathrm{om}(\mathcal{S},\mathcal{Q}))\]
where $\pi:\Quot_{d}(\BC^d\otimes \CO_C)\times C\rightarrow \Quot_{d}(\BC^d\otimes \CO_C)$ is the projection map. Restricting to the open subscheme $Q_d\subseteq \Quot_{d}(\BC^d\otimes \CO_C)$, we obtain the same description for the tangent bundle $T(Q_d)$. 
Restriction of the universal exact sequence \eqref{eq: universal exact sequence for Q_d} to $C^d\times C$ becomes 
$$0\rightarrow \bigoplus_{i=1}^d \CO_{C^d\times C}(-\Delta_{i\bullet})\bt_i\rightarrow \bigoplus_{i=1}^d \CO_{C^d\times C}\bt_i\rightarrow \bigoplus_{i=1}^d\CO_{\Delta_{i\bullet}}\bt_i\rightarrow 0
$$
where $\Delta_{i\bullet}\subseteq C^d\times C$ is the pullback of the diagonal along the projection map 
$$p_{i\bullet}:C^d\times C\rightarrow C\times C,\quad ((p_1,\dots, p_d),x)\mapsto (p_i,x). 
$$
Therefore, restriction of the tangent bundle to $C^d$ is given by 
$$T(Q_d)\big|_{C^d}=\bigoplus_{1\leq i,j\leq d}
\pi'_*\big(\mathcal{H}\mathrm{om}(\CO(-\Delta_{i\bullet}), \CO_{\Delta_{j\bullet}})\big)\bt_i^{-1}\bt_j
$$
where $\pi':C^d\times C\rightarrow C^d$ is the projection map. Note that 
$$\pi'_*\big(\mathcal{H}\mathrm{om}(\CO(-\Delta_{i\bullet}), \CO_{\Delta_{j\bullet}})\big)
=\pi'_*\big(\CO_{\Delta_{j\bullet}}(\Delta_{i\bullet})\big)=\CO_{C^d}(\Delta_{ij}). 
$$
Since the normal bundle is the moving part of the restriction of the tangent bundle, the proposition follows.

\end{proof}

Recall that for each subset $I\subseteq [d+e]$ of size $d$, we have a shuffle $\sigma\in \Sh_{d,e}\subseteq S_{d+e}$ and a component $C^{d+e}_I$ of the $T_{d+e}$-fixed locus of $(\GL_{d+e}\times_{\GL_{d,e}} Q_{d,e})$.

\begin{proposition} \label{prop: normalbundleext}
    The $K$-theory class of the normal bundle 
    of $C^{d+e}_I\subseteq (\GL_{d+e}\times_{\GL_{d,e}} Q_{d,e})$ is given by\footnote{We identify $C^{d+e}_I\simeq C^{d+e}$ via $\overline{(\sigma_I, (p_1,\dots,p_{d+e}))}\mapsto (p_1,\dots,p_{d+e})$.} 
\begin{align*}
\bigoplus_{\substack{1 \leq i \leq d \\ d+1 \leq j \leq d+e}} \bt_{\sigma_I(i)}^{-1}\bt_{\sigma_I(j)}  \ \bigoplus_{1 \leq i \neq j \leq d} \mathcal{O}_{C^{d+e}}(\Delta_{ij})\bt_{\sigma_I(i)}^{-1}\bt_{\sigma_I(j)} \ \bigoplus_{\substack{d+1 \leq i \leq d+e \\ 1 \leq j \leq d+e\\ i\neq j}} \mathcal{O}_{C^{d+e}}(\Delta_{ij}) \bt_{\sigma_I(i)}^{-1}\bt_{\sigma_I(j)}.  
\end{align*}

\end{proposition}

\begin{proof}

On $Q_{d,e}\times C$, the universal objects yield the following diagram
\begin{center}
    \begin{tikzcd}
            & 0 \arrow[d]           & 0 \arrow[d]           & 0 \arrow[d]           &   \\
0 \arrow[r] & \CG_d \arrow[r] \arrow[d] & \CO^{d} \arrow[r] \arrow[d] & \CF_d \arrow[r] \arrow[d] & 0 \\
0 \arrow[r] & \CG_{d+e} \arrow[r] \arrow[d] & \CO^{d+e} \arrow[r] \arrow[d] & \CF_{d+e} \arrow[r] \arrow[d] & 0 \\
0 \arrow[r] & \CG_{e} \arrow[r] \arrow[d] & \CO^e \arrow[r] \arrow[d] & \CF_e \arrow[r] \arrow[d] & 0 \\
            & 0                     & 0                     & 0                     &  
\end{tikzcd}
\end{center}
where all rows and columns are exact. Let $\pi:Q_{d,e}\times C\rightarrow Q_{d,e}$ be the projection map. By \cite[Prop. 1.10]{Minets}, the tangent bundle of $Q_{d,e}$ sits inside the short exact sequence
$$0\rightarrow T(Q_{d,e})\rightarrow \pi_*\CH om(\CG_{d+e}, \CF_{d+e})\rightarrow \pi_*\CH om(\CG_d,\CF_e)\rightarrow 0,
$$
where the third arrow is induced by $\CG_d\subseteq \CG_{d+e}$ and $\CF_{d+e}\twoheadrightarrow \CF_e$. By an analysis similar to that in the proof of Proposition \ref{prop: normalofqd}, it follows that the normal bundle of $C^{d+e} \subseteq Q_{d,e}$ is given by 
\begin{equation}\label{eq: normal of Q_d,e}
    \bigoplus_{1 \leq i \neq j \leq d} \mathcal{O}_{C^{d+e}}(\Delta_{ij})\bt_{i}^{-1}\bt_{j} \ 
\bigoplus_{\substack{d+1 \leq i \leq d+e \\ 1 \leq j \leq d+e\\ i\neq j}} \mathcal{O}_{C^{d+e}}(\Delta_{ij}) \bt_{i}^{-1}\bt_{j}.
\end{equation}

Consider a fiber diagram  
\begin{center}
    \begin{tikzcd}
{\{\id\}\times Q_{d,e}} \arrow[d] \arrow[r] & {\GL_{d+e}\times_{\GL_{d,e}}Q_{d,e}} \arrow[d, "f"] \\
\{\id\} \arrow[r]                           & {\GL_{d+e}/\GL_{d,e}} .                             
\end{tikzcd}
\end{center}
Note that the top horizontal map is $T_{d+e}$-equivariant where the action of $T_{d+e}$ on $\{\mathrm{id} \}  \times Q_{d,e}$ is defined by $(\mathrm{id},\phi)\mapsto (\mathrm{id},t \cdot \phi)$ for any $t \in T$ and $\phi \in Q_{d,e}$. This is because for $(\mathrm{id},\phi) \in \mathrm{GL}_{d+e} \times _{\mathrm{GL}_{d,e}}Q_{d,e}$, $t\cdot(\mathrm{id},\phi) = (t,\phi) = (\mathrm{id}, t \cdot \phi)$. Similarly the projection morphism $f$ is $T_{d+e}$ equivariant for the action of $T_{d+e} $ on $\mathrm{GL}_{d+e}/\mathrm{GL}_{d,e}$ defined by $t \cdot \overline{g} = \overline{t\cdot g}$ for any orbit $\overline{g}$. The tangent bundle short exact sequence, restricted to the fiber, becomes
\begin{equation}\label{eq: fiber sequence}
    0\rightarrow T({\{\id\}\times Q_{d,e}})\rightarrow T(\GL_{d+e}\times_{\GL_{d,e}}Q_{d,e})\big|_{\{\id\}\times Q_{d,e}}\rightarrow \mathfrak{gl}_{d+e}/\mathfrak{p}_{d,e}\otimes \CO\rightarrow 0
\end{equation}
where $\mathfrak{gl}_{d+e}$ and $\mathfrak{p}_{d,e}$ are the Lie algebras of $\GL_{d+e}$ and $\GL_{d,e}$, respectively. Let $I_0:=\{1,2,\dots,d\}$. Then $C_{I_0}^{d+e}=\{\id\}\times C^{d+e}\subseteq (\GL_{d+e}\times_{\GL_{d,e}} Q_{d,e})$. Further restricting the sequence \eqref{eq: fiber sequence} to $C^{d+e}_{I_0}$ and taking the moving part, we obtain
$$0\rightarrow N_{C^{d+e}/Q_{d,e}}\rightarrow N_{C^{d+e}_{I_0}/(\GL_{d+e}\times_{\GL_{d,e}} Q_{d,e})}\rightarrow \mathfrak{gl}_{d+e}/\mathfrak{p}_{d,e}\otimes \CO\rightarrow 0. 
$$
By the torus weight of the normal bundle formula \eqref{eq: normal of Q_d,e} and the well-known weight decomposition of $\mathfrak{gl}_{d+e}/\mathfrak{p}_{d,e}$, we obtain that the normal bundle of $C^{d+e}_{I_0}$ is equal to  
\[ \bigoplus_{\substack{1 \leq i \leq d\\ d+1 \leq j \leq d+e}} \bt_i^{-1}\bt_j \bigoplus_{1 \leq i \neq j \leq d} \mathcal{O}_{C^{d+e}}(\Delta_{ij})\bt_{i}^{-1}\bt_{j} \ \bigoplus_{\substack{d+1 \leq i \leq d+e \\ 1 \leq j \leq d+e\\ i\neq j}} \mathcal{O}_{C^{d+e}}(\Delta_{ij}) \bt_{i}^{-1}\bt_{j} \] 
in $K$-theory.

Now we consider a general subset $I\subset[d+e]$ of size $d$. The permutation $\sigma=\sigma_I$ can be thought of as a permutation matrix in $\GL_{d+e}$. The left action by $\sigma$ on ${\GL_{d+e}\times_{\GL_{d,e}}Q_{d,e}}$ induces a commuting diagram
\begin{equation}\label{eq: useful diagram}
    \begin{tikzcd}
{{\GL_{d+e}\times_{\GL_{d,e}}Q_{d,e}}} \arrow[r, "\sigma", "\simeq"'] & {{\GL_{d+e}\times_{\GL_{d,e}}Q_{d,e}}} \\
C^{d+e}_{I_0} \arrow[r, leftrightarrow, "\simeq"'] \arrow[u, hook]                    & C^{d+e}_I. \arrow[u, hook]       
\end{tikzcd}
\end{equation}
The isomorphism $\sigma$ becomes $T_{d+e}$-equivariant if we consider the usual $T_{d+e}$-action on the codomain and the conjugate torus action by $\sigma^{-1} T_{d+e}\sigma$ on the domain. Since $$\sigma^{-1} (t_1,\dots,t_{d+e})\sigma = (t_{\sigma(1)}, \dots, t_{\sigma(d+e)}),$$
the normal bundle of $C^{d+e}_I$ is the same as the normal bundle of $C^{d+e}_{I_0}$ up to permutation of the torus weights by $\sigma_I$. This completes the proof.

\end{proof}

\subsubsection{Proof of Theorem~\ref{thm: CoHA equals shuffle}}

The proof follows the same strategy as in \cite[Thm. 2]{Kontsevich_Soibelman}. Using the quotient stack description in Section \ref{sec: quotient stack description}, we can rewrite the correspondence diagram as 
\[\begin{tikzcd}[ampersand replacement=\&]
	\& {Q_{d,e}/\GL_{d,e}} \& {(\GL_{d+e}\times_{\GL_{d,e}}Q_{d,e})/\GL_{d+e}} \\
	{Q_{d}/\mathrm{GL}_{d} \times Q_{e}/\mathrm{GL}_{e}} \&\& {Q_{d+e}/\mathrm{GL}_{d+e}}
	\arrow["\simeq", from=1-2, to=1-3]
	\arrow["q"', from=1-2, to=2-1]
	\arrow["p", from=1-2, to=2-3]
	\arrow["\rho", from=1-3, to=2-3]
\end{tikzcd}\] where $Q_{d,e}/\GL_{d,e}\simeq (\GL_{d+e}\times_{\GL_{d,e}}Q_{d,e})/\GL_{d+e}$ is the usual isomorphism and $\rho$ is induced from the $\GL_{d+e}$-equivariant morphism 
$$\tilde\rho:\GL_{d+e}\times_{\GL_{d,e}}Q_{d,e}\rightarrow Q_{d+e},\quad \overline{(g,\phi)}\mapsto g\cdot \phi.
$$

The Hall algebra product is the composition
$$H^*_{\GL_d}(Q_d)\otimes H^*_{\GL_e}(Q_e)\xrightarrow{q^*}
H^*_{\GL_{d+e}}(\GL_{d+e}\times_{\GL_{d,e}}Q_{d,e})
\xrightarrow{\rho_*}H^*_{\GL_{d+e}}(Q_{d+e}). 
$$
We compute the pushforward $\rho_*$ after localizing $T_{d+e}$-equivariant cohomology groups. For any $H^*_T$-module $M$, we denote its localization by $M_\loc:=M\otimes_{H^*_T} \textnormal{Frac}(H^*_T)$. Note the following commutative diagram 
\begin{center}
    \begin{tikzcd}
\GL_{d+e}\times_{\GL_{d,e}}Q_{d,e} \arrow[r, "\tilde \rho"]                 & Q_{d+e}                 \\
C^{d+e}_I \arrow[r, "f_I", "\simeq"'] \arrow[u, hook, "\iota_I"] & C^{d+e}, \arrow[u, hook, "\iota"]
\end{tikzcd}
\end{center}
where
$$f_I(\overline{(\sigma_I,(p_1,\dots,p_{d+e}))})=(p_{\sigma_I(1)},\dots, p_{\sigma_{I}(d+e)}).$$
The localization formula applied to both $\GL_{d+e}\times_{\GL_{d,e}}Q_{d,e}$ and $Q_{d+e}$ shows that the following diagram commutes
\begin{center}
    \begin{tikzcd}
H^*_{T_{d+e}}(\GL_{d+e}\times_{\GL_{d,e}}Q_{d,e})_\loc \arrow[r, "\rho_*"] \arrow[d, "(\iota^*_I)"] & H^*_{T_{d+e}}(Q_{d+e})_\loc \arrow[d, "\iota^*"] \\
\bigoplus\limits_{I\subseteq[d+e],\ |I|=d}\!\!\!\! H^*_{T_{d+e}}(C^{d+e}_I)_\loc \arrow[r, "\rho^\loc_*"]                     & H^*_{T_{d+e}}(C^{d+e})_\loc              
\end{tikzcd}
\end{center}
where 
$$\rho^\loc_*((x_I))=\sum_{I\subseteq[d+e],\ |I|=d} e_{T_{d+e}}(N_\iota)\cdot (f_I)_*\left(\frac{x_I}{e_{T_{d+e}}(N_{\iota_I})}\right). 
$$

After using the identification \eqref{eq: heinloth iso}, consider 
$$f(z_1,\dots, z_d)\otimes g(z_1,\dots, z_e)\in H^*_{\GL_d}(Q_d)\otimes H^*_{\GL_e}(Q_e). 
$$
If $I_0=\{1, \dots, d\}$, then 
$$\iota_{I_0}^* q^*(f\otimes g)= f(z_1,\dots, z_d)\cdot g(z_{d+1},\dots, z_{d+e})\in H^*_{T_{d+e}}(C^{d+e}_{I_0})=H^*(C)^{\otimes d+e}[z_1,\dots, z_{d+e}]. 
$$
For general $I\subseteq[d+e]$ of size $d$, denote by $\sigma_I^z$ the automorphism of $H^*_{T_{d+e}}(C^{d+e}_{I_0})=H^*(C)^{\otimes d+e}[z_1,\dots, z_{d+e}]$ permuting the variables $z_i\mapsto z_{\sigma_I(i)}$ while keeping the cohomological coefficients unchanged. Then Proposition \ref{prop: normalbundleext} implies that $e(N_{\iota_I})=\sigma_I^z(e(N_{\iota_{I_0}}))$. Similarly, we can use diagram \eqref{eq: useful diagram} to show that 
$$\iota_{I}^* q^*(f\otimes g)= \sigma^z_I(f(z_1,\dots, z_d)\cdot g(z_{d+1},\dots, z_{d+e})). 
$$

Combining the previous discussions, we can compute the Hall algebra multiplication after the localization as follows:
\begin{align*}
    f\star g
    &=\sum_{I\subseteq[d+e],\ |I|=d} e_{T_{d+e}}(N_\iota)\cdot (f_I)_*\,\sigma_I^z\left(\frac{f(z_1,\dots, z_d)\cdot g(z_{d+1},\dots, z_{d+e})}{e_{T_{d+e}}(N_{\iota_{I_0}})}\right)\\
    &=\sum_{I\subseteq[d+e],\ |I|=d} e_{T_{d+e}}(N_\iota)\cdot \sigma_I\left(\frac{f(z_1,\dots, z_d)\cdot g(z_{d+1},\dots, z_{d+e})}{e_{T_{d+e}}(N_{\iota_{I_0}})}\right)\\
    &=\sum_{\sigma\in \Sh_{d,e}} \sigma\left(\frac{e_{T_{d+e}}(N_\iota)}{e_{T_{d+e}}(N_{\iota_{I_0}})}\cdot f(z_1,\dots, z_d)\cdot g(z_{d+1},\dots, z_{d+e})\right)\\
    &=\sum_{\sigma\in \Sh_{d,e}} \sigma\Big(K_{d,e}(z_1,\dots, z_{d+e})\cdot f(z_1,\dots, z_d)\cdot g(z_{d+1},\dots, z_{d+e})\Big).
\end{align*}
The second equality holds because $f_I$ and $\sigma^z_I$ permute the cohomological coefficients and $z_i$ variables, respectively, according to the shuffle $\sigma_I\in \Sh_{d,e}$. The third equality holds because $e_{T_{d+e}}(N_\iota)$ is $S_{d+e}$-invariant. The last equality follows directly from the description of normal bundles in Proposition \ref{prop: normalofqd} and \ref{prop: normalbundleext}. 

We have shown that the Hall algebra multiplication equals the multiplication of the shuffle algebra of the curve in Definition \ref{defnofshuffleproduct}, after the localization. However, the map induced by the localization is injective because the composition
$$H^*_{\GL_{d+e}}(Q_{d+e})=H^*_{T_{d+e}}(Q_{d+e})^{S_{d+e}}\xrightarrow{\iota^*} H^*_{T_{d+e}}(C^{d+e})^{S_{d+e}}\rightarrow H^*_{T_{d+e}}(C^{d+e})^{S_{d+e}}_\loc
$$
is injective since $H^*(C)^{\otimes (d+e)}[z_1,\dots, z_{d+e}]$ is a torsion-free $H^*_{T_{d+e}}$-module. Therefore, the Hall algebra product and shuffle product of curve match before localization as well. \qed
\medskip

\subsection{Presentation of the torsion CoHA}

In this subsection, we use the comparison isomorphism $\BH\simeq \Sh^{\Delta}(H^*(C)[z])$ to prove several structural results for the torsion CoHA including spherical generation and a presentation by generators and relations. In order to do so, we introduce a degree filtration. 

\begin{definition} \label{defn:filtration}A degree filtration $F^\bullet$ on the torsion CoHA $\BH$ is the increasing filtration
\[ 0=F^{-1} \subseteq F^0\subseteq \cdots \subseteq  F^i \subseteq \cdots \subseteq \mathbb{H}\] 
where $F^i$ is spanned by $f \in (H^{*}(C)^{\otimes d}[z_1,\cdots,z_d])^{S_d}$ of total degree $\leq i$ in variables $z_1, \cdots,z_d$ for all $d\geq 0$.

\end{definition}

The degree filtration is multiplicative with respect to the Hall algebra product. 

\begin{proposition}
For any $l, m \geq 0$, we have $F^l \star F^m \subseteq F^{l+m}$. 
\end{proposition}

\begin{proof}
Given polynomials $f$ and $g$ of total degree $l$ and $m$, the shuffle product $f \star g$ from \eqref{shufflesymmetrization} has degree at most $l+m$ since the shuffle kernel is of the form $\prod (1+ \Delta_{ij}/(z_j-z_i))$. 
\end{proof}

Thanks to the multiplicativity of the degree filtration, the associated graded of the torsion CoHA
\[\mathrm{Gr}_F(\mathbb{H}) := \bigoplus_{i \geq 0 } F^{i}(\mathbb{H})/F^{i-1}(\mathbb{H})\] 
is endowed with an algebra structure. 

\begin{proposition} \label{prop: associatedgraded}
There exists an isomorphism $\mathrm{Gr}_F(\mathbb{H})\simeq \Sh_{*}(H^*(C)[z])$ of algebras. 
\end{proposition}

\begin{proof}
Given polynomials $f$ and $g$ of total degree $l$ and $m$, the shuffle product $f \star g$ from \eqref{shufflesymmetrization} is equal to the usual shuffle product together with other terms of total degree strictly less than $l+m$. Therefore, the associated graded of the shuffle product is equal to the usual shuffle product with trivial kernel.
\end{proof}

It is a natural question to ask whether a given cohomological Hall algebra is spherically generated, i.e., generated by a subspace corresponding to minimal topological types. In our case, we consider $\BH_1=H^*(C)[z].$ For every $\alpha \in H^*(C)$ and $i \geq 0$, we denote $e^{\alpha}_i = \alpha z^i \in \BH_1$.

\begin{proposition} \label{prop: sphericalgeneration}
The algebra $\BH$ is spherically generated, i.e., the elements $e^{\alpha}_i$ where $\alpha \in H^*(C)$ and $i \geq 0$ generate $\BH$ as an associative algebra. 
\end{proposition}

\begin{proof} Let $\BH^{\textnormal{sph}}\subseteq \BH$ be a subalgebra generated by $e_i^\alpha\in F_i\BH$ for all $\alpha\in H^*(C)$ and $i\geq 0$. Since $F^{-1}\BH=0$ and $\cup_{n\geq 0} F^n \BH=\BH$, we can prove $\BH^{\textnormal{sph}}=\BH$ by induction. Assume that $F^{n-1}\BH^{\textnormal{sph}}= F^{n-1}\BH$ and consider a map between short exact sequences
\begin{center}
    \begin{tikzcd}
0 \arrow[r] & F^{n-1}\BH^{\textnormal{sph}} \arrow[r] \arrow[d, equal] & F^n\BH^{\textnormal{sph}} \arrow[r] \arrow[d] & \Gr_F^n\BH^{\textnormal{sph}} \arrow[r] \arrow[d] & 0 \\
0 \arrow[r] & F^{n-1}\BH \arrow[r]                                        & F^{n}\BH \arrow[r]                            & \Gr_F^{n}\BH \arrow[r]                            & 0.
\end{tikzcd}
\end{center}
By Proposition \ref{prop: associatedgraded}, $\Gr_F\BH$ is isomorphic to the shuffle algebra with trivial kernel, which is in turn isomorphic to the coinvariant ring by Lemma \ref{lem: isomorphismofshufflealgebraandcoalgebra}. Since the coinvariant ring is spherically generated by definition, this implies that the last vertical map in the above diagram is surjective. By using the long exact sequence, this implies that the middle vertical map is also surjective, hence completing the proof.

\end{proof}

\begin{notation}
Given any $\alpha \in  H^*(C)$ and any $i,j \geq 0$, we define
 \[ e_i \star^{\alpha} e_j:= \sum e^{\alpha^{(1)}}_i \star e^{\alpha^{(2)}}_j \in \BH_2\] 
where $\Delta_{*}(\alpha) = \sum \alpha^{(1)} \otimes \alpha^{(2)}$ using Sweedler's notation. Similarly, we define 
 \[ e_i \otimes^{\alpha} e_j:= \sum e^{\alpha^{(1)}}_i \otimes e^{\alpha^{(2)}}_j\in \BH_1\otimes \BH_1. \] 
\end{notation}

\begin{theorem}\label{relations:torsioncoha}
$\BH$ admits an algebra presentation by generators $\{e^\alpha_i\,|\, \alpha\in H^*(C),\ i\geq 0\}$ and relations
\begin{enumerate}
    \item [(i)] $e^{\alpha+\beta}_i=e^\alpha_i+e^\beta_i$,
    \item [(ii)] $[e^{\alpha}_i, e^{\beta}_j ] := e^{\alpha}_i \star e^{\beta}_j - (-1)^{|\alpha| |\beta|} e^{\beta}_j \star e^{\alpha}_i =  
\begin{cases}
-\sum_{k=1}^{i-j} e_{i-k} \star^{\alpha \cup \beta} e_{j+k-1}, & \textnormal{if}\quad  i\geq j, \\[\verticaldistance] \sum_{k=1}^{j-i} e_{j-k} \star^{\alpha \cup \beta} e_{i+k-1} , & \textnormal{if}\quad  i <j,    
\end{cases} $
\end{enumerate}
for each $\alpha, \beta\in H^*(C)$ and $i, j\geq 0$.
\end{theorem}

\begin{proof}
Let $A_C$ be the associative algebra generated by $e^{\alpha}_i$ with the relations (i) and (ii) as given in the statement. In order to construct an algebra homomorphism
\[ \Phi: A_C \rightarrow \BH,\quad e^\alpha_i\mapsto \alpha z^i,\] 
we check that the relations (i) and (ii) are satisfied by $\BH$. The relation (i) holds trivially. To check the relation (ii), assume $i \geq j$. By Theorem \ref{thm: CoHA equals shuffle}, we have
\begin{align*} e^{\alpha}_{i} \star e^{\beta}_j &= (\alpha \otimes \beta) z_1^i z_2^j \left(1+\frac{\Delta_{12}}{z_2-z_1} \right) + (-1)^{|\alpha||\beta|}(\beta \otimes \alpha) z_1^j z_2^i \left(1+\frac{\Delta_{12}}{z_1-z_2} \right)\\
&=(\alpha \otimes \beta) z_1^i z_2^j+(-1)^{|\alpha||\beta|}(\beta\otimes\alpha)z_1^jz_2^i+\Delta_*(\alpha\cup\beta)\left( \frac{z_1^iz_2^j-z_1^jz_2^i}{z_2-z_1}\right).
\end{align*}
Taking the supercommutator, we obtain
\begin{equation}\label{eq: commutator explicit}
    [e^{\alpha}_{i}, e^{\beta}_j]= 2 \Delta_*(\alpha\cup \beta) \left(\frac{z_1^iz_2^j-z_1^jz_2^i}{z_2-z_1} \right)=-2\Delta_*(\alpha\cup \beta)\sum_{k=1}^{i-j} z_1^{i-k} z_2^{j+k-1}.
\end{equation}
On the other hand, 
$$e_{i-k}\star^{\alpha\cup \beta}e_{j+k-1}=\Delta_*(\alpha\cup\beta)\left(z_1^{i-k}z_2^{j+k-1}+z_1^{j+k-1}z_2^{i-k}\right)+\Delta\cdot\Delta_*(\alpha\cup\beta)\left(\frac{z_1^{i-k}z_2^{j+k-1}-z_1^{j+k-1}z_2^{i-k}}{z_2-z_1}
\right).
$$
Summing this expression over $k=1, \dots, i-j$, the second term vanishes and the first term matches precisely with \eqref{eq: commutator explicit}. The case when $i<j$ is similar. 

Therefore, the algebra homomorphism $\Phi$ is well-defined. Note that Proposition \ref{prop: sphericalgeneration} says that $\Phi$ is surjective. Consider the smallest multiplicative filtration $\mathcal{F}_\bullet$ on $A_C$ such that $e^{\alpha}_i\in \CF_i$. Furthermore, the relation (ii) becomes the supercommutativity relation for the associated graded algebra 
\[ \mathrm{Gr}_{\mathcal{F}}(A_C) = \bigoplus_{i \geq 0} \mathcal{F}^{i}(A_C)/\mathcal{F}^{i-1}(A_C). \] 
By definition of the filtrations on both sides of $\Phi$, $\Phi$ is a homomorphism of filtered algebras. The induced homomorphism on the associated graded algebras
\[ \mathrm{Gr}(\Phi) \colon \mathrm{Gr}_{\mathcal{F}}(A_C) \rightarrow \mathrm{Gr}_F(\mathbb{H}) \]
is an isomorphism by Lemma \ref{lem: isomorphismofshufflealgebraandcoalgebra} and Proposition \ref{prop: associatedgraded}. Since $\Gr(\Phi)$ is an isomorphism, $\Phi$ is also an injection and thus an isomorphism of algebras. 
\end{proof}

\subsection{Yang--Baxter equation and braided symmetric algebra} 

In this section, we prove that the torsion cohomological Hall algebra $\BH$ is isomorphic to the braided symmetric algebra of $\BH_1=H^*(C)[z]$ with respect to a certain choice of the braiding operator.

We briefly recall the Yang--Baxter equation. Let $V$ be any vector space and $R\in \End(V\otimes V)$ be an operator acting on $V\otimes V$. We say that $R$ satisfies the Yang--Baxter equation if 
\begin{align}\label{equation: Yang--Baxter}
 R_{12}R_{23}R_{12}=R_{23}R_{12}R_{23}\in \End(V\otimes V\otimes V)
\end{align}
where $R_{ij}$ applies the operator $R$ to the corresponding factors in $V\otimes V\otimes V$. If $R$ satisfies the Yang--Baxter equation, then it induces an action of the braid group $B_d$ on $T_d(V)=\underbrace{V \otimes \cdots \otimes V}_{d \ \mathrm{times}}$. If $R$ additionally satisfies that $R^2=\id_{V\otimes V}$, then the action of the braid group $B_d$ on $T_d(V)$ factors through the symmetric group $B_d\twoheadrightarrow S_d$. 

\begin{example}\label{ex: YB classical}
Let $V=\BQ[z]$ and identify $V\otimes V\simeq \BQ[z_1,z_2]$. Define the swapping operator
$$\sigma:\BQ[z_1,z_2]\rightarrow \BQ[z_1,z_2],\quad \sigma(f(z_1,z_2)):=f(z_2,z_1)
$$
and the divided difference operator
$$\partial:\BQ[z_1,z_2]\rightarrow \BQ[z_1,z_2],\quad \partial(f(z_1,z_2)):=\frac{f(z_1,z_2)-f(z_2,z_1)}{z_1-z_2}. 
$$
Combining these operators, we define the operator
$$R:=\sigma-\partial:\BQ[z_1,z_2]\rightarrow \BQ[z_1,z_2],\quad R(f(z_1,z_2)):=f(z_2,z_1)-\frac{f(z_1,z_2)-f(z_2,z_1)}{z_1-z_2}. 
$$
On the level of basis of $\BQ[z_1,z_2]$, we have 
$$R(z_1^iz_2^j)=\begin{cases}
z_1^jz_2^i-\sum_{k=1}^{i-j}z_1^{i-k}z_2^{j+k-1},&\quad\textnormal{if}\ i\geq j,\\[\verticaldistance]
z_1^jz_2^i+\sum_{k=1}^{j-i}z_1^{j-k}z_2^{i+k-1},&\quad\textnormal{if}\ i<j.\\
\end{cases}
$$
It is a classical result that the operator $R$ satisfies the Yang--Baxter equation and $R^2=\id$, see for example \cite{LLT} and references therein. In fact, $R^2=\id$ follows from the identities 
\begin{equation}\label{eq: square identity}
    \sigma^2=\id,\quad \sigma\partial+\partial \sigma=0,\quad \partial^2=0,
\end{equation}
and the Yang--Baxter equation follows from the identities
\begin{enumerate}
    \item [YB-0)] $\sigma_{12}\sigma_{23}\sigma_{12}=\sigma_{23}\sigma_{12}\sigma_{23},$
    \item [YB-1)] $\sigma_{12}\sigma_{23}\partial_{12}=\partial_{23}\sigma_{12}\sigma_{23}$,\quad $\sigma_{12}\partial_{23}\sigma_{12}=\sigma_{23}\partial_{12}\sigma_{23}$,\quad  $\partial_{12}\sigma_{23}\partial_{12}=\partial_{23}\sigma_{12}\partial_{23}$,
    \item [YB-2)] $\sigma_{12}\partial_{23}\partial_{12}+\partial_{12}\sigma_{23}\partial_{12}+\partial_{12}\partial_{23}\sigma_{12}
=\sigma_{23}\partial_{12}\partial_{23}+\partial_{23}\sigma_{12}\partial_{23}+\partial_{23}\partial_{12}\sigma_{23}$,
    \item [YB-3)] $\partial_{12}\partial_{23}\partial_{12}=\partial_{23}\partial_{12}\partial_{23}$.
\end{enumerate}

\end{example}

Motivated by the above example, we define the analogous operator when the underlying vector space $\BQ[z]$ is replaced by $\BH_1=H^*(C)[z]$. 

\begin{definition}\label{def: curve YB operator}
Define a linear map
\[R:  \mathbb{H}_1 \otimes \mathbb{H}_1 \rightarrow  \mathbb{H}_1\otimes \mathbb{H}_1\] 
such that
$$
R(e^{\alpha}_i \otimes e^{\beta}_j)=
\begin{cases}
(-1)^{|\alpha||\beta|} e^{\beta}_j \otimes e^{\alpha}_i - \sum_{k=1}^{i-j} e_{i-k} \otimes^{\alpha \cup \beta} e_{j+k-1},&\quad\textnormal{if}\ i\geq j, \\[\verticaldistance]
 (-1)^{|\alpha||\beta|} e^{\beta}_j \otimes e^{\alpha}_i + \sum_{k=1}^{j-i} e_{j-k} \otimes^{\alpha \cup \beta} e_{i+k-1},&\quad\textnormal{if}\ i<j.\\
\end{cases}
$$
\end{definition}

\begin{proposition} \label{prop:twist}
The operator $R$ in Definition \ref{def: curve YB operator}
satisfies that
$$R_{12}R_{23}R_{12}=R_{23}R_{12}R_{23},\quad R^2=\id.
$$
\end{proposition}

\begin{proof}
Recall the operators 
$$\sigma,\ \partial\in \End(\BQ[z_1,z_2])$$ 
from Example \ref{ex: YB classical}. By identifying $\BH_1\otimes \BH_1\simeq H^*(C)^{\otimes 2}[z_1,z_2]$, we may regard $\sigma$ and $\partial$ as elements in $\End(\BH_1\otimes \BH_1)$ which keeps the cohomological coefficients from $H^*(C)^{\otimes 2}$ unchanged. On the other hand, we define 
$$\sigma^C,\ \Delta^C\in \End(H^*(C)^{\otimes 2})
$$
such that 
$$\sigma^C(\alpha\otimes \beta):=(-1)^{|\alpha||\beta|}\beta\otimes \alpha,\quad \Delta^C(\alpha\otimes \beta):=\Delta\cup(\alpha\otimes \beta)=\Delta_*(\alpha\cup \beta). 
$$
Similarly, we may regard $\sigma^C$ and $\Delta^C$ as an element in $\End(\BH_1\otimes \BH_1)$ which keeps the variables $z_1$ and $z_2$ unchanged. 

The operator $R$ in Definition \ref{def: curve YB operator} can be written as
$$R=\sigma^C\sigma-\Delta^C\partial\in \End(\BH_1\otimes \BH_1).
$$
We first check that $R^2=\id$. Note that
\begin{align*}
    R^2
    &=(\sigma^C\sigma-\Delta^C\partial)^2\\
    &=(\sigma^C)^2\sigma^2-\sigma^C\Delta^C\sigma\partial-\Delta^C\sigma^C\partial\sigma+(\Delta^C)^2\partial^2
\end{align*}
because the operators in $\{\sigma,\partial\}$ commute with the operators in $\{\sigma^C, \Delta^C\}$. The identities in \eqref{eq: square identity} together with basic identities in the cohomology ring
\begin{equation}\label{eq: diagonal identities}
    (\sigma^C)^2=\id,\quad \sigma^C\Delta^C=\Delta^C\sigma^C=\Delta^C,
\end{equation} 
imply that $R^2=\id$. 

Now we prove the Yang--Baxter equation. Again, using the fact that the operators in $\{\sigma,\partial\}$ commutes with the operators in $\{\sigma^C, \Delta^C\}$, we can write
\begin{align*}
    R_{12}R_{23}R_{12}=\ &\sigma^C_{12}\sigma_{23}^C\sigma_{12}^C\sigma_{12}\sigma_{23}\sigma_{12}\\
    &-\Delta^C_{12}\sigma_{23}^C\sigma_{12}^C\partial_{12}\sigma_{23}\sigma_{12}
    -\sigma^C_{12}\Delta_{23}^C\sigma_{12}^C\sigma_{12}\partial_{23}\sigma_{12}
    -\sigma^C_{12}\sigma_{23}^C\Delta_{12}^C\sigma_{12}\sigma_{23}\partial_{12}\\
    &+\sigma^C_{12}\Delta_{23}^C\Delta_{12}^C\sigma_{12}\partial_{23}\partial_{12}
    +\Delta^C_{12}\sigma_{23}^C\Delta_{12}^C\partial_{12}\sigma_{23}\partial_{12}
    +\Delta^C_{12}\Delta_{23}^C\sigma_{12}^C\partial_{12}\partial_{23}\sigma_{12}\\
    &-\Delta^C_{12}\Delta_{23}^C\Delta_{12}^C\partial_{12}\partial_{23}\partial_{12},
\end{align*}
and similarly for $R_{23}R_{12}R_{23}$. Then $R_{12}R_{23}R_{12}=R_{23}R_{12}R_{23}$ follows from the identities YB-0, $\dots$, YB-3 from Example \ref{ex: YB classical}, together with the following identities for the operators acting on $H^*(C\times C\times C)$:
\begin{align*}
    \sigma^C_{12}\sigma^C_{23}\sigma^C_{12} &= \sigma^C_{23}\sigma^C_{12}\sigma^C_{23}, \\
    \sigma^C_{12}\sigma^C_{23}\Delta^C_{12} &=  \Delta^C_{23}\sigma^C_{12}\sigma^C_{23},\quad
    \sigma^C_{12}\Delta^C_{23}\sigma^C_{12}= \sigma^C_{23}\Delta^C_{12}\sigma^C_{23},\quad
    \Delta^C_{12}\sigma^C_{23}\sigma^C_{12} = \sigma^C_{23}\sigma^C_{12}\Delta^C_{23}, \\
    \sigma^C_{12}\Delta^C_{23}\Delta^C_{12} &= \Delta^C_{12}\sigma^C_{23}\Delta^C_{12} = \Delta^C_{12}\Delta^C_{23}\sigma^C_{12} = \sigma^C_{23}\Delta^C_{12}\Delta^C_{23} = \Delta^C_{23}\sigma^C_{12}\Delta^C_{23} = \Delta^C_{23}\Delta^C_{12}\sigma^C_{23}, \\
    \Delta^C_{12}\Delta^C_{23}\Delta^C_{12} &= \Delta^C_{23}\Delta^C_{12}\Delta^C_{23}.
\end{align*}
Note that these are consequences of the following basic facts:
$$\sigma^C_{ij}\Delta^C_{ab}=\Delta^C_{a'b'}\sigma^C_{ij}
$$
where $\{a',b'\}$ is the image of $\{a,b\}$ under the transposition $(ij)$, and 
$$\Delta^C_{12}\Delta^C_{13}=\Delta^C_{12}\Delta^C_{23}=\Delta^C_{13}\Delta^C_{23}=\Delta^C_{123}.
$$
\end{proof}

\begin{definition}
Let $R\in \End(V\otimes V)$ be an operator satisfying the Yang--Baxter equation and $R^2=\id$. We define the $R$-braided symmetric algebra of $V$ as 
$$\Sym^R(V):=T_*(V)/\langle v_1\otimes v_2=R(v_1\otimes v_2)\rangle. 
$$
\end{definition}

We prove that the torsion CoHA is a braided symmetric algebra. 

\begin{corollary}\label{cor: YB-twisted symmetric}
    There is an isomorphism of algebras 
\[ \mathbb{H} \simeq \Sym^{R}(H^{*}(C)[z])   \]
where $R$ is the operator in Definition \ref{def: curve YB operator}. 
\end{corollary}

\begin{proof}
This is essentially a rephrasing of Theorem \ref{relations:torsioncoha} which characterizes $\BH$ by generators and relations. By definition of the operator $R$ and Theorem \ref{relations:torsioncoha}, we have 
$$e_i^\alpha \star e^\beta_j= \star\circ R(e^\alpha_i\otimes e^\beta_j)\in \BH_2.
$$
Furthermore, Theorem \ref{relations:torsioncoha} says that this is the complete set of relations up to $H^*(C)$-linearity. By the universal property of the $R$-braided symmetric algebra, we obtain an isomorphism of algebras
$$\Sym^R(H^*(C)[z])\xrightarrow{\simeq} \BH,\quad \alpha z^i\mapsto e_i^\alpha. 
$$
\end{proof}

\subsection{CoHA of $\mathbb{P}^{1}$ and Kronecker quiver}\label{sec: Kronecker}

In this section, we prove that the torsion CoHA of $\BP^1$ is isomorphic to a certain semistable CoHA of the Kronecker quiver. Upon this identification, our Proposition \ref{prop:twist} about Yang--Baxter equation in the case of $\BP^1$ proves a conjecture of Franzen--Reineke \cite[Conj. 5]{franzen2019cohomological}. 

By Beilinson's theorem \cite{beilinson}, the bounded derived category $\Db(\Coh(\PP^1))$ of coherent sheaves on $\BP^1$ admits a strong full exceptional collection $(\CO,\ \CO(1))$. This induces an equivalence of categories
\begin{equation}\label{eq: Beilinson iso}
\Phi: \Db(\Coh(\PP^1)) \simeq \Db(\Rep(K)),
\end{equation}
where $\Db(\Rep(K))$ is the bounded derived category of finite dimensional representations of the Kronecker quiver $K$ consisting of two vertices and two arrows as below:
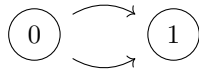
\begin{center}\vspace{5pt}
    \begin{tikzcd}
\tikz \node[draw, circle]{$0$}; \arrow[r, bend left] \arrow[r, bend right]   & \tikz \node[draw, circle]{1};
\end{tikzcd}\vspace{5pt}
\end{center}
Precisely, the equivalence $\Phi$ is defined as
$$\Phi(F)=\begin{tikzcd}
	{\RHom(\CO(1),F)} & {\RHom(\CO, F)}
	\arrow["x_1", curve={height=-15pt}, from=1-1, to=1-2]
	\arrow["x_2"', curve={height=+15pt}, from=1-1, to=1-2]
\end{tikzcd}
$$
where $x_1,x_2$ are global sections of $\mathcal{O}(1)$.

In order to study the torsion CoHA of $\BP^1$ via Kronecker quiver $K$, we need to understand the image of torsion coherent sheaves under the derived equivalence \eqref{eq: Beilinson iso}. To state the result and set up notations, we recall some basic notions in quiver representations. Let $Q=(Q_0,Q_1)$ be any quiver. For any dimension vector $\bd=(d_i)_{i\in Q_0} \in \mathbb{N}^{Q_0}$, let $\mathfrak{M}_{\bd}(Q)$ denote the stack of $\bd$-dimensional representations of $Q$. Let $\theta: \mathbb{Z}^{Q_0} \rightarrow \mathbb{Z}$ be a group homomorphism, called the stability condition. For any $\bd$-dimensional representation $\rho$ with $\bd\neq 0$, the slope of $\rho$ with respect to $\theta$ is defined as
\[ \mu^\theta(\bd) = \frac{\theta(\bd)}{\sum_{i\in Q_0} d_i}.\]
We say that a representation $\rho$ is $\theta$-semistable if for every subrepresentation $0\subsetneq \rho^{\prime}\subsetneq \rho$, we have $\mu^\theta(\rho^\prime) \leq \mu^\theta(\rho)$, and it is called stable if the inequality is strict. Let $\mathfrak{M}^{\theta-\mathrm{ss}}_{\bd}(Q) \subseteq \mathfrak{M}_{\bd}(Q)$ be the open substack of $\theta$-semistable $\bd$-dimensional representations of the quiver $Q$.

The following result is well known to experts\footnote{Timm Peerenboom has informed us that it is shown in \cite[Prop. 4.3]{GeigleLenzing} and \cite[Prop. 2.20]{Timm} that the derived equivalence restricts to an equivalence between torsion sheaves and regular representations. The notion of regular representation is a priori different from semi-stability but it can be shown that these two definitions coincide. } but we record the proof since we could not find a precise reference.  

\begin{proposition} \label{prop:belinsonisoofstacks} Let $\theta(d_0,d_1)=d_0$ be a stability condition of the Kronecker quiver $K$. The derived equivalence \eqref{eq: Beilinson iso} restricts to an equivalence between the abelian category of torsion sheaves on $\mathbb{P}^1$ and the abelian category of $\theta$-semistable, slope $1/2$ representations of the Kronecker quiver $K$. In particular, for any $d \geq 0$, there is an isomorphism of moduli stacks
   \[  \Coh_{d}(\mathbb{P}^1) \simeq \mathfrak{M}^{\theta-\mathrm{ss}}_{(d,d)}(K). \]
\end{proposition}

\begin{proof}

Let $\underline{\Coh}_d(\BP^1)$ be the subcategory of torsion sheaves of length $d$, and $\Rep^{\theta-\mathrm{ss}}_{(d,d)}(K)$ be the subcategory of $\theta$-semistable $(d,d)$-dimensional representations of $K$. We show that the derived equivalence $\Phi$ from \eqref{eq: Beilinson iso} restricts to an equivalence 
$$\Phi:\underline{\Coh}_d(\BP^1)\simeq \Rep^{\theta-\mathrm{ss}}_{(d,d)}(K). 
$$
First, we show that if $F\in \underline{\Coh}_d(\BP^1)$ then $\Phi(F)$ is a $\theta$-semistable $(d,d)$-dimensional representation of $K$ by induction on $d$. If $d=1$, then $F=\CO_x$ for some $x=[x_1:x_2]\in \BP^1$. Then $\RHom(\CO(1),F) = \RHom(\CO,F)= \mathbb{C}$ and $x_1,x_2$ are not both zero. Thus $\Phi(F)$ is $\theta$-semistable representations of dimension vector $(d,d)$. Now let $d>1$. Then any torsion sheaf $F$ of length $d>1$ on $\PP^1$ is an extension 
$$0\rightarrow F'\rightarrow F\rightarrow F''\rightarrow 0,
$$
where $F', F''$ are torsion sheaves of length less than $d$. Now applying $\Phi$ gives an exact triangle 
\[ \Phi(F') \rightarrow \Phi(F) \rightarrow \Phi(F'') \rightarrow \Phi(F')[1]. \]
By induction hypothesis, $\Phi(F')$ and $\Phi(F'')$ are both $\theta$-semistable representations of slope $1/2$. Since such representations are closed under exact triangles, this implies that $\Phi(F)$ is $\theta$-semistable. 

In the above, we proved that the derived equivalence restricts to a fully faithful functor
$$\Phi:\underline{\Coh}_d(\BP^1)\rightarrow \Rep^{\theta-\mathrm{ss}}_{(d,d)}(K). 
$$
We now show that this is essentially surjective. Note that a $(1,1)$-dimensional representation
\[\rho = \begin{tikzcd}
	{\mathbb{C}^1} & {\mathbb{C}^1}
	\arrow["b"', curve={height=12pt}, from=1-1, to=1-2]
	\arrow["a", curve={height=-12pt}, from=1-1, to=1-2]
\end{tikzcd}\]
is $\theta$-semistable if and only if $a$ and $b$ are not both zero. Such a representation is isomorphic to $\Phi(\CO_x)$ where $x=[a:b]\in \BP^1$. Now we prove the essential surjectivity for $d\geq 2$. By the same strategy as in the previous paragraph, it suffices to show that every $\theta$-semistable representation of the Kronecker quiver of dimension $(d,d)$ with $d\geq 2$ is strictly semistable, hence built from iterated extensions of stable representations of dimension $(1,1)$. Consider a $\theta$-semistable representation 
    \item \[\rho = \begin{tikzcd}
	{\mathbb{C}^d} & {\mathbb{C}^d}
	\arrow["b"', curve={height=12pt}, from=1-1, to=1-2]
	\arrow["a", curve={height=-12pt}, from=1-1, to=1-2]
\end{tikzcd}\]
with $d\geq 2$. We show that this is strictly semistable. We first consider the case where there exists a vector $0\neq v\in \BC^d$ such that $\Span(a(v))=\Span(b(v))$. In this case, $v$ spans a subrepresentation of dimension vector either $(1,1)$ or $(1,0)$. The second case of $(1,0)$ cannot occur because $\rho$ is semistable and the case with $(1,1)$ shows that $\rho$ is strictly semistable. Now we assume that there is no nonzero vector $v\in \BC^d$ such that $\Span(a(v))=\Span(b(v))$. This implies that $a-b:\BC^d\rightarrow \BC^d$ is injective, hence an isomorphism. Since we consider quiver representations up to conjugation, we may assume that $a-b=\textnormal{Id}_{\BC^d}$. If there is $v\neq 0$ such that $b(v)=0$, then $a(v)=v$ so $v$ spans a subrepresentation of dimension vector $(1,1)$. So we may assume that $b$ has no kernel hence is invertible. The same argument shows that we may assume $a$ is invertible. Since both $a,b$ are invertible, there exists an eigenvector $v$ for $a^{-1}\circ b$ with nonzero eigenvalue $\lambda$, i.e., $b(v)=\lambda a(v)$. But we have already considered the case where $\Span(a(v))=\Span(b(v))$. This completes the proof.

\end{proof}

Since we have an isomorphism of moduli stacks induced by an equivalence of abelian categories, this induces an isomorphism of cohomological Hall algebras. Denote the cohomological Hall algebra of $\theta$-semistable, slope $1/2$ representations of the Kronecker quiver $K$ by 
$$\BH^{\theta}_{1/2}(K):=\left(\bigoplus_{d=0}^\infty H^*(\mathfrak{M}^{\theta-\mathrm{ss}}_{(d,d)}(K))\,,\, \star\right).$$

\begin{corollary}\label{cor: Kronecker CoHA}
    The derived equivalence \eqref{eq: Beilinson iso} induces an isomorphism of cohomological Hall algebras
    \[\BH_{\PP^1}  \simeq  \BH^{\theta}_{1/2}(K). \]
\end{corollary}

\begin{remark}
    This precisely describes the semistable CoHA of the Kronecker quiver with slope $1/2$, which was previously computed in \cite{franzen2019cohomological} and \cite{Timm}. Since both sides of the isomorphism in Corollary \ref{cor: Kronecker CoHA} are spherically generated, the isomorphism is uniquely specified by the vector space isomorphism
    \begin{align*}
        H^*(\Coh_1(\BP^1))&\simeq H^*(\mathfrak{M}_{(1,1)}^{\theta-\ss}(K))\\
        z^i&\mapsto e_i\\
        [\pt]z^i&\mapsto f_{i+1}
    \end{align*}
    where $e_i$ and $f_i$ on the right hand side follow the notation used in loc. cit. Under this isomorphism, Proposition \ref{prop:twist} proves a conjecture of Franzen--Reineke \cite[Conj. 5]{franzen2019cohomological} and its generalized version for arbitrary smooth quasi-projective curves. 

\end{remark}

\section{Cohomology ring of punctual Quot schemes}
  
In this section, we determine the cohomology ring of punctual Quot schemes in terms of the torsion CoHA of curves, computed in Theorem \ref{thm: CoHA equals shuffle}. In particular, we obtain a natural basis of the cohomology of punctual Quot schemes. Our approach is motivated by a work of Franzen \cite{franzen} which corresponds to the $C=\BA^1$ case. For this section, we let $C$ be a smooth connected quasi-projective curve.

\subsection{Kernel of the cyclic module}

In this subsection, we characterize the cohomology of punctual Quot schemes purely in terms of the torsion CoHA. A key idea is to use a stratification by open loci of the total space of a certain tautological bundle on $\Coh_d$. 

Let $V$ be a fixed vector bundle of rank $N$ throughout the section. We write $\Quot_d:=\Quot_d(V)$ for the punctual Quot scheme and consider 
$$\BV:=\bigoplus_{d=0}^\infty \BV_d=\bigoplus_{d=0}^\infty H^*(\Quot_d)\simeq \bigoplus_{d=0}^\infty H^\BM_*(\Quot_d). 
$$
Define a tautological bundle on $\Coh_d$ as
$$V^{[d]}:=\pr_{1*}(\mathcal{H}om(\pr_2^*V,\CF))$$ 
where $\CF$ is a universal sheaf on $\Coh_d\times C$ and $\pr_1, \pr_2$ are the two projections from $\Coh_d\times C$. Note that the total space $\Tot(V^{[d]})$ parametrizes a pair $(F,\phi)$ of a length $d$ torsion sheaf $F$ and any morphism $\phi:V\rightarrow F$. Then $\Quot_d$ is an open subset of $\Tot(V^{[d]})$ where $\phi$ is surjective. More generally, we define the intermediate open subsets as below.

\begin{definition}
For each $0\leq k\leq d$, let $Q_d^{\geq k}\subseteq \Tot(V^{[d]})$ be the open locus parametrizing pairs $(F,\phi)$ where the image of $\phi$ is of length at least $k$.  
\end{definition}

These intermediate open loci interpolate the total space of a tautological bundle and the Quot scheme as follows
$$\Quot_d=Q_d^{\geq d}\subseteq Q_d^{\geq d-1}\subseteq \cdots\subseteq Q_d^{\geq 1}\subseteq Q_d^{\geq 0}=\Tot(V^{[d]}). 
$$
Let $\pi_d:\Quot_d\rightarrow \Coh_d$ be a forgetful map. For each $0\leq k<d$, the complement of the successive open subsets $Q^{\geq k+1}_d\subseteq Q^{\geq k}_d$, denoted by $Q^{=k}_d$, is defined as a fiber product
\begin{equation*}
    \begin{tikzcd}
Q^{=k}_d \arrow[d, "\pi_{k, d-k}"'] \arrow[r, "\overline{q}"] & \Quot_k\times \Coh_{d-k} \arrow[d, "\pi_k\times \id"] \\
{\Coh_{k, d-k}} \arrow[r, "q"]    & \Coh_k\times \Coh_{d-k}.       
\end{tikzcd}
\end{equation*}
By definition, the stack $Q^{=k}_d$ is smooth of dimension $N\cdot k$ and parametrizes a pair consisting of a short exact sequence and a quotient
\begin{equation}\label{eq: Q=k}
    \begin{tikzcd}
            & V \arrow[d, two heads, "\phi'"] &             &               &   \\
0 \arrow[r] & F' \arrow[r]           & F \arrow[r] & F'' \arrow[r] & 0
\end{tikzcd}
\end{equation}
where $F'$ and $F''$ are torsion sheaves of length $k$ and $d-k$, respectively. There exists a natural morphism 
$$\iota:Q^{=k}_d\rightarrow Q_{d}^{\geq k}
$$
by sending the data \eqref{eq: Q=k} to the composition $\phi:V\xrightarrow{\phi'}F'\rightarrow F$. The morphism $\iota$ is a monorphism whose image is the complementary closed locus $Q_d^{\geq k}\backslash Q_d^{\geq k+1}$, hence explaining the notation $Q^{=k}_d$.

In order to understand the morphism $\iota$ better, we consider a diagram 
\begin{equation}\label{eq: key diagram}
    \begin{tikzcd}
Q^{=k}_d \arrow[rd, "\widetilde\iota"] \arrow[rdd, bend right, "\pi_{k, d-k}"'] \arrow[rrd, bend left, "\iota"] &                                         &                        \\
                                                                   & \widetilde Q^{=k}_d \arrow[d, "\overline{\pi}_{d}^{\geq k}"'] \arrow[r, "\overline{p}"] & Q^{\geq k}_d \arrow[d, "\pi_{d}^{\geq k}"] \\
                                                                   & {\Coh_{k, d-k}} \arrow[r, "p"]               & \Coh_d                
\end{tikzcd}
\end{equation}
where $\widetilde Q^{=k}_d$ is the fiber product of a square. By definition, $\widetilde Q^{=k}_d$ parametrizes 
\begin{equation}\label{eq: point on tilde Q}
    \begin{tikzcd}
            &  &     V \arrow[d, "\phi"]        &               &   \\
0 \arrow[r] & F' \arrow[r]           & F \arrow[r] & F'' \arrow[r] & 0
\end{tikzcd}
\end{equation}
where $F'$, $F''$ are torsion sheaves of length $k$ and $d-k$, respectively, and $\phi$ is a morphism whose image is of length at least $k$. Let us denote the pullback of the tautological bundle $V^{[d-k]}$ via the forgetful morphism $\widetilde{Q}^{=k}_d\rightarrow \Coh_{d-k}$ by the same notation. There exists a tautological section 
\begin{equation}\label{eq: tautological section}
    \tau\in H^0(\widetilde Q^{=k}_d, V^{[d-k]})
\end{equation}
given by the composition $[V\xrightarrow{\phi}F\rightarrow F'']\in \Hom(V,F'')$ over the point corresponding to \eqref{eq: point on tilde Q}.

\begin{lemma}\label{lem: tautological zero}
In diagram \eqref{eq: key diagram}, $\widetilde \iota$ is a closed embedding onto the zero locus of the tautological section $\tau$ from \eqref{eq: tautological section}. Furthermore, $\iota$ in the same diagram is a regular embedding. 

\end{lemma}

\begin{proof}

Over the zero locus of the tautological section $\tau$, the morphism $\phi$ from \eqref{eq: point on tilde Q} factors through the subsheaf $F'$, i.e., 
\begin{equation}\label{eq: point on zero locus of tilde Q}
    \begin{tikzcd}
            &  &     V \arrow[d, "\phi"] \arrow[ld, dashed, "\phi'"']        &               &   \\
0 \arrow[r] & F' \arrow[r]           & F \arrow[r] & F'' \arrow[r] & 0.
\end{tikzcd}
\end{equation}
Since the image of $\phi$ has length at least $k$ and $F'$ is of length $k$, the factoring morphism $\phi'$ has to be surjective. Therefore, the zero locus of the tautological section identifies with $Q^{=k}_{d}$ via $\iota$. 

Note that $\widetilde \iota$ is proper as a closed embedding and $\overline{p}$ is a pullback of a proper morphism $p$. Therefore, the composition $\iota=\overline{p}\circ \widetilde{\iota}\,$ is also proper. Since $\iota$ is a proper monomorphism, it is a closed embedding which is also regular since both $Q^{=k}_d$ and $Q^{\geq k}_d$ are smooth stacks.

\end{proof}

The regular embedding $\iota$ of codimension $N\cdot (d-k)$ and a complementary open embedding $j$
$$Q^{=k}_d \overset{\iota}{\hookrightarrow} Q^{\geq k}_d \overset{j}{\supseteq} Q^{\geq k+1}_d
$$
induce a Gysin long exact sequence
\begin{equation*}
    \cdots \rightarrow H^{*-2N(d-k)}(Q^{=k}_d)\xrightarrow{\iota_*} H^*(Q^{\geq k}_d)\rightarrow H^*(Q^{\geq k+1}_d)\rightarrow \cdots. 
\end{equation*}

\begin{lemma}\label{lem: Gysin splitting}
The above Gysin long exact sequence splits into short exact sequences
$$0\rightarrow H^{*-2N(d-k)}(Q^{=k}_d)\xrightarrow{\iota_*} H^*(Q^{\geq k}_d)\rightarrow H^*(Q^{\geq k+1}_d)\rightarrow 0.$$
\end{lemma}

\begin{proof}
    The splitting holds if $\iota_*$ is injective. We show injectivity by proving that the composition $\iota^*\iota_* = e(N_\iota)$ is injective. We use the factorization $\iota=\overline{p}\circ \widetilde \iota$ to describe the normal bundle $N_\iota$. Since $\iota$ and $\widetilde\iota$ are regular embeddings, we have 
    $$\mathbb{T}_{\widetilde{\iota}}\simeq N_{\widetilde{\iota}}\,[-1],\quad \mathbb{T}_{\iota}\simeq N_{\iota}[-1]. 
    $$
    On the other hand, since $\overline{p}$ is the pullback of $p:\Coh_{k,d-k}\rightarrow \Coh_d$ via flat morphism $\pi_d^{\geq k}:Q^{\geq k}_d\rightarrow \Coh_d$, we have 
$$\widetilde\iota^*\mathbb{T}_{\overline{p}}\simeq \pi_{k, d-k}^*\mathbb{T}_p. 
    $$
    Therefore, the exact triangle of tangent complexes induced from the factorization $\iota=\overline{p}\circ \widetilde \iota$ reads
    \begin{equation}\label{eq: exact triangle}
        \pi_{k, d-k}^*\mathbb{T}_p\rightarrow N_{\widetilde \iota}\rightarrow N_\iota\xrightarrow{[1]}. 
    \end{equation}
    Denote by $\pr_1, \pr_2$ the two projections from $Q^{=k}_d\times C$ and the universal objects by 
\begin{equation*}
    \begin{tikzcd}
            & \pr_2^*V \arrow[d, two heads, "\Phi'"] &             &               &   \\
0 \arrow[r] & \CF' \arrow[r]           & \CF \arrow[r] & \CF'' \arrow[r] & 0.
\end{tikzcd}
\end{equation*}
Since $p$ is the relative Quot scheme, its tangent complex can be described by the relative RHom complex from the universal subsheaf to the universal quotient. On the other hand, the normal bundle of $\widetilde\iota$ is the restriction of the tautological bundle $V^{[d-k]}$ by Lemma \ref{lem: tautological zero}. Therefore, the exact triangle \eqref{eq: exact triangle} can be identified with
$$\R\pr_{1*}\R\mathcal{H}om(\CF',\CF'')\rightarrow
\R\pr_{1*}\R\mathcal{H}om(\pr_2^*V,\CF'')\rightarrow N_\iota\xrightarrow{[1]}. 
$$
Writing $\CS':=\ker(\Phi':\pr_2^*V\rightarrow \CF'')$ for the kernel vector bundle, we have
$$N_\iota=\R\pr_{1*}\R\mathcal{H}om(\CS',\CF'')
$$
in $K$-theory of $Q^{=k}_d$.\footnote{We expect that the equality holds on the level of vector bundles, but we do not need this fact.} Note that the vector bundle on the right hand side is naturally a pullback from $\Quot_{k}\times \Coh_{d-k}$. Since $\overline{q}$ is a vector bundle stack morphism, it induces an isomorphism 
$$\overline{q}^*:H^*(\Quot_k\times\Coh_{d-k})\xrightarrow{\simeq}H^*(Q^{=k}_d). 
$$
Therefore, it suffices to show that the Euler class of the vector bundle $\R\pr_{1*}\R\mathcal{H}om(\CS',\CF'')$ acts injectively on the cohomology ring $H^*(\Quot_k\times \Coh_{d-k})$. Since $d-k\geq 1$, we have a $\mathbb{G}_m$-gerbe morphism
$$\Quot_k\times \Coh_{d-k}\rightarrow \Quot_k\times (\Coh_{d-k})^{\textnormal{rig}}
$$
where $(\Coh_{d-k})^{\textnormal{rig}}$ is the rigidification of the stack $\Coh_{d-k}$ by removing the constant stabilizer group $\BC^*$ at every point; see \cite{AOV} for the precise definition. Since the vector bundle $\R\pr_{1*}\R\mathcal{H}om(\CS',\CF'')$ has weight $1$ with respect to the above $\mathbb{G}_m$-gerbe morphism, its Euler class acts injectively by \cite[Lem. 5.13]{Heinloth_lecturenote}.

\end{proof}

By the above lemma, we have a sequence of surjective ring homomorphisms
$$H^*(\Coh_d)\xrightarrow{\ \simeq\ } H^*(\Tot(V^{[d]}))=H^*(Q^{\geq 0}_d)\twoheadrightarrow H^*(Q^{\geq 1}_d)\twoheadrightarrow\cdots \twoheadrightarrow H^*(Q^{\geq d}_d)=H^*(\Quot_{d}) 
$$
which in turn defines a filtration of ideals
\begin{equation}\label{eq: filtration of the ideal}
    0=I_d^{\geq 0}\subseteq I_d^{\geq 1}\subseteq\cdots\subseteq I_d^{\geq d}=I_d \subseteq H^*(\Coh_{d})
\end{equation}
where $$I_d:= \ker\Big(\pi_d^*:H^*(\Coh_d)\twoheadrightarrow H^*(\Quot_d)\Big),\quad I^{\geq k}_d:=\ker\Big(\pi^{\geq k}_d: H^*(\Coh_d)\twoheadrightarrow H^*(Q^{\geq k}_d)\Big).$$
To study the ideal $I_d$, we first describe the associated graded pieces of the filtration of $I_d$ in terms of the Hall product $\star$.

\begin{proposition}\label{prop: main proposition for basis}

For each $0\leq k<d$, the cohomological Hall algebra product satisfies
$$\BH_k\star(e(V^{[d-k]})\cdot \BH_{d-k})\subseteq I^{\geq k+1}_d,$$
and it induces an isomorphism 
$$\star: \Big(\BH_k/I_k\Big)\otimes \Big(e(V^{[d-k]})\cdot \BH_{d-k}\Big)\overset{\simeq}{\longrightarrow} I^{\geq k+1}_d/I^{\geq k}_d. 
$$
\end{proposition}
\begin{proof}
We use the commutative diagram below.  
\begin{equation*}
\begin{tikzcd}
                                                        &  & Q^{=k}_d \arrow[rd, "\widetilde \iota"'] \arrow[rrd, "\iota"] \arrow[dd, "{\pi_{k, d-k}}"'] \arrow[lld, "\overline{q}"'] &                                                                                        &                                           \\
\Quot_k\times \Coh_{d-k} \arrow[dd, "\pi_k\times \id"'] &  &                                                                                                                          & \widetilde Q^{=k}_d \arrow[r, "\overline{p}"'] \arrow[ld, "\overline{\pi}_d^{\geq k}"] & Q^{\geq k}_d \arrow[dd, "\pi_d^{\geq k}"] \\
                                                        &  & {\Coh_{k, d-k}} \arrow[rrd, "p"] \arrow[lld, "q"']                                                                       &                                                                                        &                                           \\
\Coh_k\times \Coh_{d-k}                                 &  &                                                                                                                          &                                                                                        & \Coh_d                                   
\end{tikzcd}
\end{equation*}
Recall that the square at the right bottom corner is Cartesian and $\widetilde{\iota}$ is identified with a zero locus of a tautological section of $V^{[d-k]}$ pulled back from $\Coh_{d-k}$. Therefore, we have
\begin{align*}
    \iota_*\circ \overline{q}^*\circ (\pi_k\times \id)^*
    &=\overline{p}_*\circ \widetilde{\iota}_*\circ \widetilde{\iota}^*\circ (\overline{\pi}^{\geq k}_d)^*\circ q^*\\
    &=\overline{p}_*\circ e((\overline{\pi}_d^{\geq k})^*q^*V^{[d-k]})\circ (\overline{\pi}^{\geq k}_d)^*\circ q^*\\
    &=\overline{p}_*\circ (\overline{\pi}^{\geq k}_d)^*\circ q^*\circ e(V^{[d-k]})\\
    &=(\pi_d^{\geq k})^*\circ p_*\circ q^*\circ e(V^{[d-k]})\\
    &=(\pi_d^{\geq k})^*\circ \star \circ e(V^{[d-k]}). 
\end{align*}
In other words, we have a commutative diagram 
\begin{equation}\label{eq: key comm diagram}
    \begin{tikzcd}
H^{*-2N(d-k)}(\Quot_k\times \Coh_{d-k}) \arrow[r, "\simeq"', "\overline{q}^*"]                                     & H^{*-2N(d-k)}(Q^{=k}_d) \arrow[r, hook, "\iota_*"]                  & H^*(Q_d^{\geq k})                           \\
H^{*-2N(d-k)}(\Coh_k\times\Coh_{d-k}) \arrow[u, "(\pi_k\times\id)^*", two heads] \arrow[r, hook, "e(V^{[d-k]})"] & H^*(\Coh_k\times\Coh_{d-k}) \arrow[r, "\star"] & H^*(\Coh_d). \arrow[u, two heads, "(\pi_d^{\geq k})^*"]
\end{tikzcd}
\end{equation}
Note that surjectivity of the vertical arrows and injectivity of $\iota_*$ and $e(V^{[d-k]})$ follow from Lemma \ref{lem: Gysin splitting}. By the definition of the ideal $I_d^{\geq k}$, the Gysin short exact sequence from Lemma \ref{lem: Gysin splitting} can be rewritten as 
$$0\rightarrow H^*(Q^{=k}_d)\xrightarrow{\iota_*} \BH_d/I_d^{\geq k}\rightarrow \BH_d/I_d^{\geq k+1}\rightarrow 0. 
$$
In other words, the image of the map from the bottom-left corner to the upper-right corner of the diagram \eqref{eq: key comm diagram} identifies with $I^{\geq k+1}_d/I^{\geq k}_d$. The proposition then follows by using the lower part of the diagram. 

\end{proof}

\begin{theorem} \label{thm:quotschemekernel}
    Let $C$ be a smooth quasi-projective curve and $V$ be a vector bundle on $C$. Then $\BV$ is a cyclic $\BH$-module such that the kernel
$$I:=\ker\Big(\BH\twoheadrightarrow\BV,\quad \alpha\mapsto \alpha|0\rangle\Big)
    $$
    is equal to 
    $$\sum_{d_1\geq 0,\ d_2>0} \BH_{d_1}\star(e(V^{[d_2]})\cdot\BH_{d_2}). 
    $$
\end{theorem}

\begin{proof}
    The cyclicity follows from the fact that $\pi_d^*:\BH_d=H^*(\Coh_d)\twoheadrightarrow \BV_d=H^*(\Quot_d)$ is surjective. Recall the filtration \eqref{eq: filtration of the ideal} of the ideal $I_d$. By Proposition \ref{prop: main proposition for basis}, we have 
    $$I_d^{\geq k+1}
        =I_d^{\geq k}+\Big(\BH_k\star(e(V^{[d-k]})\cdot\BH_{d-k})\Big)
    $$
    for any $0\leq k<d$. By applying this identity inductively, we obtain
    \begin{align*}
        I_d
        &=I_d^{\geq d-1}+\Big(\BH_{d-1}\star(e(V^{[1]})\cdot\BH_{1})\Big)\\
        &=I_{d}^{\geq d-2}+\Big(\BH_{d-2}\star(e(V^{[2]})\cdot\BH_2)\Big)+\Big(\BH_{d-1}\star(e(V^{[1]})\cdot\BH_{1})\Big)\\
        &\hspace{7
        pt}\vdots\\
        &=\sum_{s=1}^{d} \BH_{d-s}\star(e(V^{[s]})\cdot\BH_{s}).
    \end{align*}

\end{proof}

\subsection{Ring presentation and basis}

Since $\BH_d\simeq \big(H^*(C)^{\otimes d}[z_1,\dots,z_d]\big)^{S_d}$ is an explicit algebra and $I_d\subseteq \BH_d$ admits a formula given in terms of the CoHA product by Theorem \ref{thm:quotschemekernel}, we get a presentation of the cohomology ring of $\Quot_d$ by
$$H^*(\Quot_d)\simeq \BH_d/I_d. 
$$
We illustrate how to work with this presentation in the following two examples of $d=1$ and $d=2$.

\begin{example} \label{ex: d=1 ring}
When $d=1$, the Quot scheme $\Quot_{d}(V)$ is isomorphic to a projective bundle $\mathbb{P}(V)$. Then Theorem \ref{thm:quotschemekernel} implies that we have an isomorphism of rings 
\[ H^{*}(\mathbb{P}(V)) \simeq \mathbb{H}_1/I_1\] 
with 
\[  I_1= \mathbb{H}_0 \star \big(e(V^{[1]}) \cdot \mathbb{H}_1\big) = e(V^{[1]})\cdot \mathbb{H}_1.\] 
We may identify $\mathbb{H}_1 \simeq H^*(C)[z]$ where $z=c_1(\mathcal{L})$ and $\mathcal{L}$ is the universal line bundle on $\mathrm{B}\mathbb{G}_m$. But $e(V^{[1]}) = e(V^{*} \otimes \mathcal{L}) = \prod_{i=1}^{N}(z-v_i)$, where $v_1,\cdots, v_N$ are the Chern roots of $V$.  Thus \[ H^{*}(\mathbb{P}(V)) \simeq H^*(C)[z]/\prod_{i=1}^N (z-v_i) \] recovering the well known projective bundle formula. 
\end{example}

\begin{example} \label{ex: d=2 ring}

When $d=2$, we have isomorphism of rings \[ H^*(\mathrm{Quot}_2(V)) \simeq \BH_2/I_2\] with $\mathbb{H}_2 \simeq (H^*(C)^{\otimes 2}[z_1,z_2])^{S_2}$ and $I_2= e(V^{[2]}) \cdot \mathbb{H}_2+ \mathbb{H}_1 \star (e(V^{[1]}) \cdot \BH_1). $ We describe the quotient more explicitly. Let $P(z) = e(V^{[1]}) = \prod_{i=1}^N (z-v_i)$. Then \[ e(V^{[2]})=P(z_1)P(z_2).\]
For $\alpha,\beta \in H^*(C)$ and integers $n_1,n_2 \geq 0$, define \[ A^{\alpha,\beta}_{n_1,n_2} := \alpha z^{n_1} \star P(z) \beta z^{n_2}. \] Then by definition of shuffle product
\[ A^{\alpha,\beta}_{n_1,n_2}= (\alpha \otimes \beta) z_1^{n_1} P(z_2)z_2^{n_2}+(-1)^{|\alpha||\beta|}(\beta\otimes\alpha)z_2^{n_1} P(z_1)z_1^{n_2}+\Delta_*(\alpha\cup\beta)\left( \frac{z_1^{n_1}z_2^{n_2}P(z_2)-z_1^{n_2}z_2^{n_1}P(z_1)}{z_2-z_1}\right). \] The kernel $I_2$ is generated by $P(z_1)P(z_2)$ and $A^{\alpha,\beta}_{n_1,n_2}$. Since \[ P(z) \star P(z) = 2P(z_1)P(z_2),\] $P(z_1)P(z_2)$ is contained in the ideal generated by $A^{\alpha,\beta}_{n_1,n_2}$. Note that\[ (z_1+z_2) A^{\alpha,\beta}_{n_1,n_2} =  A^{\alpha,\beta}_{n_1+1,n_2} +  A^{\alpha,\beta}_{n_1,n_2+1}\] and \[ (z_1z_2) A^{\alpha,\beta}_{n_1,n_2}= A^{\alpha,\beta}_{n_1+1,n_2+1}.\] So $A^{\alpha,\beta}_{0,0}$ and $A^{\alpha,\beta}_{0,1}$ generate the ideal $I_2$. Thus we have an isomorphism of rings \[ H^*(\mathrm{Quot}_2(V)) \simeq \big(H^*(C)^{\otimes 2}[z_1,z_2]\big)^{S_2}/\langle A^{\alpha,\beta}_{0,0}, A^{\alpha,\beta}_{0,1} \mid \alpha,\beta \in \CB \rangle\]
where $\CB$ is a basis of $H^*(C)$. 
\end{example}

We now use Theorem \ref{thm:quotschemekernel} to give a natural basis of the cohomology of punctual Quot schemes and thereby calculate the Poincar\'e polynomial of Quot schemes for quasi-projective curves. For this application, we need another input from cohomological Hall algebras. Note that each $\mathbb{H}_d$ is equipped with another product structure, given by the cup product in the cohomology ring. We then have the following compatibility between cup product and CoHA algebra structure:

\begin{proposition} \label{CoHAcupcompatibility}
    Let $d=d_1+d_2$ and consider the direct sum morphism
\[ \oplus_{d_1,d_2}: \Coh_{d_1} \times \Coh_{d_2} \rightarrow \Coh_{d}. \]
    If $x\in \BH_d,\ y\in \BH_{d_1},\ z\in \BH_{d_2}$, then we have 
    $$x\cup(y\star z)=\sum (-1)^{|x_{d_2}||y|}(x_{d_1}\cup y)\star (x_{d_2}\cup z),
    $$
    where $\oplus_{d_1,d_2}^*(x)=\sum x_{d_1}\otimes x_{d_2}\in \BH_{d_1}\otimes \BH_{d_2}$ using the Sweedler notation.

\end{proposition}

\begin{proof}
Recall the CoHA diagram
\begin{center}
    \begin{tikzcd}
	& {\Coh_{d_1,d_2}} \\
	{\Coh_{d_1} \times \Coh_{d_2}} && {\Coh_{d},}
	\arrow["q"{description}, from=1-2, to=2-1]
	\arrow["p", from=1-2, to=2-3]
	\arrow["s", shift left=3, dotted, from=2-1, to=1-2]
	\arrow["{\oplus_{d_1,d_2}}"', dotted, from=2-1, to=2-3]
\end{tikzcd}
\end{center}
where $s$ is the direct sum section and we have $p\circ s=\oplus_{d_1,d_2}$. Since $q\circ s=\mathrm{id}$, we have $s^{*}q^{*}=\id$. Since $q^{*}$ is an isomorphism, $s^{*} = (q^{*})^{-1}$. The proposition follows from a direct computation
\begin{align*}
   p_{*}q^{*}(\oplus^{*}(x) \cup (y \otimes z)) &=  p_{*}(q^{*}s^{*}p^{*}(x) \cup q^{*}(y \otimes z)) \\ 
   &= p_{*}(p^{*}(x) \cup q^{*}(y \otimes z)) \\ 
   &= x \cup (y \star z).
\end{align*}
\end{proof}

\begin{theorem}\label{thm: Nakajima type basis} 
Fix an ordered basis $\alpha_1 < \alpha_2 <  \cdots< \alpha_m $ of $H^{*}(C)$. Then $\mathbb{V}$ has a basis given by monomials \[ e^{\alpha_{k_1}}_{i_1} \star e^{\alpha_{k_2}}_{i_2} \star \cdots \star e^{\alpha_{k_d}}_{i_d} \vac \] satisfying 
\[ d\geq 0,\quad 0\leq i_1 \leq i_2\leq \cdots \leq i_d < N\] and, whenever $i_j=i_{j+1}$ then one has $\alpha_{k_j} \leq \alpha_{k_{j+1}}$ with strict inequality when both classes are odd. 
\end{theorem}

\begin{proof}

By Theorem \ref{thm: CoHA equals shuffle}, $\BH$ admits a PBW type basis 
\[
\CB:=\left\{
e^{\alpha_{k_1}}_{i_1}\star\cdots\star e^{\alpha_{k_d}}_{i_d}\,\big|\, 0\leq i_1\leq\cdots\leq i_d\ \textnormal{and ($\dagger$)} 
\right\}
\]
where 
\begin{center}
    ($\dagger$)\quad if $i_j=i_{j+1}$ then $\alpha_{k_j}\leq \alpha_{k_{j+1}}$ with strict inequality when both $\alpha_{k_j}$, $\alpha_{k_{j+1}}$ are odd. 
\end{center}
We define another PBW type basis 
\[
\widetilde\CB:=\left\{
\tilde e^{\alpha_{k_1}}_{i_1}\star\cdots \star \tilde e^{\alpha_{k_d}}_{i_d}\,\big|\, 0\leq i_1\leq\cdots\leq i_d\ \textnormal{and ($\dagger$)} 
\right\}
\]
where
$$\tilde e_i^{\alpha_k}:=
\begin{cases}
    e_i^{\alpha_k}&\textnormal{if}\quad i<N,\\
    e(V^{[1]})\cup e_{i-N}^{\alpha_k}&\textnormal{if}\quad i\geq N. 
\end{cases}
$$
Since
\[ \tilde{e}^{\alpha_{k_1}}_{i_1} \star \tilde{e}^{\alpha_{k_2}}_{i_2} \star \cdots \star \tilde{e}^{\alpha_{k_d}}_{i_d}  = e^{\alpha_{k_1}}_{i_1} \star e^{\alpha_{k_2}}_{i_2} \star \cdots \star e^{\alpha_{k_d}}_{i_d} +\textrm{(element of smaller filtered degree)},\] 
the transition matrix between $\CB$ and $\widetilde\CB$ is invertible, hence $\widetilde{\mathcal{B}}$ indeed forms a basis of $\BH$. 

By a simple application of Proposition \ref{CoHAcupcompatibility}, the morphism 
$$\Phi:\BH\rightarrow \BH,\quad x=\sum_{d=0}^\infty x_d \mapsto \sum_{d=0}^\infty e(V^{[d]})\cup x_d
$$
is an algebra homomorphism. For every element $\tilde e^{\alpha_{k_1}}_{i_1}\star\cdots \star \tilde e^{\alpha_{k_d}}_{i_d}\in \widetilde\CB$, there exists a unique $0\leq \ell\leq d$ such that
$$i_1\leq \cdots \leq i_\ell<N\leq i_{\ell+1}\leq \cdots \leq i_d
$$
and so 
\begin{align}\label{eq: new PBW}
    \tilde e^{\alpha_{k_1}}_{i_1}\star\cdots \star \tilde e^{\alpha_{k_d}}_{i_d}
    &=e^{\alpha_{k_1}}_{i_1}\star\cdots \star e^{\alpha_{k_\ell}}_{i_\ell}\star \Phi(e^{\alpha_{k_{\ell+1}}}_{i_{\ell+1}-N})\star \cdots \star 
    \Phi(e^{\alpha_{k_d}}_{i_{d}-N}) \notag\\
    &=e^{\alpha_{k_1}}_{i_1}\star\cdots \star e^{\alpha_{k_\ell}}_{i_\ell}\star \Phi(e^{\alpha_{k_{\ell+1}}}_{i_{\ell+1}-N}\star \cdots \star 
    e^{\alpha_{k_d}}_{i_{d}-N}).
\end{align}
Define a subset 
\[
\CB_N:=\left\{
\tilde e^{\alpha_{k_1}}_{i_1}\star\cdots \star \tilde e^{\alpha_{k_d}}_{i_d}\,\big|\, 0\leq i_1\leq\cdots\leq i_d<N\ \textnormal{and ($\dagger$)}
\right\} \subseteq \widetilde \CB.
\]
Since every element in the basis $\widetilde{\CB}$ can be written uniquely in the form of \eqref{eq: new PBW}, the CoHA multiplication induces an isomorphism 
\begin{equation}\label{eq: key isomorphism}
    \star:\textnormal{span}(\CB_N)\otimes \Phi(\BH)\xrightarrow{\simeq}\BH. 
\end{equation}
Consider a splitting 
$$\BH=\BQ\cdot 1\oplus \BH_+,\quad \BH_+:=\bigoplus_{d=1}^\infty \BH_d. 
$$
Using this notation, Theorem \ref{thm:quotschemekernel} says $I=\BH\star \Phi(\BH_+)$. Therefore, the isomorphism \eqref{eq: key isomorphism} restricts to an injective map
$$\star:\textnormal{span}(\CB_N)\otimes \Phi(\BH_+)\rightarrow I. 
$$
This map is also surjective because every element $x\star \Phi(y)\in \BH\star \Phi(\BH_+)=I$ can be written as 
$$
    x\star \Phi(y)
    =\sum x_1\star \Phi(x_2)\star \Phi(y)
    =\sum x_1\star \Phi(x_2\star y)\in \textnormal{span}(\CB_N)\star \Phi(\BH_+)
$$
where $x=\sum x_1\otimes \Phi(x_2)$ with $x_1\in \textnormal{span}(\CB_N)$ and $x_2\in \BH$ using Sweedler's notation according to the isomorphism \eqref{eq: key isomorphism}. Therefore, the isomorphism \eqref{eq: key isomorphism} induces a splitting 
$$\textnormal{span}(\CB_N)\otimes\big(\BQ\cdot 1\oplus \BH_+)\xrightarrow{\simeq} \textnormal{span}(\CB_N)\oplus I=\BH. 
$$
Since $\textnormal{span}(\CB_N)$ gives a splitting of the kernel of the surjective map $\BH\twoheadrightarrow \BV$, this completes the proof. 

\end{proof}

\begin{corollary}\label{cor: Poincare}
For any quasi-projective curve $C$ and a vector bundle $V$ of rank $N$, the generating series of the Poincar\'e polynomial of punctual Quot schemes
\[ P_{\mathrm{Quot}(V)}(t,z) := \sum_{d=0}^{\infty}\sum_{k=0}^{\infty} \dim(H^{k}(\Quot_{d}(V)))\,z^kt^d \] is given by 
\[ \prod_{i=0}^{N-1} \frac{(1+tz^{2i+1})^{b_1(C)}}{(1-tz^{2i})^{b_0(C)}(1-tz^{2i+2})^{b_2(C)}}. \] 
\end{corollary}

\begin{remark}\label{rem: Poincare}
For smooth projective curves with trivial vector bundles, the above formula was established in \cite{Bifet}. Since the motive of punctual Quot schemes depends on the vector bundle $V$ only through its rank $N$ by \cite{BFP}, the same formula holds for any smooth projective curves with arbitrary vector bundles. However, we note that for smooth quasi-projective curves, the motivic computation in \cite{BFP} recovers the generating series of virtual Poincar\'e polynomials, but not necessarily actual Poincar\'e polynomials. For this reason, the above formula appears to be new for general smooth quasi-projective curves. 

\end{remark}

\section{Commutator relations} \label{sec: Commutator relations}

In this section, we compute the commutators between the creation, annihilation, and multiplication operators acting on the virtual homology of Quot schemes. The main idea is to describe the Hecke modification as a derived projective bundle. This then allows us to compute the one-point creation and annihilation operators explicitly via residues, leading to the commutator formulas between them. The commutator relations proven in this section will motivate the definition of the double of torsion CoHA in Section \ref{sec: Drinfeld double}. For the remaining sections, we let $C$ be a smooth connected projective curve.

\subsection{Hecke correspondence}

In this subsection, we identify the Hecke correspondence diagram with a certain derived projective bundle. This description is well known and has been exploited in geometric representation theory in the work of Negu\c t \cite{N1, N2, N3}. It will be our main tool to compute the creation and annihilation operators via the residue formula. 

Consider the one-point modification nested Quot scheme
$$\Quot_{(r,[d,d+1])}:=\Quot_{(0,1),(r,d)}.$$ 
The Hecke correspondence refers to the correspondence diagram 
\begin{center}
\begin{tikzcd}
  & \Quot_{(r,[d,d+1])} \arrow[ld, "f"'] \arrow[rd, "g"] &   \\
\Coh_{(0,1)}\times \Quot_{(r,d)} &                                    & \Quot_{(r,d+1)}.
\end{tikzcd}
\end{center}
Since $\Coh_{(0,1)}\simeq C\times \BGm$, we may equivalently understand the above diagram using the following geometric structures of the nested Quot scheme
\begin{equation}\label{eq: geometric data for Hecke}
    \begin{tikzcd}
  & \CL \arrow[d]                                         &   \\
  & \Quot_{(r,[d,d+1])} \arrow[ld, "p_-"'] \arrow[d, "\rho"] \arrow[rd, "p_+"] &   \\
\Quot_{(r,d)} & C                                                   & \Quot_{(r,d+1)}.
\end{tikzcd}
\end{equation}
Precisely, if we write
$$(\rho,\CL):\Quot_{(r,[d,d+1])}\rightarrow C\times \BGm\simeq \Coh_{(0,1)},
$$
then 
\begin{equation}\label{eq: comparison for Hecke and CoHA}
    f=\big((\rho,\CL),p_-\big),\quad g=p_+. 
\end{equation}

\begin{definition}
We define the following complexes on $\Quot_{(r,d)}\times C_*$\footnote{The subscript on $C_*$ is used to distinguish it from other copies of the curve $C$ when there are multiple factors of $C$.}
\begin{align*}
    \BK^+_{(r,d),*}&:=\R p_*\RHHom(\CS_\bullet, \CO_{\Delta_{*\bullet}}),\\
    \BK^-_{(r,d),*}&:=\R p_*\RHHom(\CO_{\Delta_{*\bullet}}, \CQ_\bullet),
\end{align*}
where $p:\Quot_{(r,d)}\times C_*\times C_\bullet\rightarrow \Quot_{(r,d)}\times C_*$ is the projection, $\CS_\bullet$ and $\CQ_\bullet$ are the pullbacks of the universal subsheaf and quotient on $\Quot_{(r,d)}\times C_\bullet$, respectively, and $\Delta_{*\bullet}\hookrightarrow C_*\times C_\bullet$ is the diagonal.
\end{definition}

\begin{remark}\label{rem: explicit K complex}
Using the Fourier-Mukai transform of the diagonal and Serre duality, we can rewrite the above complexes as
$$\BK^+_{(r,d),*}\simeq \CS_*^\vee,\quad \BK^-_{(r,d),*}\simeq \CQ_*\otimes K_{C_*}^\vee[-1],
$$
where $(-)^\vee$ denotes the derived dual and $K_{C_*}$ denotes the canonical bundle of the curve $C_*$. 
\end{remark}

Since we are working with a smooth projective curve, $\BK^{\pm}_{(r,d),*}$ is a complex of perfect-amplitude contained in $[0,1]$. We can describe the nested Quot scheme $\Quot_{(r,[d,d+1])}$ explicitly as a derived projective bundle in the sense of Jiang \cite{Jiang}.

\begin{proposition}\label{prop: virtual projective bundle}
    There are isomorphisms $\phi^{\pm}$ forming the commutative diagrams
    \begin{center}
        \begin{tikzcd}
\BP((\BK^+_{(r,d),*})^\vee) \arrow[d, "\pi_+"'] \arrow[r, "\phi^+", , "\simeq"'] & \Quot_{(r,[d,d+1])} \arrow[ld, "p_-\times \rho"] \arrow[rd, "p_+\times \rho"'] & \BP((\BK^-_{(r,d+1),*})^\vee) \arrow[d, "\pi_-"] \arrow[l, "\phi^-"', "\simeq"] \\
\Quot_{(r,d)}\times C_*                        &                                        & \Quot_{(r,d+1)}\times C_*                      
\end{tikzcd}
    \end{center}
    such that $\big(\phi^{\pm}\big)^*\CL\simeq \CO(\pm 1)$.
\end{proposition}

\begin{proof}

We first consider the case of $(p_-\times \rho)$. Recall from \eqref{eq: morphism f} the Cartesian diagram
\begin{equation*}
    \begin{tikzcd}
\Quot_{(r,[d,d+1])} \arrow[r] \arrow[d, "f"] & \Coh_{[V]-(r,d+1), \delta} \arrow[d] \\
\Coh_\delta\times \Quot_{(r,d)} \arrow[r]                 & \Coh_\delta\times \Coh_{[V]-(r,d)}     
\end{tikzcd}
\end{equation*}
where we write $\delta=(0,1)$. Since $\Coh_\delta\simeq C\times \BGm$, we may use the above diagram to obtain another Cartesian diagram 
\begin{equation*}
    \begin{tikzcd}
\Quot_{(r,[d,d+1])} \arrow[r] \arrow[d, "\rho\times p_-"'] & \Coh_{[V]-(r,d+1), \delta} \arrow[d] \\
C_*\times \Quot_{(r,d)} \arrow[r, "\psi"]                 & C_*\times \Coh_{[V]-(r,d)}     
\end{tikzcd}
\end{equation*}
Following the notation of \cite[Sec. 8.1]{Jiang}, the stack $\Coh_{[V]-(r,d+1), \delta}$ parametrizes a short exact sequence $0\rightarrow F'\rightarrow F\rightarrow x_*L\rightarrow 0$ and the right vertical map sends it to $(x,F)$. Then \cite[Prop. 8.5]{Jiang} identifies the right vertical map with the derived projectivization of the universal sheaf $\CF_*$ on $C_*\times \Coh_{[V]-(r,d)}$ such that the relative ample line bundle $\CO(1)$ identifies with the universal line bundle $\CL$ on $\Coh_{[V]-(r,d+1), \delta}$ pulled back from the $\BGm$ factor. Since $\psi^*\CF_*\simeq \CS_*$ and $\CS_*\simeq \big(\BK^+_{(r,d),*}\big)^\vee$ by Remark \ref{rem: explicit K complex}, we obtain the desired derived projectivization description of the morphism $(p_-\times \rho)$. 

Now we consider the case of $(p_+\times \rho)$. As above, by refining the diagram \eqref{eq: morphism g}, we have a Cartesian diagram 
\begin{equation*}
    \begin{tikzcd}
\Quot_{(r,[d,d+1])} \arrow[r] \arrow[d, "\rho\times p_+"'] & \Coh_{\delta,(r,d)} \arrow[d] \\
C_*\times \Quot_{(r,d+1)} \arrow[r]                 & C_*\times \Coh_{(r,d+1)}       
\end{tikzcd}
\end{equation*}
where the right vertical map sends a short exact sequence $0\rightarrow x_*L\rightarrow F\rightarrow F'\rightarrow 0$ to $(x,F)$. Just as before, the desired description of $(\rho\times p_+)$ would follow if we show that the right vertical map identifies with the derived projectivization of the $(\CF_*)^\vee\otimes K_{C_*}[1]$ such that the universal line bundle $\CL$ identifies with $\CO(-1)$. This can be done closely following the proof of \cite[Prop. 8.5]{Jiang} up to application of relative Serre duality as in the proposition below. This completes the proof.

\end{proof}

\begin{proposition}
    Let $C$ be a smooth projective curve and
    $$f:\Coh_{\delta, (r,d)}\rightarrow C\times \Coh_{(r,d+1)}
    $$
    be a morphism sending a short exact sequence $0\rightarrow x_*L\rightarrow F\rightarrow F'\rightarrow 0$ to $(x,F)$. Then we have an isomorphism $\phi$ making the diagram 
    \begin{center}
\begin{tikzcd}
    \Coh_{\delta,(r,d)}\arrow[rd,"f"']\arrow[rr,"\phi","\simeq"']&&\BP(\CF^\vee\otimes K_{C}[1])\arrow[ld,"\pi"]\\
    &C\times \Coh_{(r,d+1)}&
\end{tikzcd}
\end{center}
commutes and $\phi^*\CO(-1)\simeq \CL$.
\end{proposition}
\begin{proof}
    We note that $\CF^\vee\otimes K_C[1]$ is a connective perfect complex on $C\times \Coh_{(r,d+1)}$, hence Jiang's derived projectivization makes sense. We first construct a morphism $\phi$. Let $T$ be a derived scheme and $(x,\CF_T):T\rightarrow C\times \Coh_{(r,d+1)}$ be a morphism induced from $x:T\rightarrow C$ and a perfect complex $\CF_T$ over $T\times C$ which is flat over $T$ and of numerical type $(r,d+1)$. Lifting this morphism to $\Coh_{\delta, (r,d)}$ amounts to giving a line bundle $\CL$ on $T$ together with a morphism 
    \begin{equation}\label{eq: relative SES}
        \iota:(\id,x)_*\CL\rightarrow \CF_T
    \end{equation}
    such that the cone $\cone(\iota)$ is flat over $T$. Note that $\iota$ lies in $\RHom_{T\times C}((\id,x)_*\CL,\CF_T)$ which can be identified as follows:
    \begin{equation}\label{eq: long identification}
    \begin{aligned}
        \RHom_{T\times C}((\id,x)_*\CL,\CF_T)&
        \simeq \RHom_{T\times C}(p^*\CL\otimes (\id,x)_*\CO, \CF_T)\\
        &\simeq \RHom_{T\times C}(p^*\CL, \RHHom((\id,x)_*\CO,\CF_T))\\
        &\simeq \RHom_T(\CL, \R p_*\RHHom((\id,x)_*\CO,\CF_T))\\
        &\simeq \RHom_T(\big(\R p_*\RHHom((\id,x)_*\CO,\CF_T)\big)^\vee, \CL^\vee)\\
        &\simeq \RHom_T(\R p_*\RHHom(\CF_T, K_C[1]\otimes (\id,x)_*\CO),\CL^\vee)\\
        &\simeq \RHom_T((\id,x)^*(\CF_T^\vee\otimes K_C[1]),\CL^\vee)\\
        &\simeq \RHom_T((\id,\CF_T)^*(\CF^\vee\otimes K_C[1]), \CL^\vee)
    \end{aligned}
    \end{equation}
    where $p:T\times C\rightarrow T$ is the projection map. Note that we adjunctions and $p\circ (\id,x)=\id$ multiple times together with a $p$-relative Serre duality at the fifth isomorphism. Therefore, the morphism $\iota$ in \eqref{eq: relative SES} corresponds to a certain morphism 
    \begin{equation}\label{eq: derived projectivization morphism}
        q:(\id,\CF_T)^*(\CF^\vee\otimes K_C[1])\rightarrow \CL^\vee.
    \end{equation}
    Note that the procedure can be reversed, hence sending a morphism $q$ to $\iota$. Recall that lifting a morphism $(x,\CF_T)$ to the derived projectivization amounts to giving a morphism $q$ as in \eqref{eq: derived projectivization morphism} which is surjective after taking $H^0(-)$. 

    Based on the above discussion, in order to construct the isomorphism $\phi$ such that $\phi^*\CO(-1)\simeq \CL$, it suffices to show that $\cone(\iota)$ is flat over $T$ if and only if $q$ is surjective at $H^0(-)$. Since both properties can be checked after the base change to the classical truncation, it suffices to work on the level of classical stacks, hence assuming that $T$ is a classical scheme. Note that when $T$ is a closed point, the two conditions become both equivalent to $\iota$ and $q$ being a nonzero map, which are clearly intertwined under the isomorphism \eqref{eq: long identification}. When $T$ is a classical scheme, the condition of $\cone(\iota)$ being flat over $T$ is equivalent to the morphism $\iota$ being injective after arbitrary base change $T'\rightarrow T$ and the condition of $q$ being surjective can be checked at the level of closed points of $T$. Having this noted, the two conditions can be easily identified, hence completing the proof. 
    
\end{proof}

\begin{remark}\label{rem: derived projective bundle as derived zero}
When the complex $E^\bullet$ admits a 2-term resolution by vector bundles, we can describe the derived projective bundle explicitly. Let $X$ be a derived stack and $E^\bullet=[E^{-1}\xrightarrow{f}E^0]\in \Perf(X)$ be a perfect complex with a choice of a 2-term resolution by vector bundles in degrees $-1$ and $0$. By \cite[Prop. 4.33]{Jiang}, the derived projective bundle $\BP(E^\bullet)$ is equal to the derived zero locus 
\begin{equation*}
    \begin{tikzcd}
\BP(E^\bullet)=\Zero(s) \arrow[d, "\pi"'] \arrow[r, "i", hook] & \BP(E^0) \arrow[ld, "\tilde\pi"] \\
X                                       &                       
\end{tikzcd}
\end{equation*}
where $s$ is a section of a vector bundle defined as a composition
$$s:\tilde{\pi}^*E^{-1}\xrightarrow{\tilde{\pi}^*f}\tilde{\pi}^*E^0\twoheadrightarrow \CO(1).
$$
The tautological bundle $\CO(1)$ on the derived projective bundle $\BP(E^\bullet)$ is equal to the pullback $i^*\CO(1)$. By construction, the morphism $\pi$ is quasi-smooth, inducing a relative perfect obstruction theory for the classical truncation. 

\end{remark}

Just like the usual projective bundles, the virtual projective bundle admits a virtual pushforward formula in terms of a certain Segre class. 

\begin{definition}
    Define the shifted Segre series of the perfect complex $E^\bullet$ as
$$\widehat{s}_w(E^\bullet):=\prod_{i\in I}\frac{1}{w+x_i}\in H^*(X)(\!(w^{-1})\!)
$$
where $\{x_i\}_{i\in I}$ is the set of Chern roots of $E^\bullet$.\footnote{If $E^\bullet=E^0-E^1$ in $K$-theory for some vector bundles $E^0$ and $E^1$, then the set of Chern roots of $E^\bullet$ is defined to be $\{x_i\}_{i\in I}\ominus\{y_j\}_{j\in J}$ where $x_i$'s and $y_j$'s are Chern roots of $E^0$ and $E^1$, respectively, so that $\widehat{s}_w(E^\bullet)=\widehat{s}_w(E^0)/\widehat{s}_w(E^1)$.} This can be written in terms of the usual Segre polynomial as $\widehat s_w(E^\bullet)=w^{-\rk(E^\bullet)}\cdot s_{w^{-1}}(E^\bullet)$. 

\end{definition}

Since we often work with formal series involving infinitely many negative powers, such as the shifted Segre series, we introduce the following notion of residue.

\begin{definition}
    Let $R$ be a ring and $F(z)=\sum_{k\in \BZ} a_k z^k\in R(\!(z^{-1})\!)$, i.e., $a_k=0$ if $k$ is sufficiently positive. We define the residue of $F(z)$ as
    $$\Res_z F(z):=a_{-1}\in R. 
    $$
\end{definition}
\begin{remark}
    If $F(z)$ is a meromorphic function expanded at infinity, then the above definition equals the {\it minus} of the residue at infinity, or equivalently, the sum of residues of $F(z)$ at every finite poles. 
\end{remark}

\begin{lemma}
\label{lem: virtual pushforward}
    Let $X$ be a derived stack and $E^\bullet$ be a perfect complex admitting a $2$-term resolution by vector bundles in degrees $0$ and $1$. Let $\pi:\BP((E^\bullet)^\vee)\rightarrow X$ be a derived projective bundle, $z:=c_1(\CO(1))$. For any $f(z)\in H^*(X)[z]$, the virtual Umkehr map is given by
    $$\pi_!(f(z))=\Res_z\Big(f(z)\cdot \widehat{s}_z(E^\bullet)\Big)\in H^*(X).
    $$
\end{lemma}

\begin{proof}

This is essentially proven in \cite[eq. (43)]{marian2026cohomologyquotschemesmooth}, but we record the proof for this version for completeness. Pick a resolution $E^\bullet=[E^0\xrightarrow{f}E^1]$ by vector bundles in degree $0$ and $1$. Then $(E^\bullet)^\vee=[(E^1)^\vee\rightarrow (E^0)^\vee]$ lies in degree $-1$ and $0$ for which the setting of Remark \ref{rem: derived projective bundle as derived zero} applies. By linearity and projection formula, we may assume that $f(z)=z^k$. In the case of the usual projective bundle, the formula we need to prove follows from the definition of the Segre class with the chosen shift. Setting $\tilde z:=c_1(\CO(1))\in H^2(\BP((E^0)^\vee))$, $r_i:=\rk(E^i)$ and $r:=\rk(E^\bullet)=r_0-r_1$, the general case follows by straightforward computation:
    \begin{align*}
        \pi_!(z^k)
        &=\tilde \pi_!(\tilde z^k\cdot i_!1)\\
        &=\tilde \pi_!\left(\tilde z^k\cdot e(\tilde\pi^*E^1(1))\right)\\
        &=\sum_{i=0}^{r_1}\tilde \pi_!(\tilde z^{k+r_1-i})\cdot c_i(E^1)\\
        &=\sum_{i=0}^{r_1}s_{k-i-r+1}(E^0)\cdot s_i(-E^1)\\
        &=s_{k-r+1}(E^\bullet)\\
        &=\Res_z\Big(z^k\cdot \widehat s_z(E^\bullet)\Big).
    \end{align*}
\end{proof}

\begin{remark}
    By the virtual projection formula (A13), the above lemma implies
    $$\pi_*(f(z)\cdot \pi^!\alpha)=\Res_z\Big(f(z)\cdot \widehat{s}_z(E^\bullet)\Big)\cdot \alpha,\quad \alpha\in H^{\BM}_*(X).
    $$
\end{remark}

\subsection{Generating series of the operators}\label{sec: generating series}

We now apply results from the previous section to the setting of one-point creation and annihilation operators. Using the geometric data of the one-pointed nested Quot scheme \eqref{eq: geometric data for Hecke}, we define the operators (without cohomological decoration)\footnote{In the definition of $e_i$ and $f_i$, $z^i=c_1(\CL)^i$ where $\CL$ is the line bundle on the relevant nested Quot scheme corresponding to a one-point modification. } 
\begin{align}
e_i&:=(p_+\times \rho)_*\circ z^i\circ (p_-)^!:\BV\rightarrow \BV\otimes H^*(C),\label{eaction}\\
f_i&:=(p_-\times \rho)_*\circ z^i\circ (p_+)^!:\BV\rightarrow \BV\otimes H^*(C) \label{faction},
\end{align} 
and their generating series
\begin{align*}
e(u)&:=\sum_{i=0}^\infty e_i u^{-i-1}:\BV\rightarrow \BV\otimes H^*(C)(\!(u^{-1})\!),\\
f(v)&:=\sum_{i=0}^\infty f_i v^{-i-1}:\BV\rightarrow \BV\otimes H^*(C)(\!(v^{-1})\!).
\end{align*}
It suffices to consider the operators $e_i$ and $f_i$ (and their generating series) because they recover the CoHA action operators $e_i^\alpha$ and $f_i^\alpha$ with a cohomological decoration as follows.

\begin{lemma}\label{lem: Hecke and CoHA actions}
    Let $\alpha\otimes z^i\in H^*(\Coh_{(0,1)})\simeq H^*(C)[z]$ and $e_i^\alpha\in \BH$ and $f_i^\alpha\in \BH^\op$ be the corresponding elements. For $x\in H_*(\Quot_{(r,d)})$, we have
    \begin{align*}
        e_i^\alpha\star x &= \pr_{1*}(\pr_2^*\alpha\cdot e_i(x)),\\
        f_i^\alpha\star x &= \pr_{1*}(\pr_2^*\alpha\cdot f_i(x)).
    \end{align*}
    
\end{lemma}
\begin{proof}

The first equality readily follows from the equality \eqref{eq: comparison for Hecke and CoHA}. Indeed, we have 
\begin{align*}
    g_*f^!(e_i^\alpha\boxtimes x)
    &=(p_+)_*\Big((\rho,\CL)^*(\alpha\otimes z^i)\cdot (p_-)^!(x)\Big)\\
    &=\pr_{1*}\Big(\pr_2^*\alpha\cdot (\rho\times p_+)_*(z^i\cdot (p_-)^!x)\Big)\\
    &=\pr_{1*}(\pr_2^*\alpha\cdot e_i(x)).
\end{align*}

For the second equality, recall from Section \ref{sec: annihilation} that the annihilation operator $f_i^\alpha\star x $ is defined using the topological duality and replacing the right action by the left action of the Koszul signed opposite algebra. Therefore, it suffices to show that 
$$(e_i^\alpha\star\gamma, x)_{(r,d)} = (-1)^{|\alpha||x|}\Big(\gamma, \pr_{1*}\big(\pr_2^*\alpha\cdot f_i(x)\big)\Big)_{(r,d-1)},
$$
for all $\gamma\in H^*(\Quot_{(r,d-1)})$. The left hand side is given by 
\begin{equation}\label{eq: LHS with sign}
\begin{aligned}
    (e_i^\alpha\star\gamma, x)_{(r,d)}
    &=(-1)^{|\alpha|+|\gamma|}\cdot \deg\left(
    (e_i^\alpha\star \gamma)\cdot x
    \right)\\
    &=(-1)^{|\alpha|+|\gamma|}\cdot \deg\left(
    (p_+)_!(\rho^*\alpha\cdot z^i\cdot (p_-)^*\gamma)\cdot x
    \right)\\
    &=(-1)^{|\alpha|+|\gamma|}\cdot \deg\left(
    \rho^*\alpha\cdot z^i\cdot (p_-)^*\gamma\cdot (p_+)^!x
    \right)
\end{aligned}
\end{equation}
where the sign $(-1)^{|\alpha|+|\gamma|}$ comes from our convention for the topological duality and the last equality uses the virtual projection formula (A13). On the other hand, the right hand side is given by 
\begin{equation}\label{eq: RHS with sign}
\begin{aligned}
    (-1)^{|\alpha||x|}\Big(\gamma, \pr_{1*}\big(\pr_2^*\alpha\cdot f_i(x)\big)\Big)_{(r,d-1)}
    &=(-1)^{|\alpha||x|}\cdot (-1)^{|\gamma|}\cdot
    \deg\left(\gamma\cdot \pr_{1*}\big(\pr_2^*\alpha\cdot f_i(x)\big)
    \right)\\
    &=(-1)^{|\alpha||x|}\cdot (-1)^{|\gamma|}\cdot
    \deg\left(
    \pr_1^*\gamma\cdot \pr_2^*\alpha\cdot (\rho\times p_-)_*(z^i\cdot (p_+)^!x)
    \right)\\
    &=(-1)^{|\alpha||x|}\cdot (-1)^{|\gamma|}\cdot
    \deg\left(
    (p_-)^*\gamma\cdot \rho^*\alpha\cdot z^i\cdot (p_+)^!x
    \right)\\
    &=(-1)^{|\alpha||x|}\cdot (-1)^{|\gamma|}\cdot (-1)^{|\gamma||\alpha|}\cdot
    \deg\left(
    \rho^*\alpha\cdot (p_-)^*\gamma\cdot z^i\cdot (p_+)^!x
    \right)
\end{aligned}
\end{equation}
where the sign $(-1)^{|\gamma|}$ comes from the topological duality and the second equality uses the definition of $f_i(x)$. Note that both \eqref{eq: LHS with sign} and \eqref{eq: RHS with sign} vanish unless $(-1)^{|\alpha|+|\gamma|}=(-1)^{|x|}$, and when this holds the signs in front of the two expressions match. This completes the proof.

\end{proof}

\begin{lemma}
    The generating series of operators $e(u)$ and $f(v)$ preserve the virtual homology group, i.e., they restrict to 
    $$e(u):\BV^\vir\rightarrow \BV^\vir\otimes H^*(C)(\!(u^{-1})\!),\quad
    f(v):\BV^\vir\rightarrow \BV^\vir\otimes H^*(C)(\!(v^{-1})\!). 
    $$
\end{lemma}
\begin{proof}
    This immediately follows from Theorem \ref{thm: virtual homology is preserved} and Lemma \ref{lem: Hecke and CoHA actions}.
\end{proof}

Besides the generating series of the creation and annihilation operators, we have two types of generating series of multiplication operators using tautological classes. Define
\begin{align} \label{haction}{}^{\pm}h(w):=\widehat s_{\pm w}(\BK^\pm)\circ \pr_1^!:\BV\rightarrow  \BV\otimes H^*(C)(\!(w^{-1})\!)
\end{align}
where $\pr_1:\Quot\times C\rightarrow \Quot$ is the projection. This defines the operators ${}^{\pm}h_i$ and ${}^{\pm}h^{\alpha}_i$ for all $i\in \BZ$ and $\alpha\in H^*(C)$ via expansion and integration 
$${}^{\pm}h(w)=\sum_{i\in \BZ}{}^{\pm}h_i w^{-i-1},\quad {}^{\pm}h^{\alpha}_i\star x:=\pr_{1*}\big(\pr_2^*\alpha\cdot {}^{\pm}h_i(x)\big). 
$$
Note that 
\begin{equation}\label{eq: vanishing for small enough i}
    \begin{cases}
        {}^+h_i=0,&\textnormal{if}\quad i<\rk(V)-r-1,\\
        {}^-h_i=0,&\textnormal{if}\quad i<r-1,
    \end{cases}
\end{equation}
as operators on $H_*(\Quot_{(r,d)})$ for cohomological degree reasons.

\begin{remark} \label{rem: relation with NM action}
Note that $\hat{s}_{z}(E) = 1/c(E^\vee,z)$ where \[ c(E,z) = \sum_{i=0}^{\mathrm{rk}(E)}(-1)^i z^{\mathrm{rk}(E)-i}c_i(E)\] is the universal Chern series in the notation of Marian--Negu\c t in \cite{marian2026cohomologyquotschemesmooth}. By Remark \ref{rem: explicit K complex}, we have \[ {}^{-}h(z) = (-1)^{r}\frac{c(\mathcal{V},z+c_1(K_C))}{c(\mathcal{S},z+c_1(K_C))}, \quad {}^{+}h(z) = \frac{1}{c(\mathcal{S},z)}.\] In particular when we are working with a punctual Quot scheme, i.e., in the $r=0$ case, we get 
\begin{equation} \label{eqn:relationwithalinanegut}
h^{MN}(z)  = {}^{-}h(z){}^{+}h(z)
\end{equation}
where $h^{MN}(z)$ is the formula in \cite[eq. (16)]{marian2026cohomologyquotschemesmooth}. Note that we have a relation between the ${}^+h$ and ${}^-h$ classes depending on the choice of $V$ and $r$:
\begin{equation}\label{eqn: h+ h- relation}
{}^-h(z-c_1(K_C)) =(-1)^{r} c (V,z) {}^+h(z)  = (-1)^r(z^{\mathrm{rk}(V)}-c_1(V)z^{\mathrm{rk}(V)-1}){}^+h(z).
\end{equation}
\end{remark}

\subsection{$[{}^{\pm}h,e]$, $[f,{}^{\pm}h]$ type commutator}

In this section, we compute the commutator between the creation and annihilation operators and the multiplication operators. We start with the $[{}^{\pm}h,e]$ type, i.e., comparison of 
$${}^{\pm}h(w)\circ e(u):\BV\rightarrow \BV\otimes H^*(C_e\times C_h)(\!(u^{-1},w^{-1})\!)
$$
and 
$$e(u)\circ {}^{\pm}h(w):\BV\rightarrow \BV\otimes H^*(C_e\times C_h)(\!(u^{-1},w^{-1})\!),
$$
where the subscripts on $C_e$ and $C_h$ indicate the factor of the curve added in the creation and multiplication operators, respectively. Consider the diagram for the creation operator with an extra copy of $C_h$
\begin{equation}\label{eq: [h,e] diagram}
\begin{tikzcd}
  & \Quot_{(r, [d,d+1])} \times C_h\arrow[ld, "p_-"'] \arrow[rd, "p_+\times \rho"] &   \\
\Quot_{(r,d)} \times C_h&                                    & \Quot_{(r,d+1)}\times C_e\times C_h.
\end{tikzcd}
\end{equation}
The next lemma compares $\BK^\pm_{(r,d),h}$ and $\BK^\pm_{(r,d+1),h}$ after the appropriate pullback.

\begin{lemma}\label{lem: comparing K complex}
In the $K$-theory of $\Quot_{(r, [d,d+1])} \times C_h$, we have
$$(p_-)^*\BK^\pm_{(r,d),h}-(p_+\times\rho)^*\BK^\pm_{(r,d+1),h}=\pm \CL^{\mp}\otimes(\CO - \CO(\Delta_{eh})),
$$
where $\Delta_{eh}\hookrightarrow C_e\times C_h$ is the diagonal. 
\end{lemma}
\begin{proof}
We first consider the case of $\BK^+$. By definition of the complex $\BK^+$, we have
$$\BK^+_{(r,d),h}-\BK^+_{(r,d+1),h}=\R p_*\RHHom(\CS_{(r,d),\bullet} - \CS_{(r,d+1),\bullet}, \CO_{\Delta_{\bullet h}})
$$
where we omit various pull backs from the notation. On the other hand, we have a short exact sequence
$$0\rightarrow \CS_{(r,d+1),\bullet}\rightarrow \CS_{(r,d),\bullet}\rightarrow \CO_{\Delta_{\bullet e}}\otimes \CL\rightarrow 0.
$$
Therefore, we obtain 
\begin{align*}
\BK^+_{(r,d),h}-\BK^+_{(r,d+1),h}
&=\R p_* \RHHom(\CO_{\Delta_{\bullet e}}\otimes \CL,\CO_{\Delta_{\bullet h}})\\
&=\R p_* \RHHom((\CO-\CO(-\Delta_{\bullet e}))\otimes \CL,\CO_{\Delta_{\bullet h}})\\
&=\CL^{-1}-\CL^{-1}\otimes \CO(\Delta_{eh}).
\end{align*}

Similarly, by definition of the complex $\BK^-$, we have
$$\BK^-_{(r,d),h}-\BK^-_{(r,d+1),h}=\R p_*\RHHom(\CO_{\Delta_{\bullet h}},\CQ_{(r,d),\bullet} - \CQ_{(r,d+1),\bullet}).
$$
On the other hand, we have a short exact sequence
$$0\rightarrow \CO_{\Delta_{\bullet e}}\otimes \CL\rightarrow \CQ_{(r,d+1),\bullet}\rightarrow \CQ_{(r,d),\bullet}\rightarrow 0
$$
Therefore, we obtain 
\begin{align*}
\BK^-_{(r,d),h}-\BK^-_{(r,d+1),h}
&=-\R p_*\RHHom(\CO_{\Delta_{\bullet h}},\CO_{\Delta_{\bullet e}}\otimes \CL)\\
&=-\CL+\CL\otimes \CO(\Delta_{eh}).
\end{align*}
\end{proof}

In the statement below, we denote 
$${}^{+}h_k \circ^\gamma e_\ell:=\sum_{s\in I} {}^{+}h_k^{\gamma_s^L} \circ e_\ell^{\gamma_s^R}
$$
where $\Delta_*\gamma=\sum_{s\in I} \gamma_s^L\otimes \gamma_s^R$. The other expressions
$$e_k^\gamma\circ {}^{-}h_\ell,\quad f_k\circ^\gamma {}^{+}h_\ell,\quad {}^{-}h_k\circ^\gamma f_\ell,\quad {}^{-}h_k\circ^\gamma {}^{+}h_\ell,\quad e_k\circ^\gamma e_\ell,\quad f_k\circ^\gamma f_\ell,\quad \dots $$ 
are defined in the same way. 

\begin{proposition}\label{prop: [e,m]}
As operators $\BV\rightarrow \BV\otimes H^*(C_e\times C_h)(\!(u^{-1},w^{-1})\!)$, we have
\begin{align*}
    [{}^{+}h(w), e(u)]&=\Delta_{eh}\circ {}^{+}h(w)\circ\left(\frac{e(u)-e(w)}{u-w}
\right),\\
    [{}^{-}h(w), e(u)]&=\Delta_{eh}\circ \left(\frac{e(u)-e(w)}{u-w}\right)\circ {}^{-}h(w).
\end{align*}
In particular, as operators $\BV\rightarrow \BV$, we have 
\begin{align*}
    [\+h_i^\alpha,e^\beta_j]=-\big(\+h_{i-1}\circ^{\alpha\cup\beta}e_{j}+\+h_{i-2}\circ^{\alpha\cup\beta}e_{j+1}+\cdots\big),\\
    [{}^{-}h_i^\alpha,e^\beta_j]=-\big(e_{j}\circ^{\alpha\cup\beta}{}^{-}h_{i-1}   +e_{j+1}\circ^{\alpha\cup\beta}{}^{-}h_{i-2}+\cdots\big).
\end{align*}
\end{proposition}

\begin{proof}
Let $x\in \BV_{(r,d)}$. Recall the diagram \eqref{eq: [h,e] diagram}. By definition of the operators, we have
\begin{align*}
    e(u)\circ {}^{+}h(w)(x)
    &=e(u)\left(s_w(\BK^+_{(r,d),h})\cdot (\pr_1)^!x\right)\\
    &=(p_+\times \rho)_*\Big(\frac{1}{u-z}\cdot (p_-)^!(\widehat s_w(\BK^+_{(r,d),h})\cdot (\pr_1)^!x)\Big)\\
    &=(p_+\times \rho)_*\Big(\frac{1}{u-z}\cdot (p_-)^*(\widehat s_w(\BK^+_{(r,d),h}))\cdot (p_-)^!x)\Big)
\end{align*}
where we used (A13) and $(p_-)^!(\pr_1)^!=(p_-)^!$ for the last equality. Using Lemma \ref{lem: comparing K complex}, we can further simplify this to 
\begin{align*}
    e(u)\circ {}^{+}h(w)(x)
    &=\widehat s_w(\BK^+_{(r,d+1),h})\cdot (p_+\times \rho)_*\Big(\frac{1}{u-z}\cdot \widehat s_w(\CL^{-1} - \CL^{-1}(\Delta_{eh}))\cdot  (p_-)^!(x)\Big)\\
    &=\widehat s_w(\BK^+_{(r,d+1),h})\cdot (p_+\times \rho)_*\left(\frac{1}{u-z}\cdot \frac{w- z+\Delta_{eh}}{w- z} \cdot (p_-)^!(x)\right).
\end{align*}
Using the identity
\begin{align*}
    \frac{1}{u-z}\cdot \frac{w- z+\Delta_{eh}}{w- z}=\frac{1}{u-z}-\Delta_{eh}\cdot \left(\frac{\frac{1}{u-z}-\frac{1}{w-z}}{u-w}\right), 
\end{align*}
we obtain
$$[{}^{+}h(w), e(u)](x)=\Delta_{eh}\circ {}^{+}h(w)\circ\left(\frac{e(u)-e(w)}{u-w}
\right)(x).
$$

Similarly, we have
\begin{align*}
{}^{-}h(w)\circ e(u) (x)
&=\widehat s_{-w}(\BK^-_{(r,d+1),h})\cdot (p_+\times \rho)_*\Big(
\frac{1}{u-z}\cdot (p_-)^!(x)\Big)\\
&=(p_+\times \rho)_*\Big(
\widehat s_{-w}(\CL-\CL(\Delta_{eh}))\cdot
\frac{1}{u-z}\cdot (p_-)^*(s_{-w}(\BK^-_{(r,d),h}))\cdot (p_-)^!x)\Big)\\
&=(p_+\times \rho)_*\Big(
\frac{w-z-\Delta_{eh}}{w-z}
\cdot \frac{1}{u-z}\cdot (p_-)^!(\widehat s_{-w}(\BK^-_{(r,d),h})\cdot x)\Big)
\end{align*}
which can be rewritten as 
$$[{}^{-}h(w), e(u)](x)=\Delta_{eh}\circ \left(\frac{e(u)-e(w)}{u-w}\right)\circ {}^{-}h(w)(x).
$$

\end{proof}

By the same computation, one can compute the $[f,{}^{\pm}h]$ case. 

\begin{proposition}\label{prop: [f,m]}
As operators $\BV\rightarrow \BV\otimes H^*(C_f\times C_h)(\!(v^{-1},w^{-1})\!)$, we have
\begin{align*}
    [f(v),{}^{+}h(w)]&=\Delta_{f,h}\circ \left(\frac{f(v)-f(w)}{v-w}
\right)\circ {}^{+}h(w),\\
    [f(v), {}^{-}h(w)]&=\Delta_{f,h}\circ {}^{-}h(w)\circ \left(\frac{f(v)-f(w)}{v-w}\right).
\end{align*}
In particular, as operators $\BV\rightarrow \BV$, we have 
\begin{align*}
    [f^\alpha_j,\+h_i^\beta]=-\big(f_{j}\circ^{\alpha\cup\beta}{}^{+}h_{i-1}   +f_{j+1}\circ^{\alpha\cup\beta}{}^{+}h_{i-2}+\cdots\big),\\
    [f^\alpha_j,{}^{-}h_i^\beta]=-\big({}^{-}h_{i-1}\circ^{\alpha\cup\beta}f_{j}+{}^{-}h_{i-2}\circ^{\alpha\cup\beta}f_{j+1}+\cdots\big).
\end{align*}
\end{proposition}

\begin{proof}
The proof is identical to that of Proposition \ref{prop: [e,m]}.
\end{proof}

\subsection{$[e,f]$ type commutator}

Unlike the commutators of type $[{}^\pm h, e],\ [f,{}^\pm h]$, we prove the $[e,f]$ type commutators only after restricting to the virtual homology groups. It would be interesting to determine whether the same $[e,f]$ type commutator formula holds on the full homology groups. 

\begin{theorem} \label{thm: [e,f]}
    As operators $\BV^\vir\rightarrow \BV^\vir\otimes H^*(C_e\times C_f)(\!(u^{-1},v^{-1})\!)$, we have 
    $$[e(u), f(v)]=-\Delta_{e,f}\circ \Res_w\left(
    \frac{{}^{-}h(w)}{u-w}\circ \frac{{}^{+}h(w)}{v-w}\right).
    $$
    In particular, as operators $\BV^\vir\rightarrow \BV^\vir$, we have
    $$[e^\alpha_i, f^\beta_j]=-\hspace{-10pt}\sum_{\substack{n,m\in \BZ\\ n+m=i+j-1}} {}^{-}h_n\circ^{\alpha\cup \beta}{}^{+}h_m.
    $$
\end{theorem}

\begin{remark}
Let $F(w)=\sum_{k\in \BZ} a_kw^{-k-1}\in R(\!(w^{-1})\!)$ be a formal series with coefficients in a ring $R$ such that $a_k=0$ for small enough $k\in \BZ$. If $u$ and $v$ are formal variables, then we have
\begin{align*}
    \Res_w \left(\frac{F(w)}{(u-w)(v-w)}\right)
    &=\Res_w \left(\sum_{i, j\geq 0,\ k\in \BZ}a_k\cdot u^{-i-1}v^{-j-1}w^{-k-1+i+j}\right)\\
    &=\sum_{i,j\geq 0} a_{i+j}\cdot u^{-i-1}v^{-j-1}\\
    &=-\frac{F_{<0}(u)-F_{<0}(v)}{u-v}
\end{align*}
where $F_{<0}(w):=\sum_{k=0}^\infty a_k w^{-k-1}$. Using this notation, we may rewrite the above theorem as 
$$[e(u), f(v)]=\Delta_{e,f}\circ\frac{H_{<0}(u)-H_{<0}(v)}{u-v},
$$
where $H(w):={}^- h(w)\circ {}^+ h(w)$. For punctual Quot schemes, $H(w)_{<0}=h^{MN}(w)$ by Remark \ref{rem: relation with NM action}, so this precisely matches with the formula \cite[eq. (23)]{marian2026cohomologyquotschemesmooth}.

\end{remark}

\begin{proof}
Recall that $\BV^\vir_{(r,d)}\subseteq H_*(\Quot_{(r,d)})$ is generated by the virtual fundamental class $1_{(r,d)}$ via tautological classes. Since tautological classes are K\"unneth components of the multiplication operator ${}^{+}h(w)$, it suffices to show that 
\begin{enumerate}
    \item [i)] $\textnormal{(LHS)}1_{(r,d)}=\textnormal{(RHS)}1_{(r,d)}$,
    \item [ii)] $\big[\textnormal{(LHS)}, \prod_{i\in I} {}^{+}h(w_i)\big]1_{(r,d)}=\big[\textnormal{(RHS)}, \prod_{i\in I} {}^{+}h(w_i)\big]1_{(r,d)}$
\end{enumerate}
where (LHS) and (RHS) denote the operator on the left and right hand side of the theorem, respectively. Since (RHS) only involves multiplication operators, the part ii) amounts to showing 
$$\big[\textnormal{(LHS)}, \prod_{i\in I} {}^{+}h(w_i)\big]1_{(r,d)}=0.$$ 

We start with the part i). In order to compute $e(u)\circ f(v)1_{(r,d)}$, we use the diagram
\begin{center}
    \begin{tikzcd}[column sep=11pt]
  & \Quot_{(r,[d-1,d])}\arrow[ld, "p_+^f"'] \arrow[rd, "p_-^f\times \rho^f=\pi_+"] &    & \Quot_{(r,[d-1,d])}\times C_f \arrow[ld, "p_-^e"'] \arrow[rd, "p_+^e\times \rho^e=\pi_-"] &     \\
\Quot_{(r,d)} &                                    & \Quot_{(r,d-1)}\times C_f &                                     & \Quot_{(r,d)}\times C_e\times C_f.
\end{tikzcd}
\end{center}
By Proposition \ref{prop: virtual projective bundle}, $\pi_+$ (resp. $\pi_-$) is identified with a derived projective bundle of the dual of $\BK^{+}_{(r,d-1), f}$ (resp. $\BK^-_{(r,d),e}$) and we write $z_f:=c_1(\CO(1))$ (resp. $z_e:=c_1(\CO(1))$). Note that $z_f=z$ and $z_e=-z$ where $z=c_1(\CL)$. By repeatedly using Proposition \ref{prop: virtual projective bundle}, Lemma \ref{lem: virtual pushforward} and Lemma \ref{lem: comparing K complex}, we obtain\footnote{Even though we only specify cohomology classes in the middle step, we actually cap it with the virtual fundamental class to obtain (Borel--Moore) homology class. These are omitted for simplicity of the notation.} 
\begin{align*}
    e(u)\circ f(v)\, 1_{(r,d)}
    &=e(u)\ (\pi_+)_*\left(\frac{1}{v-z_f}
    \right)\\
    &=e(u)\ \Res_{z_f}\left(\frac{\widehat s_{z_f}(\BK^+_{(r,d-1),f})}{v-z_f}
    \right)\\
    &=(\pi_-)_*\left(\frac{1}{u+z_e}\cdot 
    \Res_{z_f}\left((p_-^e)^*\frac{\widehat s_{z_f}(\BK^+_{(r,d-1),f})}{v-z_f}\right)
    \right)\\
    &=(\pi_-)_*\left(\frac{1}{u+z_e}\cdot 
    \Res_{z_f}\left(\widehat s_{z_f}(\CO(z_e)-\CO(z_e+\Delta_{e,f}))\cdot (\pi_-)^*\frac{\widehat s_{z_f}(\BK^+_{(r,d),f})}{v-z_f}\right)
    \right)\\
    &=(\pi_-)_*\left(\frac{1}{u+z_e}\cdot 
    \Res_{z_f}\left(\frac{z_f+z_e+\Delta_{e,f}}{z_f+z_e}\cdot (\pi_-)^*\frac{\widehat s_{z_f}(\BK^+_{(r,d),f})}{v-z_f}\right)
    \right)\\
    &=\Res_{z_e}\Res_{z_f}\left(
    \frac{\widehat s_{z_e}(\BK^-_{(r,d),e})}{u+z_e}\cdot \frac{\widehat s_{z_f}(\BK^+_{(r,d),f})}{v-z_f}\cdot \frac{z_f+z_e+\Delta_{e,f}}{z_f+z_e}
    \right)1_{(r,d)}.
\end{align*}
In the third equality we have $1/(u+z_e)$ because $z_e=-z$. The same computation shows that $f(v)\circ e(u)\, 1_{(r,d)}$ is equal to the above expression except for the order of taking residues. In other words, we have
\begin{align*}
     [e(u), f(v)] 1_{(r,d)} 
     &= [\Res_{z_e},\Res_{z_f}]\left(
    \frac{\widehat s_{z_e}(\BK^-_{(r,d),e})}{u+z_e}\cdot \frac{\widehat s_{z_f}(\BK^+_{(r,d),f})}{v-z_f}\cdot \frac{z_f+z_e+\Delta_{e,f}}{z_f+z_e}
    \right)1_{(r,d)}\\
    &= - \Delta_{e,f}\circ \Res_w\left(
    \frac{{}^{-}h(w)}{u-w}\cdot \frac{{}^{+}h(w)}{v-w}
    \right)1_{(r,d)}
\end{align*}
where the second equality follows from Lemma \ref{lem: commutator for the residue} below and definition of the multiplication operators ${}^{\pm}h(w)$.

Now we show the part ii). For readability, we write $\BK^\pm_{(r,d),*}=\BK^\pm_{d,*}$ and omit various pull backs below. Set $(\dagger):=e(u)\circ f(v)\circ \prod_{i\in I} {}^{+}h(w_i)\, 1_{(r,d)}$. Similar to the previous computation, we have
\begin{align*}
    (\dagger)
    &=e(u)\,(\pi_+)_*\left(\frac{1}{v-z_f}\cdot \prod_{i\in I} \widehat s_{w_i}(\BK^{+}_{d,m_i})
    \right)\\
    &=e(u)\,(\pi_+)_*\left(\frac{1}{v-z_f}\cdot \prod_{i\in I} 
    \widehat s_{w_i}(-\CO(-z_f)+\CO(-z_f+\Delta_{m_i,f})+\BK^{+}_{d-1,m_i})
    \right)\\
    &=e(u)\,(\pi_+)_*\left(\frac{1}{v-z_f}\cdot \prod_{i\in I} 
    \frac{w_i-z_f}{w_i-z_f+\Delta_{m_i,f}}
    \cdot \widehat s_{w_i}(\BK^{+}_{d-1,m_i})
    \right)\\
    &=e(u)\,\Res_{z_f}\left(\frac{\widehat s_{z_f}(\BK^+_{d-1,f})}{v-z_f}\cdot \prod_{i\in I} 
    \frac{w_i-z_f}{w_i-z_f+\Delta_{m_i,f}}
    \cdot \widehat s_{w_i}(\BK^{+}_{d-1,m_i})
    \right)\\
    &=(\pi_-)_*\left(
    \frac{1}{u+z_e}
    \cdot \Res_{z_f}\left(\frac{\widehat s_{z_f}(\BK^+_{d-1,f})}{v-z_f}\cdot \prod_{i\in I} 
    \frac{w_i-z_f}{w_i-z_f+\Delta_{m_i,f}}
    \cdot \widehat s_{w_i}(\BK^{+}_{d-1,m_i})
    \right)
    \right).
\end{align*}
If we replace $\BK^+_{d-1,*}$ by $\BK^+_{d,*}$ using Lemma \ref{lem: comparing K complex} again, then we obtain that $(\dagger)$ equals to 
$$\Res_{z_e}\Res_{z_f}\left(
    \frac{\widehat s_{z_e}(\BK^-_{d,e})}{u+z_e}\cdot \frac{\widehat s_{z_f}(\BK^+_{d,f})}{v-z_f}\cdot \frac{z_f+z_e+\Delta_{e,f}}{z_f+z_e}\cdot 
    \prod_{i\in I} 
    \frac{w_i+z_e+\Delta_{e,m_i}}{w_i+z_e}\frac{w_i-z_f}{w_i-z_f+\Delta_{m_i,f}}
    \cdot \widehat s_{w_i}(\BK^{+}_{d,m_i})
    \right).
$$
The same computation shows that $f(v)\circ e(u)\circ \prod_{i\in I} {}^{+}h(w_i)\, 1_{(r,d)}$ is the same except for the order of taking residues. Therefore, Lemma \ref{lem: commutator for the residue} implies that
\begin{multline*}
    [e(u),f(v)]\circ\prod_{i\in I}{}^{+}h(w_i)1_{(r,d)}\\
    =\Res_w\left(
    \frac{{}^{-}h_e(w)}{u-w}\frac{{}^{+}h_f(w)}{v-w}\cdot \Delta_{e,f}\cdot \prod_{i\in I} 
    \frac{w_i-w+\Delta_{e,m_i}}{w_i-w}\frac{w_i-w}{w_i-w+\Delta_{m_i,f}}
    \cdot \widehat s_{w_i}(\BK^{+}_{d,m_i})
    \right)
\end{multline*}
Since $\Delta_{e,f}\cdot \Delta_{e,m_i}=\Delta_{e,f}\cdot \Delta_{m_i,f}=\Delta_{e,f,m_i}$, 
$$\Delta_{e,f}\cdot \prod_{i\in I} 
    \frac{w_i-w+\Delta_{e,m_i}}{w_i-w}\frac{w_i-w}{w_i-w+\Delta_{m_i,f}}=\Delta_{e,f}.
$$
Therefore, we can further simplify the computation to 
$$\Res_w\left(
    \frac{{}^{-}h_e(w)}{u-w}\frac{{}^{+}h_f(w)}{v-w}\cdot \Delta_{e,f}\cdot \prod_{i\in I}\widehat s_{w_i}(\BK^{+}_{d,m_i})
    \right)=\prod_{i\in I}{}^{+}h(w_i)\circ [e(u),f(v)]1_{(r,d)}. 
$$
This completes the proof.
\end{proof}
\begin{lemma}\label{lem: commutator for the residue}
Let $R$ be a ring and consider 
$$F(x,y)\in R(\!(x^{-1},y^{-1})\!). 
$$
Then 
$$[\Res_{x},\Res_{y}]\frac{F(x,y)}{x+y}=- \Res_w F(-w,w).
$$
\end{lemma}
\begin{proof}
    Let $F(x,y)=\sum_{i,j}a_{i,j}x^iy^j$ such that there exists $N$ such that $a_{i,j}=0$ if either $i>N$ or $j>N$. Using the expansion 
    \begin{align*}
        \frac{1}{x+y}=\sum_{n=0}^\infty (-1)^n x^{n}y^{-n-1},
    \end{align*}
 we compute the iterated residue as
 \begin{align*}
     \Res_x\circ  \Res_y\frac{F(x,y)}{x+y}
     &=\Res_x \circ \Res_y \sum_{n=0}^\infty (-1)^n x^{n}y^{-n-1} F(x,y)\\ 
     &=\Res_x \sum_{n=0}^\infty\sum_{i\leq N} (-1)^n a_{i,n} x^{i+n}\\
     &=\sum_{n=0}^\infty (-1)^n a_{-n-1, n}. 
 \end{align*}
Similarly, we have
$$\Res_y\circ  \Res_x\frac{F(x,y)}{x+y}=\sum_{n=0}^\infty (-1)^n a_{n, -n-1}. 
$$
Therefore, 
\begin{align*}
    [\Res_x,\Res_y]\frac{F(x,y)}{x+y}
    &=\sum_{n\in\BZ}^\infty (-1)^n a_{-n-1, n}
\end{align*}
which is a finite sum since $a_{i,j}=0$ if $i>N$ or $j>N$. This matches precisely with 
\begin{align*}
    - \Res_w F(-w,w)&=- \Res_w\sum_{i\leq N} \sum_{j\leq N}(-1)^i a_{i,j} w^{i+j}.
\end{align*}

\end{proof}

\subsection{$[e,e], [f,f]$ type commutator}

In Theorem \ref{relations:torsioncoha}, we computed the $[e,e]$ type commutators on the level of CoHA. This would imply the same commutator formula for any $\BH$-modules. In what follows, we reprove such a commutator formula on the level of CoHA-module $\BV^\vir$, using the derived projective bundle method. 

\begin{theorem} \label{thm: [e,e]}
    As operators $\BV^\vir\rightarrow \BV^\vir\otimes H^*(C_1\times C_2)(\!(u^{-1},v^{-1})\!)$, we have
    $$\frac{[e(u), e(v)]}{u-v}=\Delta_{1,2}\circ \left(\frac{e(u)-e(v)}{u-v}\right)\circ \left(\frac{e(u)-e(v)}{u-v}\right).
    $$
    In particular, as operators $\BV^\vir\rightarrow \BV^\vir$, we have
    $$[e^\alpha_i, e^\beta_j]=-\big(e_{i-1}\circ^{\alpha\cup \beta}e_{j}+e_{i-2}\circ^{\alpha\cup \beta}e_{j+1}+\cdots + e_{j}\circ^{\alpha\cup \beta}e_{i-1}\big),\quad\textnormal{if}\quad i\geq j\geq 0.
    $$
    
\end{theorem}
\begin{proof}
The proof strategy is the same as for Theorem \ref{thm: [e,f]}. We first compare both sides applied to the virtual fundamental class $1_{(r,d)}$. In order to compute $e(u)\circ e(v)1_{(r,d)}$, we use the diagram
\begin{center}
    \begin{tikzcd}[column sep=9pt]
  & \Quot_{(r,[d,d+1])}\arrow[ld, ""'] \arrow[rd, "\pi_-^{(2)}"] &    & \Quot_{(r,[d+1,d+2])}\times C_2 \arrow[ld, ""'] \arrow[rd, "\pi_-^{(1)}"] &     \\
\Quot_{(r,d)} &                                    & \Quot_{(r,d+1)}\times C_2 &                                     & \Quot_{(r,d+2)}\times C_1\times C_2.
\end{tikzcd}
\end{center}
By repeatedly using Proposition \ref{prop: virtual projective bundle}, Lemma \ref{lem: virtual pushforward} and Lemma \ref{lem: comparing K complex}, we obtain
\begin{align}\label{eq: e,e first}
    e(u)\circ e(v)\, 1_{(r,d)}
    &=e(u)\ \Res_{z_2}\left(\frac{\widehat s_{z_2}(\BK^-_{(r,d+1),2})}{v+z_2}
    \right)\\
    &=(\pi_-^{(1)})_*\left(\frac{1}{u+z_1}\cdot 
    \Res_{z_2}\left(\frac{z_2-z_1}{z_2-z_1+\Delta_{1,2}}\cdot \frac{\widehat s_{z_2}(\BK^-_{(r,d+2),2})}{v+z_2}\right)
    \right)\notag\\
    &=\Res_{z_1}\Res_{z_2}\left(\frac{\widehat s_{z_1}(\BK^-_{(r,d+2),1})}{u+z_1}\cdot \frac{\widehat s_{z_2}(\BK^-_{(r,d+2),2})}{v+z_2}\cdot \frac{z_2-z_1}{z_2-z_1+\Delta_{1,2}}
    \right)\notag\\
    &=\Res_{z_1}\Res_{z_2}\left(\frac{{}^{-}h_1(-z_1)}{u+z_1}\cdot \frac{{}^{-}h_2(-z_2)}{v+z_2}
    \right)\notag\\
    &\quad-\Delta_{1,2}\circ\Res_{z_1}\Res_{z_2}\left(\frac{{}^{-}h_1(-z_1)}{u+z_1}\cdot \frac{{}^{-}h_2(-z_2)}{v+z_2}\cdot \frac{1}{z_2-z_1+\Delta_{1,2}}
    \right).\notag
\end{align}
By replacing $u$ by $v$ and $C_1$ by $C_2$, we obtain 
\begin{align}\label{eq: e,e second}
    \sigma_{12}\circ e(v)\circ e(u)\, 1_{(r,d)}
    &=\Res_{z_1}\Res_{z_2}\left(\frac{{}^{-}h_2(-z_1)}{v+z_1}\cdot \frac{{}^{-}h_1(-z_2)}{u+z_2}
    \right)\\
    &\quad -\Delta_{1,2}\circ\Res_{z_1}\Res_{z_2}\left(\frac{{}^{-}h_2(-z_1)}{v+z_1}\cdot \frac{{}^{-}h_1(-z_2)}{u+z_2}\cdot \frac{1}{z_2-z_1+\Delta_{1,2}}
    \right)\notag\\
    &=\Res_{z_2}\Res_{z_1}\left(\frac{{}^{-}h_2(-z_2)}{v+z_2}\cdot \frac{{}^{-}h_1(-z_1)}{u+z_1}
    \right)\notag\\
    &\quad -\Delta_{1,2}\circ\Res_{z_1}\Res_{z_2}\left(\frac{{}^{-}h_1(-z_1)}{v+z_1}\cdot \frac{{}^{-}h_2(-z_2)}{u+z_2}\cdot \frac{1}{z_2-z_1+\Delta_{1,2}}
    \right).\notag
\end{align}
In the second equality, we simply renamed $z_1$ by $z_2$ and vice versa in the first term and swapped $C_1$ and $C_2$ in the second term without changing its value since $\Delta_{1,2}$ is multiplied. Note that the first term of \eqref{eq: e,e first} matches the first term of \eqref{eq: e,e second} since the order of taking residues does not affect its value when applied to the multiplication of one variable functions. Therefore, by subtracting \eqref{eq: e,e second} from \eqref{eq: e,e first}, we obtain 
\begin{align*}
    [e(u),e(v)]1_{(r,d)}
    &=-\Delta_{1,2}\circ\Res_{z_1}\Res_{z_2}\left(
    \frac{{}^{-}h_1(-z_1)\cdot {}^{-}h_2(-z_2)}{z_2-z_1+\Delta_{1,2}}\cdot \left(
    \frac{1}{(u+z_1)(v+z_2)}-\frac{1}{(u+z_2)(v+z_1)}
    \right)
    \right).
\end{align*}
Since 
$$\frac{1}{(u+z_1)(v+z_2)}-\frac{1}{(u+z_2)(v+z_1)}=\frac{-(z_2-z_1)(u-v)}{(u+z_1)(v+z_2)(u+z_2)(v+z_1)},
$$
we have
$$\frac{[e(u),e(v)]}{u-v}1_{(r,d)}=\Delta_{1,2}\circ\Res_{z_1}\Res_{z_2}\left(
    \frac{{}^{-}h_1(-z_1)\cdot {}^{-}h_2(-z_2)}{(u+z_1)(v+z_2)(u+z_2)(v+z_1)}\cdot \frac{z_2-z_1}{z_2-z_1+\Delta_{1,2}}
    \right).
$$

On the other hand, we have
\begin{align*}
    \left(
    \frac{e(u)-e(v)}{u-v}
    \right)1_{(r,d)}
    &=\Res_{z_2}\left(\frac{\frac{1}{u+z_2}-\frac{1}{v+z_2}}{u-v}\cdot\widehat s_{z_2}(\BK^-_{(r,d+1),2})
    \right)\\
    &=-\Res_{z_2}\left(\frac{1}{(u+z_2)(v+z_2)}\cdot\widehat s_{z_2}(\BK^-_{(r,d+1),2})
    \right).
\end{align*}
By applying $\frac{e(u)-e(v)}{u-v}$ again to this expression in a similar way, we get
$$\left(
    \frac{e(u)-e(v)}{u-v}
    \right)\circ\left(
    \frac{e(u)-e(v)}{u-v}
    \right)1_{(r,d)}
    =\Res_{z_1}\Res_{z_2}\left(
    \frac{{}^{-}h_1(-z_1)\cdot {}^{-}h_2(-z_2)}{(u+z_1)(v+z_2)(u+z_2)(v+z_1)}\cdot \frac{z_2-z_1}{z_2-z_1+\Delta_{1,2}}
    \right).
$$
We remark that a routine computation, as done in the proof of Theorem \ref{thm: [e,f]} and Theorem \ref{thm: [e,e]}, shows that 
$$\left[\frac{[e(u),e(v)]}{u-v}, \prod_{i\in I}{}^{+}h(w_i)\right]1_{(r,d)}
=\left[ \Delta_{1,2}\circ \left(\frac{e(u)-e(v)}{u-v}
    \right)\circ\left(
    \frac{e(u)-e(v)}{u-v}
    \right) , \prod_{i\in I}{}^{+}h(w_i)\right]1_{(r,d)},
$$
completing the proof.
\end{proof}

\begin{theorem} \label{thm: [f,f]}
    As operators $\BV^\vir\rightarrow \BV^\vir\otimes H^*(C_1\times C_2)(\!(u^{-1},v^{-1})\!)$, we have
    $$\frac{[f(u), f(v)]}{u-v}=-\Delta_{1,2}\circ \left(\frac{f(u)-f(v)}{u-v}\right)\circ \left(\frac{f(u)-f(v)}{u-v}\right).
    $$
    In particular, as operators $\BV^\vir\rightarrow \BV^\vir$, we have
    $$[f^\alpha_i, f^\beta_j]=f_{i-1}\circ^{\alpha\cup \beta}f_{j}+f_{i-2}\circ^{\alpha\cup \beta}f_{j+1}+\cdots + f_{j}\circ^{\alpha\cup \beta}f_{i-1},\quad\textnormal{if}\quad i\geq j\geq 0.
    $$
\end{theorem}
\begin{proof}
    The proof is essentially the same as that of Theorem \ref{thm: [e,e]}. Alternatively, one can use the fact that annihilation operators satisfy the commutator formula for the opposite algebra $\BH^\op$, hence the same commutator formula as in Theorem \ref{thm: [e,e]} up to a sign. 
\end{proof}

\section{Double of torsion CoHA}\label{sec: Drinfeld double}

In this section, we define a double of the torsion CoHA, denoted by $D(\BH)$, and its action $D(\BH)\curvearrowright\BV^\vir$ with the following properties:
\begin{enumerate}
    \item [i)] $\BH$ and $\BH^{\op}$ are subalgebras of $D(\BH)$,
    \item [ii)] the action $D(\BH)\curvearrowright\BV^\vir$ extends the creation (resp. annihilation) action of $\BH$ (resp. $\BH^{\op}$). 
\end{enumerate}
For this section, we assume that $C$ is a smooth projective curve and fix a basis $\CB$ of $H^*(C)$. 

\subsection{The completed tensor product}

In this subsection, we define the {\it underlying vector space} of the doubled torsion CoHA as a certain completed tensor product. Before getting into the details, we first explain the motivation behind its definition. 

On the virtual homology group, we have constructed a creation action $\BH\curvearrowright \BV^\vir$ and an annihilation action $\BH^\op\curvearrowright\BV^\vir$. It is natural to ask whether there is an action of a bigger algebra, say $D(\BH)\curvearrowright \BV^\vir$, which contains $\BH,\ \BH^\op\subseteq D(\BH)$ as subalgebras and extends the creation and annihilation action. If such a structure exists, then $D(\BH)$ must contain elements of the form $[e_i^\alpha, f_j^\beta]$ which act on $\BV^\vir$ according to the formula 
\begin{equation}\label{eq: [e,f] again}
    [e^\alpha_i, f^\beta_j]=-\hspace{-10pt}\sum_{\substack{n,m\in \BZ\\ n+m=i+j-1}} {}^{-}h_n\circ^{\alpha\cup \beta}{}^{+}h_m
\end{equation}
proven in Theorem \ref{thm: [e,f]}. Since this formula involves the multiplication operators, it is natural to define $D(\BH)$ such that it contains mutually supercommuting elements of the form ${}^-h_n^\alpha$ and ${}^+h_m^\beta$ for all $n, m\in \BZ$ and $\alpha, \beta\in \CB\subseteq H^*(C)$. One subtlety is that the right hand side of \eqref{eq: [e,f] again} involves an infinite expression, so $D(\BH)$ should be completed with respect to these supercommuting elements. We note that completions are ubiquitous in the construction of coproducts and Drinfeld doubles; see, for instance, \cite{Schiffmann, YZ, N4}. 

We first define the positive and negative Cartan parts of the doubled algebra.
\begin{definition}\label{def: Cartan}
For each $a\in \BZ$, the $a$-truncated positive Cartan part is defined as a free supercommutative algebra
$$\BH^{0+}_{\geq a}:=\BQ[{}^+h_n^{\alpha}\,|\,n\geq a,\ \alpha\in \CB]. 
$$
The positive Cartan part is defined as the inverse limit
$$\BH^{0+}:=\lim_{\longleftarrow}\Big(\cdots\,\twoheadrightarrow\, \BH^{0+}_{\geq a}\overset{\ \pi_a\ }{\twoheadrightarrow} \BH^{0+}_{\geq a+1}\,\twoheadrightarrow\,\cdots\Big),
$$
where $\pi_a$ maps the generators ${}^+h_n^{\alpha}$ to the corresponding generator if $n\geq a+1$ and zero if $n=a$. The negative analogues $\BH^{0-}_{\geq a}$ and $\BH^{0-}$ are defined in the same way using negative symbols ${}^-h_n^\alpha$. 
\end{definition}

\begin{remark}\label{rem: extend beyond the basis}
    If $\gamma=\sum_{\alpha\in \CB}c_\alpha\cdot \alpha\in H^*(C)$ for uniquely defined $c_\alpha\in \BQ$, then we write 
    $${}^\pm h_n^\gamma:=\sum_{\alpha\in \CB} c_\alpha\cdot {}^\pm h_n^\alpha\in \BH^{0\pm}.$$ 
\end{remark}

\begin{example}\label{ex: a-truncation}
If $x\in \BH^{0\pm}$ is a possibly infinite expression in symbols ${}^\pm h_n^\alpha$, then its image under $\BH^{0\pm}\twoheadrightarrow \BH^{0\pm}_{\geq a}$ is obtained from $x$ by setting ${}^\pm h_{n}^\alpha=0$ for all $n<a$. We call such an element the $a$-truncation of $x$ and denote it by $x_{\geq a}\in \BH^{0\pm}_{\geq a}$. An element of $\BH^{0\pm}$ can be thought of as an expression $x$ in symbols ${}^\pm h_n^\alpha$ such that its $a$-truncation $x_{\geq a}$ is a polynomial for every $a\in \BZ$. For example, the first expression below defines an element in $\BH^{0\pm}$, but the second one does not:
$$\sum_{i+j=N} {}^\pm h_i^\alpha\star {}^\pm h_j^\beta,\quad\quad \sum_{i\in \BZ}{}^\pm h_i^\alpha. 
$$
\end{example}

Recall that ${}^{\pm}h_n\big|_{\Quot_\alpha}=0$ if $n$ is sufficiently small as explained in \eqref{eq: vanishing for small enough i}. This vanishing makes the definition below well-defined, i.e., independent of the choice of a small enough $a$.

\begin{definition}\label{def: Cartan part action via truncation}
    If $x\in \BH^{0\pm}$ and $y\in \BV^\vir$, then the multiplication action $\BH^{0\pm}\curvearrowright \BV^\vir$ is defined by
    $$x\star y:=x_{\geq a}\star y,
    $$
    where $a$ is chosen small enough according to $y\in \BV^\vir=\bigoplus H^\vir_*(\Quot_\alpha)$. 
\end{definition}

Roughly speaking, the underlying vector space of the doubled torsion CoHA is the tensor product of four algebras $\BH^\op$, $\BH^{0-}$, $\BH^{0+}$ and $\BH$. However, a naive tensor product does not produce the correct answer because it does not contain the right hand side of \eqref{eq: [e,f] again}, i.e., 
$$\sum_{\substack{n,m\in \BZ\\ n+m=i+j-1}} {}^{-}h_n\otimes^{\alpha\cup \beta}{}^{+}h_m
\notin \BH^{\op}\otimes \BH^{0-}\otimes \BH^{0+}\otimes\BH. 
$$
Since $\BH^{0\pm}$ are defined as an inverse limit, it is natural to define $D(\BH)$ as a certain inverse limit.  

\begin{definition}\label{def: double}
    The double of the torsion CoHA, as a vector space, is defined as the inverse limit
    $$D(\BH):=\lim_{\longleftarrow} \Big(\BH^{\op}\otimes \BH^{0-}_{\geq a}\otimes \BH^{0+}_{\geq b}\otimes \BH\Big)
    $$
    taken over the directed system $\BZ\times \BZ$ with $(a,b)\geq (a',b')$ if and only if $a\leq a'$ and $b\leq b'$.
\end{definition}

\begin{example}
If $x\in D(\BH)$, then its image under $D(\BH)\twoheadrightarrow \BH^{\op}\otimes \BH^{0-}_{\geq a}\otimes \BH^{0+}_{\geq b}\otimes \BH$ is denoted by $x_{\geq a,\geq b}$ and called the $(a,b)$-truncation of $x$. Elements in $D(\BH)$ are possibly infinite expressions whose $(a,b)$-truncation $x_{\geq a,\geq b}$ lies in $\BH^{\op}\otimes \BH^{0-}_{\geq a}\otimes \BH^{0+}_{\geq b}\otimes \BH$ for every $(a,b)$. For example, 
    $$\sum_{\substack{n,m\in \BZ\\ n+m=i+j-1}} {}^{-}h_n\otimes^{\alpha\cup \beta}{}^{+}h_m\in D(\BH)
$$
defines an element in $D(\BH)$. We also note that the right hand side of the formulas appearing in Proposition \ref{prop: [e,m]}, \ref{prop: [f,m]} also lie in $D(\BH)$. For example, 
$$-\big(\+h_{i-1}\otimes^{\alpha\cup\beta}e_{j}+\+h_{i-2}\otimes^{\alpha\cup\beta}e_{j+1}+\cdots\big)\in D(\BH). 
$$
\end{example}

\begin{remark}
The underlying vector space of $D(\BH)$ as in Definition \ref{def: double} is a completed tensor product 
$$D(\BH)=\BH^{\textnormal{op}}\ \widehat \otimes\ \BH^{0-}\ \widehat \otimes\ \BH^{0+}\ \widehat \otimes\ \BH
$$
in the sense of \cite[Tag 07E7, Tag 0AMQ]{stacks-project}. To make this precise, we need to endow each factor above with a structure of a linearly topologized $\BQ$-algebra. We consider the discrete topology for $\BH^\op$ and $\BH$. The positive and negative Cartan parts $\BH^{0\pm}$ are linearly topologized by a set of ideals 
$$\{I_{<a}^\pm\,|\,a\in \BZ\},\quad I_{<a}^\pm:=\ker\big(\BH^{0\pm}\twoheadrightarrow \BH^{\pm}_{\geq a}\big).
$$
    
\end{remark}

\subsection{The PBW theorem and algebra structure}

In this section, we endow $D(\BH)$ with an algebra structure such that $\BH^\op,\ \BH^{0-},\ \BH^{0+},\ \BH\subseteq D(\BH)$ are subalgebras, following the strategy summarized next. As a first step, we define an algebra $D(\BH)_{a,b}$ presented by generators and relations for each $(a,b)\in \BZ\times\BZ$. The second step is to prove a PBW type theorem giving an isomorphism of vector spaces
\begin{equation}\label{eq: PBW type iso}
    \Phi_{a,b}:\BH^{\op}\otimes \BH^{0-}_{\geq a}\otimes \BH^{0+}_{\geq b}\otimes \BH\xrightarrow{\ \simeq\ }D(\BH)_{a,b}. 
\end{equation}
Lastly, we can simply check that the inverse system of the left hand side of \eqref{eq: PBW type iso} over the directed system $\BZ\times\BZ$ is compatible with the algebra structure on the right hand side via the isomorphisms $\Phi_{a,b}$. This then defines an algebra structure on the inverse limit $D(\BH)$. 

Fix $(a,b)\in \BZ\times\BZ$. We define sets of generators by
\begin{align*}
    E&:=\{e^\alpha_i\,|\,i\geq 0,\ \alpha\in \CB\},\\
    {}^{+}H_b&:=\{{}^{+}h^\alpha_i\,|\,i\geq b,\ \alpha\in \CB\},\\
    {}^{-}H_a&:=\{{}^{-}h^\alpha_i\,|\,i\geq a,\ \alpha\in \CB\},\\
    F&:=\{f^\alpha_i\,|\,i\geq 0,\ \alpha\in \CB\}.
\end{align*}
We define a free associative $\BQ$-algebra generated by the above symbols, i.e., 
$$R_{a,b}:=\BQ[\,E\sqcup {}^{+}H_b\sqcup {}^{-}H_a\sqcup F\,]. 
$$
The multiplication of $R_{a,b}$ is denoted by $\star$. Following Remark \ref{rem: extend beyond the basis}, we define $e^\gamma_i\in R_{a,b}$ for arbitrary $\gamma\in H^*(C)$ as a linear combination of $e^\alpha_i$ with $\alpha\in \CB$; the same applies to other symbols as well.

We now introduce various types of commutator relations on the free algebra $R_{a,b}$, motivated by results from Section \ref{sec: Commutator relations}. We will repeatedly use the following notation: if $x, y$ are among $e_i, {}^{\pm}h_j, f_k$ and $\gamma\in H^*(C)$, then we write
$$x\star^{\gamma}y:=\sum x^{\gamma_{(1)}}\star y^{\gamma_{(2)}}\in R_{a,b} 
$$
where $\Delta_*\gamma=\sum \gamma_{(1)}\otimes \gamma_{(2)}$ using Sweedler's notation. We emphasize that all the commutators we use below are supercommutators.

We start with the commutators of $[e,e]$ type
\begin{equation*}
    EE^{\alpha,\beta}_{i,j}:=[e_i^\alpha,e^\beta_{j}]+\big(e_{i-1}\star^{\alpha\cup \beta}e_{j}+e_{i-2}\star^{\alpha\cup \beta}e_{j+1}\cdots + e_{j}\star^{\alpha\cup \beta}e_{i-1}\big),
\end{equation*}
where $i\geq j\geq 0$ and $\alpha,\beta\in \CB$. We consider similar relations for the opposite algebra, i.e., the commutators of $[f,f]$ type
\begin{equation*}
    FF^{\alpha,\beta}_{j,i}:=[f_j^\alpha,f^\beta_{i}]+\big(f_{i-1}\star^{\alpha\cup \beta}f_{j}+f_{i-2}\star^{\alpha\cup \beta}f_{j+1}\cdots + f_{j}\star^{\alpha\cup \beta}f_{i-1}\big),
\end{equation*}
where $i\geq j\geq 0$ and $\alpha,\beta\in \CB$. Since the Cartan part needs to be supercommutative, we have all combinations of commutators of $[{}^{\pm}h,{}^{\pm}h]$ type
\begin{align*}
    {}^{--}H{}H^{\alpha,\beta}_{i, i'}&:=[{}^{-}h^\alpha_{i},{}^{-}h^{\beta}_{i'}],\\
    {}^{++}HH^{\alpha,\beta}_{j, j'}&:=[{}^{+}h^\alpha_{j},{}^{+}h^{\beta}_{j'}],\\
    {}^{-+}HH^{\alpha,\beta}_{i,j}&:=[{}^{-}h^\alpha_{i},{}^{+}h^{\beta}_{j}],
\end{align*}
where $i\geq i'\geq a$, $j\geq j'\geq b$ and $\alpha,\beta\in \CB$. We now consider the $(a,b)$-truncated commutators of $[{}^{\pm}h,e]$ type
\begin{align*}
    {}^-HE^{\alpha,\beta}_{i,j}&:=[{}^{-}h_i^\alpha,e^\beta_j]+\big(e_{j}\star^{\alpha\cup\beta}{}^{-}h_{i-1}   +e_{j+1}\star^{\alpha\cup\beta}{}^{-}h_{i-2}+\cdots+e_{i+j-a-1}\star^{\alpha\cup\beta}{}^{-}h_{a}\big),\\
    {}^+HE^{\alpha,\beta}_{i',j}&:=[\+h_{i'}^\alpha,e^\beta_j]+\big(\+h_{{i'}-1}\star^{\alpha\cup\beta}e_{j}+\+h_{{i'}-2}\star^{\alpha\cup\beta}e_{j+1}+\cdots+{}^{+}h_{b}\star^{\alpha\cup\beta}e_{i+j-b-1}\big),
\end{align*}
where $i\geq a$, $i'\geq b$, $j\geq 0$ and $\alpha,\beta\in \CB$. Similarly, there are the $(a,b)$-truncated commutators of $[f,{}^{\pm}h]$ type
\begin{align*}
    {}^{-}FH_{j,i}^{\alpha,\beta}
    &:=[f^\alpha_j,{}^{-}h_i^\beta]
    +\big({}^{-}h_{i-1}\star^{\alpha\cup\beta}f_{j}+{}^{-}h_{i-2}\star^{\alpha\cup\beta}f_{j+1}+\cdots+{}^{-}h_{a}\star^{\alpha\cup\beta}f_{i+j-a-1}\big),\\
    {}^{+}FH_{j,i'}^{\alpha,\beta}
    &:=[f^\alpha_j,{}^{+}h_{i'}^\beta]
    +\big(f_{j}\star^{\alpha\cup\beta}{}^{+}h_{i'-1}   +f_{j+1}\star^{\alpha\cup\beta}{}^{+}h_{i'-2}+\cdots+f_{i'+j-b-1}\star^{\alpha\cup\beta}{}^{+}h_{b}\big),
\end{align*}
where $i\geq a$, $i'\geq b$, $j\geq 0$ and $\alpha,\beta\in \CB$. Finally, we have the $(a,b)$-truncated commutators of $[e,f]$ type
$$EF^{\alpha,\beta}_{i,j}:=[e^\alpha_i, f^\beta_j]+\sum_{\substack{n\geq a,\ m\geq b\\ n+m=i+j-1}} {}^{-}h_n\star^{\alpha\cup \beta}{}^{+}h_m,
$$
where $i, j\geq 0$ and $\alpha,\beta\in \CB$. We define $J_{a,b}\subseteq R_{a,b}$ to be the two-sided ideal generated by all the commutator relations described above. 

\begin{definition}
For each $(a,b)\in \BZ\times\BZ$, we define $D(\BH)_{a,b}:=R_{a,b}/J_{a,b}$ as a quotient algebra. 
\end{definition}

\begin{remark}\label{rem: each factor maps}
    Since $\BH$ is generated by $e_i^\alpha$'s with [e,e] type commutator relations by Theorem \ref{relations:torsioncoha}, there is an algebra homomorphism $\BH\rightarrow D(\BH)_{a,b}$ sending $e^\alpha_i$ to $e^\alpha_i$. Similarly, we have algebra homomorphisms $\BH^{\op},\ \BH^{0-}_{\geq a},\ \BH^{0+}_{\geq b}\rightarrow D(\BH)_{a,b}$. All these homomorphisms will be shown to be injective below. 
\end{remark}

\begin{theorem}[PBW type theorem] \label{thm: PBW}
    There is an isomorphism of vector spaces
    \begin{align*}
        \Phi_{a,b}:\BH^{\op}\otimes \BH^{0-}_{\geq a}\otimes \BH^{0+}_{\geq b}\otimes \BH&\rightarrow D(\BH)_{a,b}.\\
        v_1\otimes v_2\otimes v_3\otimes v_4&\mapsto v_1\star v_2\star v_3\star v_4
    \end{align*}
\end{theorem}

\begin{proof}
    By Remark \ref{rem: each factor maps}, each factor of $\BH^{\op}\otimes \BH^{0-}_{\geq a}\otimes \BH^{0+}_{\geq b}\otimes \BH$ maps to $D(\BH)_{a,b}$, hence the linear map $\Phi=\Phi_{a,b}$ is well-defined. In order to prove that $\Phi$ is an isomorphism of vector spaces, we will apply double induction using two increasing filtrations $F_\bullet$ and $G_\bullet$ on both sides of $\Phi$. Roughly speaking, $F_\bullet$ and $G_\bullet$ keep track of the sum of the subscript indices of $e, {}^{\pm}h, f$ and the number of symbols $e, {}^{\pm}h, f$ occurring in a monomial, respectively. Precisely, we first define such filtrations on the free algebra $R_{a,b}$ and consider the induced filtrations on the quotient $D(\BH)_{a,b}$ by taking image. Similarly, one can define $F_\bullet$ and $G_\bullet$ filtrations on $\BH^{\textnormal{op}}, \BH^{0-}_{\geq a}, \BH^{0+}_{\geq b}, \BH$ and then consider the tensor product filtration on $\BH^{\op}\otimes \BH^{0-}_{\geq a}\otimes \BH^{0+}_{\geq b}\otimes \BH$. Recall that given filtered vector spaces $F_\bullet V_i$ for $i=1, \dots, n$, the tensor product filtration is defined by 
    $$F_k (V_1\otimes \cdots \otimes V_n):=\sum_{k_1+\cdots k_n\leq k} (F_{k_1}V_1)\otimes \cdots\otimes (F_{k_n}V_n)\subseteq (V_1\otimes \cdots\otimes V_n).
    $$

    Note that $F_\bullet$ and $G_\bullet$ exhaust both sides of $\Phi$ and that $\Phi$ preserves both $F_\bullet$ and $G_\bullet$ filtrations. The outer induction will be in terms of the $G_\bullet$ filtration keeping track of the number of formal symbols multiplied. By definition, $G_{-1}=0$ for both sides of $\Phi$ and in particular $G_{-1}(\Phi)$ is an isomorphism. Assume that $G_{n-1}(\Phi)$ is an isomorphism for some $n\geq 0$. If we show that $G_{n}(\Phi)$ is an isomorphism, then this would imply that $\Phi$ is an isomorphism because $G_\bullet$ filtration exhausts both sides. For the inner induction, we consider $F_\bullet$ filtration on the restriction 
    $$G_{n}(\Phi):G_{n}(\BH^{\op}\otimes \BH^{0-}_{\geq a}\otimes \BH^{0+}_{\geq b}\otimes \BH)\rightarrow G_{n}(D(\BH)_{a,b}).
    $$
    If $k$ is sufficiently negative, then $F_{k}G_{n}=0$ for both sides of $G_{n}(\Phi)$ since we multiply at most $n$ formal symbols $e_{\geq 0}, {}^{-}h_{\geq a}, {}^{+}h_{\geq b}, f_{\geq 0}$ whose indices are bounded below. Assume that $F_{m-1}G_{n}(\Phi)$ is an isomorphism for some $m$. If we show that $F_{m}G_n(\Phi)$ is an isomorphism, then this would imply that $G_n(\Phi)$ is an isomorphism since $F_\bullet$ filtration exhausts $G_n$ of both sides. By the five lemma, it suffices to show that the associated graded map
    $$\gr^F_m G_n(\Phi):\gr^F_mG_{n}(\BH^{\op}\otimes \BH^{0-}_{\geq a}\otimes \BH^{0+}_{\geq b}\otimes \BH)\rightarrow \gr^F_mG_{n}(R_{a,b}/J_{a,b})
    $$
    is an isomorphism. Note that all the commutation relations we consider become the trivial commutation relations after taking the associated graded in terms of the $F_\bullet$ filtration. So we may assume that all the formal symbols in $e, {}^{\pm}h, f$ commute in $gr^F_m G_n$. This proves that $\gr^F_m G_n(\Phi)$ is an isomorphism, hence completing the proof.     
\end{proof}

If $(a,b)\geq (a',b')$, i.e., $a\leq a'$ and $b\leq b'$, then we have a diagram 
\begin{center}
    \begin{tikzcd}
\BH^{\op}\otimes \BH^{0-}_{\geq a}\otimes \BH^{0+}_{\geq b}\otimes \BH \arrow[d] \arrow[r, "{\Phi_{a,b}}", "\simeq"'] & D(\BH)_{a,b} \arrow[d, dotted] \\
\BH^{\op}\otimes \BH^{0-}_{\geq a'}\otimes \BH^{0+}_{\geq b'}\otimes \BH \arrow[r, "{\Phi_{a',b'}}", "\simeq"']        & D(\BH)_{a',b'}                 
\end{tikzcd}
\end{center}
where the left vertical map is the surjection killing the generators ${}^-h_n^\alpha$ with $a\leq n<a'$ and ${}^+h_m^\alpha$ with $b\leq m<b'$. This defines a unique right vertical map making the above diagram commutes. Such a map $D(\BH)_{a,b}\rightarrow D(\BH)_{a',b'}$ is an algebra homomorphism because all the $(a,b)$-truncated relations defining $D(\BH)_{a,b}$ are clearly sent to zero in $D(\BH)_{a',b'}$. Therefore, the double of the torsion CoHA
$$D(\BH)=\lim_{\longleftarrow} \Big(\BH^{\op}\otimes \BH^{0-}_{\geq a}\otimes \BH^{0+}_{\geq b}\otimes \BH\Big)\simeq \lim_{\longleftarrow} D(\BH)_{a,b}
$$
is given an algebra structure. It is straightforward to check that we have subalgebras 
$$\BH^{\op},\ \BH^{0-},\ \BH^{0+},\ \BH \subseteq D(\BH). $$

\subsection{Completed relations and $D(\BH)$-actions}

Previously, we have defined the double of the torsion CoHA as the inverse limit of the truncated algebras, i.e., 
$$D(\BH)=\lim_{\longleftarrow} R_{a,b}/J_{a,b}. 
$$
Since taking the inverse limit may obscure the resulting structure, one could reverse the procedure and first take the inverse limit $\widehat R:=\lim\limits_{\longleftarrow} R_{a,b}$ of free associative algebras and then quotient it by the two-sided ideal $\widehat J\subseteq \widehat R$ generated by the completed relations of type $[e,e]$, $[f,f]$, $[{}^\pm h,e]$, $[f, {}^\pm h]$ and $[e,f]$. For example, the completed relations of $[e,f]$ type are defined by 
$$\widehat EF^{\alpha,\beta}_{i,j}:=[e^\alpha_i, f^\beta_j]+\sum_{\substack{n, m\in \BZ\\n+m=i+j-1}} {}^{-}h_n\star^{\alpha\cup \beta}{}^{+}h_m\in \widehat R,
$$
where $i, j\geq 0$ and $\alpha,\beta\in \CB$.
The other cases are similar. Since completed relations are obtained as a limit of truncated relations, we have an inclusion $\widehat J\subseteq \lim\limits_{\longleftarrow} J_{a,b}$ of ideals in $\widehat R$. We can also prove the other inclusion:

\begin{proposition}
    We have $\displaystyle \widehat J= \lim_{\longleftarrow} J_{a,b}$. In particular, we have an algebra isomorphism 
    $$D(\BH)\simeq \widehat R/\widehat J.
    $$
\end{proposition}
\begin{proof}
There exists an inverse system of short exact sequences
$$0\rightarrow J_{a,b}\rightarrow R_{a,b}\rightarrow R_{a,b}/J_{a,b}\rightarrow 0.
$$
Note that a natural map $h_{a,b}:J_{a,b}\rightarrow J_{a+1,b}$ is surjective because all $(a+1,b)$-truncated relations are obtained by some $(a,b)$-truncated relations via ring homomorphism $R_{a,b}\rightarrow R_{a+1,b}$ which sets ${}^{-}h_a^\alpha=0$ for all $\alpha\in \CB$. Similarly, $J_{a,b}\rightarrow J_{a,b+1}$ is also surjective. Since $J_{a,b}$'s form a surjective inverse system with respect to $\BZ\times\BZ$, the limits again define a short exact sequence
$$0\rightarrow \lim_{\longleftarrow}J_{a,b}\rightarrow \lim_{\longleftarrow}R_{a,b}\rightarrow \lim_{\longleftarrow}(R_{a,b}/J_{a,b})\rightarrow 0.
$$
Therefore, assuming the equality $\displaystyle \widehat J= \lim_{\longleftarrow} J_{a,b}$, the second statement of the proposition follows easily:
$$D(\BH)\simeq\lim_{\longleftarrow} D(\BH)_{a,b}\simeq \lim_{\longleftarrow} (R_{a,b}/J_{a,b})\simeq (\lim_{\longleftarrow}R_{a,b})/(\lim_{\longleftarrow}J_{a,b})\simeq \widehat{R}/\widehat{J}.
$$

We are left to prove $\displaystyle\lim_{\longleftarrow} J_{a,b}\subseteq \widehat J$. Let $S_{a,b}$ be the index set of $(a,b)$-truncated relations and for each $s\in S_{a,b}$ write $r_s^{a,b}\in R_{a,b}$ for the corresponding relation. Similarly, let $\widehat S$ be the index set of completed relations and for each $s\in \widehat S$ write $\widehat r_s\in \widehat R$ for the corresponding relation. Recall that $(a,b)\geq (a',b')$ if and only if $a\leq a'$ and $b\leq b'$. We write $(a,b)>(a',b')$ if $(a,b)\geq (a',b')$ and $(a,b)\neq (a',b')$. Since there are fewer relations for bigger truncation, if $(a,b)>(a',b')$, then $S_{a,b}\supsetneq S_{a',b'}$ and $\widehat S = \cup S_{a,b}$. For each $(a,b)\in \BZ\times \BZ$, consider
$$0\rightarrow K_{a,b}\rightarrow \bigoplus_{s\in S_{a,b}}R_{a,b}\xrightarrow{\phi_{a,b}} J_{a,b}\rightarrow 0
$$
where $\phi_{a,b}\big((x_s)\big):=\sum_{s\in S_{a,b}} x_s\cdot r^{a,b}_s$ and $K_{a,b}:=\ker(\phi_{a,b})$. For $(a,b)>(a+1,b)$, we have horizontal comparison maps between short exact sequences
\begin{equation}\label{eqn: fa diagram}
    \begin{tikzcd}
0 \arrow[r] & K_{a,b} \arrow[r] \arrow[d, "f_{a,b}"]   & \bigoplus_{s\in S_{a,b}} R_{a,b} \arrow[r] \arrow[d, two heads, "g_{a,b}"]   & J_{a,b} \arrow[r] \arrow[d, two heads, "h_{a,b}"]     & 0 \\
0 \arrow[r] & K_{a+1,b} \arrow[r]& \bigoplus_{s\in S_{a+1,b}}R_{a+1,b} \arrow[r] & J_{a+1,b} \arrow[r] &0
\end{tikzcd}
\end{equation}
where $g_{a,b}$ is defined on each summand by
$$g_{a,b}\big((\dots, 0, x_s, 0, \dots)\big)=\begin{cases}
    (\dots, 0, \overline{x_s}, 0, \dots) &\textnormal{if}\quad s\in S_{a+1,b},\\
    0&\textnormal{if}\quad s\in S_{a,b}\backslash S_{a+1,b},
\end{cases}
$$
and $h_{a,b}$ is induced from $R_{a,b}\twoheadrightarrow R_{a+1,b}$. It is straightforward to check that the right square commutes by definition. For $(a,b)>(a,b+1)$, similar vertical comparison maps between short exact sequences exist. Assume for the moment that $\{K_{a,b}\}_{(a,b)\in \BZ\times\BZ}$ forms a surjective inverse system, i.e., $f_{a,b}$ is surjective for all $(a,b)$ and the same holds for the vertical analogues. Then by taking limits, we obtain a short exact sequence 
$$0\rightarrow \lim_{\longleftarrow} K_{a,b}\rightarrow \lim_{\longleftarrow} \bigoplus_{s\in S_{a,b}}R_{a,b}\xrightarrow{\lim\phi_{a,b}} \lim_{\longleftarrow} J_{a,b}\rightarrow 0.
$$
By definition of comparison maps $g$'s, we can rewrite this as
$$0\rightarrow \lim_{\longleftarrow} K_{a,b}\rightarrow \bigoplus_{s\in \widehat S}\widehat R\xrightarrow{\lim\phi_{a,b}} \lim_{\longleftarrow} J_{a,b}\rightarrow 0
$$
where $\lim\phi_{a,b}$ is given by the completed relations. In other words, $\displaystyle \lim_{\longleftarrow} J_{a,b}$ is generated by completed relations, i.e., $\displaystyle\lim_{\longleftarrow} J_{a,b}= \widehat J$.

We now prove that the horizontal comparison map $f_{a,b}$ is surjective. The vertical case can be proven identically. Let $J_{a,b}'$ be the image of the composition
    $$ \bigoplus_{s\in S_{a+1,b}}R_{a,b}\hookrightarrow \bigoplus_{s\in S_{a,b}}R_{a,b}\xrightarrow{\phi_{a,b}}J_{a,b}
    $$
    and let $K_{a,b}'$ be the kernel of the surjection $\bigoplus_{s\in S_{a+1,b}} R_{a,b}\twoheadrightarrow J_{a,b}'$. Then surjectivity of $f_{a,b}$ would follow from surjectivity of $f_{a,b}'$ in the diagram below
    \begin{center}
        \begin{tikzcd}
    
0 \arrow[r] & K_{a,b}' \arrow[r] \arrow[d, "f_{a,b}'"]   & \bigoplus_{s\in S_{a+1,b}} R_{a,b} \arrow[r] \arrow[d, two heads, "g_{a,b}'"]   & J_{a,b}'\arrow[r] \arrow[d, two heads, "h_{a,b}'"]     & 0 \\
0 \arrow[r] & K_{a+1,b} \arrow[r]  & \bigoplus_{s\in S_{a+1,b}}R_{a+1,b} \arrow[r] & J_{a+1,b} \arrow[r]  & 0
\end{tikzcd}
    \end{center}
    Since $R_{a,b}$ is obtained by adding new formal symbols \{${}^- h^\alpha_a\,|\,\alpha\in \CB\}$ to $R_{a+1,b}$, both surjections $g'_{a,b}$ and $h'_{a,b}$ admit compatible sections. This clearly implies that $f'_{a,b}$ is also surjective by diagram chasing. This completes the proof. 
\end{proof}

\begin{proposition}\label{prop: doubled algebra action}
    There exists an action $D(\BH)\curvearrowright\BV^\vir$ such that it extends the action of subalgebras $\BH$, $\BH^\op$ and $\BH^{0\pm}$ via creation and annihilation operators and multiplication by tautological classes, respectively.
\end{proposition}
\begin{proof}
The inverse limit $\widehat R=\lim\limits_{\longleftarrow} R_{a,b}$ clearly acts on $\BV^\vir$ via creation and annihilation and multiplication by tautological classes. These actions satisfy the completed relations by Propositions \ref{prop: [e,m]}, \ref{prop: [f,m]}, \ref{thm: [e,f]}, \ref{thm: [e,e]} and \ref{thm: [f,f]}. Therefore, the action factors through the quotient $\widehat R/\widehat J\simeq D(\BH)$. 

\end{proof}

\begin{remark}\label{rem: relation to MN}
We emphasize that the definition of $D(\BH)$ is independent of the choices of $V$ and $r$ used to fix the type of the Quot schemes $\Quot_{(r,*)}(V)$. Once we fix $V$ and $r$, the action
$$D(\BH)\curvearrowright \BV^\vir_{(r,*)}:=\bigoplus_{d\in \BZ}\BV^\vir_{(r,d)}
$$
from the above proposition factors through a smaller algebra
$$D(\BH)\twoheadrightarrow D(\BH)_{V,r}
$$
defined by quotienting out the relations coming from \eqref{eq: vanishing for small enough i} and \eqref{eqn: h+ h- relation}. If $V$ is a vector bundle and $r=0$, then the action of $D(\BH)_{V,0}\curvearrowright\BV$ is essentially the same as the action of a shifted Yangian constructed by Marian--Negu\c t \cite{marian2026cohomologyquotschemesmooth}. 
    
\end{remark}

\appendix

\section{Bivariant intersection theory}

In this section, we collect basic definitions and results we need from bivariant intersection theory. Such a theory was originally developed by Fulton--MacPherson \cite{FM} for classical schemes and then extended to derived Artin stacks by Khan \cite{Khan}. In this section, all spaces are derived Artin stacks over $\pt=\textnormal{Spec}(\BC)$ and all morphisms are locally of finite type. 

For any morphism $f:X\rightarrow Y$, the relative Borel--Moore homology group is defined as\footnote{The notation $X/Y$ omits the structural morphism $f:X\rightarrow Y$ when it is understood.}
$$H^k(X/Y):=H^k(X,f^!\underline{\BQ}_Y),\quad k\in \BZ.
$$
The relative Borel--Moore homology group interpolates the cohomology ring and Borel--Moore homology group, i.e., 
$$H^k(X)=H^k(X/X),\quad H^{\BM}_k(X)=H^{-k}(X/\pt),
$$
where $X/X$ and $X/\pt$ are with respect to the identity morphism and the structural morphism to $\pt$, respectively. We remark that the relative Borel--Moore homology groups depend only on the classical truncations, i.e., 
$$H^*(X/Y)=H^*(X_{\textnormal{cl}}/Y_{\textnormal{cl}})
$$
even though functorial properties we explain below will depend on the derived structures.

The relative Borel--Moore homology groups are endowed with the following structures.
\begin{enumerate}
    \item [(i)] Product: For any $X\rightarrow Y\rightarrow Z$, there is a natural map
    $$H^{a}(X/Y)\otimes H^{b}(Y/Z)\rightarrow H^{a+b}(X/Z),\quad \alpha\otimes\beta\mapsto \alpha\cdot \beta.$$
    \item [(ii)] Proper pushforward: For any $X\xrightarrow{f}Y\rightarrow Z$ where $f$ is proper, there is a natural map 
    $$f_*:H^{a}(X/Z)\rightarrow H^a(Y/Z).$$
    \item [(iii)] Pullback: For any Cartesian diagram
    \begin{center}
        \begin{tikzcd}
X' \arrow[d] \arrow[r] & Y' \arrow[d,"g"] \\
X \arrow[r]            & Y      
\end{tikzcd}
    \end{center}
    there is a natural map
    $$g^*:H^{a}(X/Y)\rightarrow H^{a}(X'/Y').
    $$
\end{enumerate}
These structures satisfy various properties. The first four formulas below hold whenever the expression makes sense according to the setting of the three structures. We give the precise settingsa for the rest of the formulas.
\begin{enumerate}
    \item [(A1)] Associativity of products: $\alpha\cdot(\beta\cdot\gamma)=(\alpha\cdot\beta)\cdot\gamma$
    \item [(A2)] Functoriality of proper pushforwards: $(fg)_*\alpha=f_*(g_*\alpha)$
    \item [(A3)] Functoriality of pullbacks: $(fg)^*\alpha=g^*(f^*\alpha)$
    \item [(A4)] Product and proper pushforward commute: $(f_*\alpha)\cdot \beta=f_*(\alpha\cdot \beta)$
    \item [(A5)] Product and pullback commute: For any Cartesian diagram
    \begin{center}
        \begin{tikzcd}
X' \arrow[d] \arrow[r] & Y' \arrow[d,"g'"] \arrow[r]&Z'\arrow[d,"g"]\\
X \arrow[r]            & Y \arrow[r]&Z     
\end{tikzcd}
    \end{center}
    and $\alpha\in H^*(X/Y)$, $\beta\in H^*(Y/Z)$, we have $g^*(\alpha\cdot\beta)=((g')^*\alpha)\cdot (g^*\beta)$. 
    \item [(A6)] Proper pushforward and pullback commute: For any Cartesian diagram
    \begin{center}
        \begin{tikzcd}
X' \arrow[d] \arrow[r,"f'"] & Y' \arrow[d,"g'"] \arrow[r]&Z'\arrow[d,"g"]\\
X \arrow[r,"f"]            & Y \arrow[r]&Z     
\end{tikzcd}
    \end{center}
    with $f$ being proper and $\alpha\in H^*(X/Y)$, we have $g^*f_*\alpha=(f')_*(g')^*\alpha$.
    \item [(A7)] Projection formula: For any Cartesian diagram
    \begin{center}
        \begin{tikzcd}
X' \arrow[d] \arrow[r,"f'"] & Y' \arrow[d] &\\
X \arrow[r,"f"]            & Y \arrow[r]&Z     
\end{tikzcd}
    \end{center}
    with $f$ being proper and $\alpha\in H^*(Y'/Y)$, $\beta\in H^*(X/Z)$, we have $(f')_*((f^*\alpha)\cdot \beta)=\alpha\cdot (f_*\beta).$
    \item [(A8)] Graded commutativity: For any Cartesian diagram
    \begin{center}
        \begin{tikzcd}
X' \arrow[d, "g'"] \arrow[r,"f'"] & Y' \arrow[d, "g"] \\
X \arrow[r,"f"]            & Y 
\end{tikzcd}
    \end{center}
    and $\alpha\in H^*(X/Y)$, $\beta\in H^*(Y'/Y)$, we have $(g^*\alpha)\cdot \beta = (-1)^{|\alpha||\beta|} (f^*\beta)\cdot \alpha$. 
\end{enumerate}

The main theorem of Khan \cite{Khan} constructs the relative virtual class
\[ 1_{X/Y}\in H^{-2d}(X/Y) \]
associated to any quasi-smooth morphism $f:X\rightarrow Y$ of relative dimension $d$. The relative virtual classes satisfy the following properties:
\begin{enumerate}
    \item [$\bullet$] Functoriality of relative virtual classes: For any quasi-smooth morphisms $X\rightarrow Y\rightarrow Z$, we have $1_{X/Y}\cdot 1_{Y/Z}=1_{X/Z}$. 
    \item [$\bullet$] Base change of relative virtual classes: For any Cartesian diagram
    \begin{center}
        \begin{tikzcd}
X' \arrow[d] \arrow[r] & Y' \arrow[d,"g"] \\
X \arrow[r, "f"]            & Y      
\end{tikzcd}
    \end{center}
    with $f$ being quasi-smooth, we have $g^*(1_{X/Y})=1_{X'/Y'}$. 
\end{enumerate}
The relative virtual class allows us to define the following additional structures.
\begin{enumerate}
    \item [(iv)] Virtual pullback: For any $X\xrightarrow{f}Y\rightarrow Z$ where $f$ is quasi-smooth, there is a natural map\footnote{If $Z=\pt$, then this construction becomes the usual virtual pullback between Borel--Moore homology groups. If furthermore $Y=Z=\pt$, then the image of $1$ under $\BQ=H^0(Y/Z)\rightarrow H^{\BM}_{2d}(X)$ is the virtual fundamental class of $X$.}
    $$f^!:H^{a}(Y/Z)\rightarrow H^{a-2d}(X/Z),\quad \alpha\mapsto 1_{X/Y}\cdot \alpha.
    $$
    
    \item [(v)] Virtual Umkehr map: For any quasi-smooth and proper morphism $f:X\rightarrow Y$, there is a natural map
    $$f_!:H^{a}(X)\rightarrow H^{a-2d}(Y),\quad \alpha\mapsto f_*(\alpha\cdot 1_{X/Y}).
    $$
\end{enumerate}
By functoriality and base change properties of the relative virtual classes, these new structures satisfy the following properties whenever the expression makes sense. For example, whenever we write $f^!$ (resp. $f_!$), the morhpism $f$ is understood to be quasi-smooth (resp. quasi-smooth and proper).

\begin{enumerate}
    \item [(A9)] Functoriality of virtual pullbacks: $(fg)^!\alpha=g^!(f^!\alpha)$ 
    \item [(A10)] Functoriality of virtual Umkehr maps: $(fg)_!\alpha=f_!(g_!\alpha)$ 
    \item [(A11)] Base change of the virtual pullbacks: For any Cartesian diagram
    \begin{center}
        \begin{tikzcd}
X' \arrow[d,"g'"] \arrow[r,"f'"] & Y' \arrow[d,"g"] \\
X \arrow[r, "f"]            & Y      
\end{tikzcd}
    \end{center}
    where $f$ is quasi-smooth, $g$ is proper and $\alpha\in H^{\BM}_*(Y')$, we have $f^!g_*\alpha=(g')_*(f')^!\alpha$.
    \item [(A12)] Base change of the virtual Umkehr maps: For any Cartesian diagram
    \begin{center}
        \begin{tikzcd}
X' \arrow[d,"g'"] \arrow[r,"f'"] & Y' \arrow[d,"g"] \\
X \arrow[r, "f"]            & Y      
\end{tikzcd}
    \end{center}
    with $f$ being quasi-smooth and proper and $\alpha\in H^*(X)$, we have $g^*f_!(\alpha)=(f')_!(g')^*\alpha$.
    \item [(A13)] Virtual projection formulas: For any quasi-smooth proper morphism $f:X\rightarrow Y$, we have 
    $$\begin{cases}
        f_*(\alpha\cdot (f^!\beta)) = (f_!\alpha)\cdot \beta,& \textnormal{if } \alpha\in H^*(X), \beta\in H^{\BM}_*(Y), \\
            f_!((f^*\alpha)\cdot\beta)=\alpha\cdot(f_!\beta),& \textnormal{if } \alpha\in H^*(Y), \beta\in H^*(X).
    \end{cases}$$
    \item [(A14)] Virtual pullback and cap product: For any quasi-smooth morphism $X\rightarrow Y$, we have
    $$f^!(\alpha\cdot \beta)=(f^*\alpha)\cdot (f^!\beta),\quad \alpha\in H^*(Y),\ \beta\in H^\BM_*(Y). 
    $$
\end{enumerate}

\bibliographystyle{myamsplain} 
\bibliography{refs}

\end{document}